\documentclass{article}

\usepackage{amsfonts}
\usepackage{amsmath}
\usepackage{mathtools}
\usepackage{multirow}
\usepackage{epstopdf}
\usepackage{algorithm}
\usepackage{algorithmic}
\usepackage{bm}
\usepackage{microtype}
\usepackage{amssymb}
\usepackage{enumerate}

\newtheorem{theorem}{Theorem}

\newtheorem{remark}{Remark}
\newtheorem{lemma}{Lemma}
\newtheorem{proof}{Proof}
\newtheorem{corollary}{Corollary}

\newtheorem{assumption}{Assumption}

\newcommand{\x}{\mathbf{x}}
\newcommand{\y}{\mathbf{y}}
\newcommand{\z}{\mathbf{z}}
\newcommand{\w}{\mathbf{w}}

\newcommand{\s}{\mathbf{s}}
\newcommand{\1}{\mathbf{1}}

\newcommand{\A}{A}

\newcommand{\W}{W}
\renewcommand{\S}{\mathbb{S}}
\newcommand{\E}{\mathbb{E}}
\newcommand{\I}{I}
\newcommand{\U}{U}
\newcommand{\V}{V}
\newcommand{\R}{\mathbb{R}}
\newcommand{\bO}{{\cal O}}
\newcommand{\N}{\mathcal{N}}
\newcommand{\D}{\mathcal{D}}
\newcommand{\blambda}{\mathbf{\lambda}}

\newcommand{\<}{\left\langle}
\renewcommand{\>}{\right\rangle}
\newcommand{\tabincell}[2]{\begin{tabular}{@{}#1@{}}#2\end{tabular}}

\usepackage[accepted]{icml2020}

\icmltitlerunning{Variance Reduced EXTRA and DIGing and Their Optimal Acceleration}

\begin{document}

\onecolumn
\icmltitle{Variance Reduced EXTRA and DIGing and Their Optimal Acceleration for Strongly Convex Decentralized Optimization}

\begin{icmlauthorlist}
\icmlauthor{Huan Li}{to}
\icmlauthor{Zhouchen Lin}{goo}
\icmlauthor{Yongchun Fang}{to}
\end{icmlauthorlist}

\icmlaffiliation{to}{Institute of Robotics and Automatic Information Systems, College of Artificial Intelligence, Nankai University, Tianjin, China.\\}
\icmlaffiliation{goo}{Key Laboratory of Machine Perception, School of Artificial Intelligence, Peking University, Beijing, China. Institute for Artificial Intelligence, Peking University, Beijing, China. Peng Cheng Laboratory, Shenzhen, China.\\}
\icmlcorrespondingauthor{Huan Li and Zhouchen Lin}{lihuanss@nankai.edu.cn, zlin@pku.edu.cn}





\vskip 0.3in



\printAffiliationsAndNotice{}  

\begin{abstract}
  We study stochastic decentralized optimization for the problem of training machine learning models with large-scale distributed data. We extend the widely used EXTRA and DIGing methods with variance reduction (VR), and propose two methods: VR-EXTRA and VR-DIGing. The proposed VR-EXTRA requires the time of $\bO((\kappa_s+n)\log\frac{1}{\epsilon})$ stochastic gradient evaluations and $\bO((\kappa_b+\kappa_c)\log\frac{1}{\epsilon})$ communication rounds to reach precision $\epsilon$, which are the best complexities among the non-accelerated gradient-type methods, where $\kappa_s$ and $\kappa_b$ are the stochastic condition number and batch condition number for strongly convex and smooth problems, respectively, $\kappa_c$ is the condition number of the communication network, and $n$ is the sample size on each distributed node. The proposed VR-DIGing has a little higher communication cost of $\bO((\kappa_b+\kappa_c^2)\log\frac{1}{\epsilon})$. Our stochastic gradient computation complexities are the same as the ones of single-machine VR methods, such as SAG, SAGA, and SVRG, and our communication complexities keep the same as those of EXTRA and DIGing, respectively. To further speed up the convergence, we also propose the accelerated VR-EXTRA and VR-DIGing with both the optimal $\bO((\sqrt{n\kappa_s}+n)\log\frac{1}{\epsilon})$ stochastic gradient computation complexity and $\bO(\sqrt{\kappa_b\kappa_c}\log\frac{1}{\epsilon})$ communication complexity. Our stochastic gradient computation complexity is also the same as the ones of single-machine accelerated VR methods, such as Katyusha, and our communication complexity keeps the same as those of accelerated full batch decentralized methods, such as MSDA. To the best of our knowledge, our accelerated methods are the first to achieve both the optimal stochastic gradient computation complexity and communication complexity in the class of gradient-type methods.
\end{abstract}

\section{Introduction}
Emerging machine learning applications involve huge amounts of data samples, and the data are often distributed across multiple machines for storage and computational reasons. In this paper, we consider the following distributed convex optimization problem with $m$ nodes, and each node has $n$ local training samples:
\begin{eqnarray}\label{problem}
\min_{x\in\R^p} \sum_{i=1}^m f_{(i)}(x),\quad\mbox{where}\quad f_{(i)}(x)=\frac{1}{n}\sum_{j=1}^n f_{(i),j}(x),
\end{eqnarray}
where the local component function $f_{(i),j}$ represents the $j$th sample of node $i$, and it is not accessible by any other node in the communication network. The network is abstracted as a connected and undirected graph $\mathcal{G}=(\mathcal{V},\mathcal{E})$, where $\mathcal{V}=\{1,2,...,m\}$ is the set of nodes, and $\mathcal{E}\subset\mathcal{V}\times\mathcal{V}$ is the set of edges. Nodes $i$ and $j$ can send information to each other if and only if $(i,j)\in\mathcal{E}$. The goal of the networked nodes is to cooperatively solve problem (\ref{problem}) via local computation and communication, that is, each node $i$ makes its decision only based on the local computations on $f_{(i)}$, for example, the gradient, and the local information received from its neighbors in the network.

When the local data size $n$ is large, the cost of computing the full batch gradient $\nabla f_{(i)}$ at each iteration is expensive. To address the issue of large-scale distributed data, stochastic decentralized algorithms are often used to solve problem (\ref{problem}), where each node only randomly samples one component gradient at each iteration (extendable to the mini-batch settings with more than one randomly selected component). Most decentralized algorithms alternate between computations and communications. Thus to compare the performance of such methods, two measures are used: the number of communication rounds and the number of stochastic gradient evaluations, where one communication round allows each node to send information to their neighbors, for example, $\bO(1)$ vectors of size $p$, and one stochastic gradient evaluation refers to computing the randomly sampled $\nabla f_{(i),j}$ for all $i\in \mathcal{V}$ in parallel \citep{richtaric-2020-dist}.

Although stochastic decentralized optimization has been a hot topic in recent years, and several algorithms have been proposed, to the best of our knowledge, in the class of algorithms not relying on the expensive dual gradient evaluations, there is no algorithm optimal in both the number of communication rounds and the number of stochastic gradient evaluations \citep{richtaric-2020-dist}, where ``optimal" means matching the corresponding lower bounds. In this paper, we extend two widely used decentralized algorithms of EXTRA \citep{shi2015extra} and DIGing \citep{shi2017,qu2017}, which have sparked a lot of interest in the distributed optimization community, to stochastic decentralized optimization by combining them with the powerful variance reduction technique. Furthermore, we propose two accelerated stochastic decentralized algorithms, which are optimal in the above two measures of communications and stochastic gradient computations.

\subsection{Notations and Assumptions}\label{sec-assumption}
Denote $x_{(i)}\in\R^p$ to be the local variable for node $i$. To simplify the algorithm description in a compact form, we introduce the aggregate objective function $f(\x)$ with its aggregate variable $\x$ and aggregate gradient $\nabla f(\x)$ as
\begin{eqnarray}
\qquad\begin{aligned}\label{aggregate}
\x=\left(
  \begin{array}{c}
    x_{(1)}^T\\
    \vdots\\
    x_{(m)}^T
  \end{array}
\right),\quad f(\x)=\sum_{i=1}^m f_{(i)}(x_{(i)}),\quad\nabla f(\x)=\left(
  \begin{array}{c}
    \nabla f_{(1)}(x_{(1)})^T\\
    \vdots\\
    \nabla f_{(m)}(x_{(m)})^T
  \end{array}
\right).
\end{aligned}
\end{eqnarray}
Denote $x^*$ to be the optimal solution of problem (\ref{problem}), and let $\x^*=\1 (x^*)^T$, where $\1$ is the column vector of $m$ ones. Denote $\I$ as the identity matrix, and $\N_{(i)}$ as the neighborhood of node $i$. Denote $\mbox{Ker}(\U)=\{x\in\R^m|\U x=0\}$ as the kernel space of matrix $\U\in\R^{m\times m}$, and $\mbox{Span}(\U)=\{y\in\R^m|y=\U x,\forall x\in\R^m\}$ as the linear span of all the columns of $\U$. For matrices, we denote $\|\cdot\|$ as the Frobenius norm for simplicity without ambiguity, since it is the only matrix norm we use in this paper. The notation $A\succeq B$ means $A-B$ is positive semidefinite.

We make the following assumptions for the functions in (\ref{problem}).
\begin{assumption}\label{assumption_f}
Each $f_{(i)}(x)$ is $L_{(i)}$-smooth and $\mu$-strongly convex. Each $f_{(i),j}(x)$ is $L_{(i),j}$-smooth and convex.
\end{assumption}
We say a function $g(x)$ is $L$-smooth if its gradient satisfies $\|\nabla g(\y)-\nabla g(\x)\|\leq L\|\y-\x\|$. Motivated by \citep{ADFS,DVR}, we define several notations as follows:
\begin{equation}
L_f=\max_i L_{(i)},\quad\overline L_{(i)}=\frac{1}{n}\sum_{j=1}^nL_{(i),j},\quad \overline L_f=\max_i \overline L_{(i)},\quad\kappa_s=\frac{\overline L_f}{\mu},\quad\kappa_b=\frac{L_f}{\mu}.\label{def_f}
\end{equation}
Then  $f(\x)$ is also $\mu$-strongly convex and $L_f$-smooth. It always holds that $L_{(i)}\leq\overline L_{(i)}\leq nL_{(i)}$\footnote{See footnote 14 in \citep{kat} for the analysis.}, which further gives
\begin{equation}
L_f\leq\overline L_f\leq nL_f\quad\mbox{and}\quad \kappa_b\leq\kappa_s\leq n\kappa_b. \label{L-relation}
\end{equation}
We follow \citep{ADFS} to call $\kappa_b$ the batch condition number, and $\kappa_s$ the stochastic condition number, which are classical quantities in the analysis of batch optimization methods and finite-sum optimization methods, respectively. Generally, we have $\kappa_s\ll n\kappa_b$, see \citep{kat} for the example and analysis.

In decentralized optimization, communication is often represented as a matrix multiplication with a weight matrix $\W\in\R^{m\times m}$. We make the following assumptions for this weight matrix associated to the network\footnote{The weights can be assigned heuristically or optimized given the fixed graph structure \citep{boyd2004}.}.
\begin{assumption}\label{assumption_w}
$\vspace*{-0.4cm}\linebreak$
\begin{enumerate}
\item $\W_{i,j}\neq 0$ if and only if agents $i$ and $j$ are neighbors or $i=j$. Otherwise, $\W_{i,j}=0$.
\item $\W=\W^T$, $\I\succeq\W\succeq \omega\I$, and $\W\1=\1$.
\end{enumerate}
\end{assumption}

We let $\omega=0$ for EXTRA, and $\omega=\frac{\sqrt{2}}{2}$ for DIGing. We can relax $\omega$ to be any small positive constant for DIGing\footnote{In this case, condition (\ref{UVbound}) is relaxed to $\|\V\x\|^2\leq (1-\omega^2)\|\x\|^2$. By the similar proofs of Theorem \ref{non-acc-theorem}, we can obtain the $\bO(\frac{1}{\omega^2}(\kappa_b+\kappa_c^2)\log\frac{1}{\epsilon})$ complexity for DIGing.}, and fix it to $\frac{\sqrt{2}}{2}$ to simplify the analysis. For EXTRA, we can also relax the condition to $\I\succeq\W\succeq (-1+\delta)\I$ for any small positive constant $\delta$\footnote{In this case, the complexity of EXTRA becomes $\bO(\frac{1}{\delta}(\kappa_b+\kappa_c)\log\frac{1}{\epsilon})$.}.
Part 2 of Assumption \ref{assumption_w} implies that the eigenvalues of $\W$ lie in $[\omega,1]$, and its largest one $\sigma_1(\W)$ equals 1. Moreover, if the network is connected, we have $\sigma_2(\W)<1$, where $\sigma_2(\W)$ means the second largest eigenvalue. We often use
\begin{equation}\label{define-kappa-c}
\kappa_c=\frac{1}{1-\sigma_2(\W)}
\end{equation}
as the condition number of the communication network, which upper bounds the ratio between the largest eigenvalue and the smallest non-zero eigenvalue of $(\I-\W)$, which is a gossip matrix \citep{dasent}.

As will be introduced in the next section, we often use $\kappa_b$ and $\kappa_c$ to describe the number of communication rounds, and $\kappa_s$ for the number of stochastic gradient evaluations in stochastic decentralized optimization.
\subsection{Literature Review}
In this section, we give a brief review for the decentralized and stochastic methods, as well as their combination. Table \ref{table-comp} sums up the complexities of the representative ones.

\subsubsection{Full Batch Decentralized Algorithms}
Distributed optimization has gained significant attention for a long time \citep{Bertsekas1983,Tsitsiklis1986}. The modern distributed gradient descent (DGD) was proposed in \citep{Nedic-2009} for the general network topology, and was further extended in \citep{nedic2011asynchronous,Ram-2010,Yuan-2016}. These algorithms are usually slow due to the diminishing step-size, and suffer from the sublinear convergence even for strongly convex and smooth objectives. To avoid the diminishing step-size and speed up the convergence, several methods relying
on tracking the differences of gradients have been proposed. Typical examples include EXTRA \citep{shi2015extra}, DIGing \citep{shi2017,qu2017}, NIDS \citep{NIDS}, and other similar algorithms \citep{aug-dgm,ranxin2018}. Especially, EXTRA \citep{li-2019-extra} and NIDS \citep{NIDS} have the $\bO((\kappa_b+\kappa_c)\log\frac{1}{\epsilon})$ complexity both in communications and full batch gradient evaluations to solve problem (\ref{problem}) to reach precision $\epsilon$, which is the best among the non-accelerated algorithms. DIGing has a slight higher complexity of $\bO((\kappa_b+\kappa_c^2)\log\frac{1}{\epsilon})$ \citep{sulaiman2020}. Another typical class of distributed algorithms is based on the Lagrangian function, and they work with the Fenchel dual. Examples include the dual ascent \citep{Terelius-2011,dasent,Uribe-2017}, ADMM \citep{Iutzeler-2016,makhdoumi-2017,Aybat2018}, and the primal-dual method \citep{Lam-2017,scaman-2018,hong-2017,jakovetic-2017}. However, the dual-based methods often need to compute the gradient of the Fenchel conjugate of the local functions, called dual gradient in the sequel, which is expensive.

Nesterov's acceleration technique is an efficient approach to speed up the convergence of first-order methods, and it has also been successfully applied to decentralized optimization. Typical examples include the distributed Nesterov gradient with consensus \citep{Jakovetic-2014}, the distributed Nesterov gradient descent \citep{qu2017-2}, the multi-step dual accelerated method (MSDA) \citep{dasent,scaman-2019-jmlr}, accelerated penalty method \citep{li-2018-pm}, accelerated EXTRA \citep{li-2019-extra}, and the accelerated proximal alternating predictor-corrector method (APAPC) \citep{richtaric-2020-dist}. Some of these methods have suboptimal computation complexity, and Chebyshev acceleration (CA) \citep{Arioli-2014} is a powerful technique to further reduce the computation cost. \citet{dasent,scaman-2019-jmlr} proved the $\varOmega(\sqrt{\kappa_b\kappa_c}\log\frac{1}{\epsilon})$ lower bound on the number of communication rounds and the $\varOmega(\sqrt{\kappa_b}\log\frac{1}{\epsilon})$ lower bound on the number of full batch gradient evaluations, which means that any first-order full batch decentralized methods cannot be faster than these bounds. The MSDA and APAPC methods with CA achieve these lower bounds.

\subsubsection{Stochastic Algorithms on a Single Machine}\label{sec:stointro}

Stochastic gradient descent (SGD) has been the workhorse in machine learning. However, since the variance of the noisy gradient will not go to zero, SGD often suffers from the slow sublinear convergence. Variance reduction (VR) was designed to reduce the negative effect of the noise, which can improve the stochastic gradient computation complexity to $\bO((\kappa_s+n)\log\frac{1}{\epsilon})$. On the other hand, full batch methods, such as gradient descent, require $\bO(\kappa_b\log\frac{1}{\epsilon})$ iterations, and thus $\bO(n\kappa_b\log\frac{1}{\epsilon})$ individual gradient evaluations for finite-sum problems with $n$ samples, which may be much larger than $\bO((n+\kappa_s)\log\frac{1}{\epsilon})$ when $\kappa_s\ll n\kappa_b$. Representative examples of VR methods include SAG \citep{SAG}, SAGA \citep{SAGA}, and SVRG \citep{SVRG,xiao-2014-svrg-siam}. We can further accelerate the VR methods to the $\bO((\sqrt{n\kappa_s}+n)\log\frac{1}{\epsilon})$ stochastic gradient computation complexity by Nesterov’s acceleration technique. Examples include Katyusha \citep{kat} and its extensions in \citep{zhou2018simple,Kovalev-2020-loopless}. Other accelerated stochastic algorithms can be found in \citep{lan-2017-MP,catalyst,fercoq-2015-siam,xiao-2015-siam}. \citet{lan-2017-MP} proved the $\varOmega((\sqrt{n\kappa_s}+n)\log\frac{1}{\epsilon})$ lower bound for strongly convex and smooth stochastic optimization, and Katyusha achieves this lower bound.

\subsubsection{Stochastic Decentralized Algorithms}

To address the issue of large-scale distributed data, \citet{chen2012} and \citet{Ram-2010} extended the DGD method to the distributed stochastic gradient descent (DSGD). To further improve the convergence of stochastic decentralized algorithms, \citet{pushiMP} combined DSGD with gradient tracking, \citet{Mokhtari-2016} combined EXTRA with SAGA, and proposed the decentralized double stochastic averaging gradient algorithm, \citet{ranxin2019} combined gradient tracking with the VR technique, and two algorithms are proposed, namely, GT-SAGA and GT-SVRG. \citet{BoyueLi2020} generalized the approximate Newton-type method called DANE with gradient tracking and variance reduction. See \citep{ranxin2020} for a detailed review for the non-accelerated stochastic decentralized algorithms. \citet{ADFS} proposed an accelerated decentralized stochastic algorithm called ADFS for problems with finite-sum structures, which achieves the optimal $\bO(\sqrt{\kappa_b\kappa_c}\log\frac{1}{\epsilon})$ communication complexity. However, ADFS is a dual-based method, and it needs to compute the dual gradient at each iteration, which is expensive. Recently, \citet{DVR} further proposed a dual-free decentralized method with variance reduction, called DVR, which achieves the $\bO(\kappa_b\sqrt{\kappa_c}\log\frac{1}{\epsilon})$ communication complexity and the $\bO((\kappa_s+n)\log\frac{1}{\epsilon})$ stochastic gradient computation complexity. These complexities can be further improved to $\widetilde \bO(\sqrt{\kappa_b\kappa_c}\sqrt{\frac{n\kappa_b}{\kappa_s}}\log\frac{1}{\epsilon})$ and $\widetilde \bO((\sqrt{n\kappa_s}+n)\log\frac{1}{\epsilon})$ by the Catalyst acceleration \citep{catalyst}, respectively, where $\widetilde \bO$ hides the poly-logarithmic factor, which is at least $\bO(\log\kappa_b)$\footnote{See Proposition 17 in \citep{catalyst} and Corollary 7 in \citep{li-2019-extra}.}. We see that DVR-Catalyst achieves the optimal stochastic gradient computation complexity up to log factor. However, its communication cost is increased by a factor $\bO(\sqrt{\frac{n\kappa_b}{\kappa_s}})$ compared with ADFS, which is always much larger than 1 in machine learning applications\footnote{As discussed in Section \ref{sec:stointro} for the comparison between the VR methods and gradient descent, stochastic methods have no advantage when $\kappa_s\approx n\kappa_b$. We often assume $\kappa_s\ll n\kappa_b$.}, and it is of the $\bO(\sqrt{n})$ order in the worst case. \citet{ADFS} proved the $ \varOmega((\sqrt{n\kappa_s}+n)\log\frac{1}{\epsilon})$ stochastic gradient computation and the $\varOmega(\sqrt{\kappa_b\kappa_c}\log\frac{1}{\epsilon})$ communication lower bounds. The study on acceleration for the general stochastic problems without finite-sum structures can be found in \citep{Dvinskikh2019}, \citep{Dvinskikh20192}, and \citep{Fallah2019}. See the recent review \citep{Dvinskikh20193} for the accelerated stochastic decentralized algorithms.

\subsection{Contributions}

Although both the decentralized methods and stochastic methods have been well studied, their combination still has much work to do. For example, as far as we know, there is no gradient-type stochastic decentralized method achieving both the state-of-the-art communication and stochastic gradient computation complexities (either accelerated or non-accelerated) of the decentralized methods and stochastic methods simultaneously. In this paper we aim to address this issue. Our contributions include:

\begin{table*}
\caption{Comparisons of various state-of-the-art decentralized and stochastic algorithms. See (\ref{def_f}) and (\ref{define-kappa-c}) for the definitions of $\kappa_b$, $\kappa_s$, and $\kappa_c$. $\widetilde \bO$ hides the poly-logarithmic factors. The complexities of Acc-VR-EXTRA and Acc-VR-DIGing hold under some conditions to restrict the size of $\kappa_c$. See part 1 of Remarks \ref{acc-remark} and \ref{acc-remark2}. Acc-VR-EXTRA-CA and Acc-VR-DIGing-CA remove these restrictions.}\label{table-comp}
\begin{center}
\footnotesize
\begin{tabular}{|c|c|c|c|}
\hline
 Methods & \tabincell{c}{stochastic gradient\\computation complexity}      & \tabincell{c}{communication\\complexity} & \tabincell{c}{dual\\ gradient\\based ?} \\
\hline
\multicolumn{4}{c}{Full batch decentralized algorithms}\\
\hline
 \tabincell{c}{EXTRA\\ \citep{shi2015extra}\\\citep{li-2019-extra}} & $\bO\left(n\left(\kappa_b+\kappa_c\right)\log\frac{1}{\epsilon}\right)$ & $\bO\left(\left(\kappa_b+\kappa_c\right)\log\frac{1}{\epsilon}\right)$ & no\\
 \tabincell{c}{DIGing\\ \citep{shi2017}\\\citep{sulaiman2020}} & $\bO\left(n\left(\kappa_b+\kappa_c^2\right)\log\frac{1}{\epsilon}\right)$ & $\bO\left(\left(\kappa_b+\kappa_c^2\right)\log\frac{1}{\epsilon}\right)$ & no\\
 \tabincell{c}{MSDA+CA\\ \citep{dasent}}  & $\bO\left(n\sqrt{\kappa_b}\log\frac{1}{\epsilon}\right)$ & $\bO\left(\sqrt{\kappa_b\kappa_c}\log\frac{1}{\epsilon}\right)$ & yes\\
 \tabincell{c}{APAPC+CA\\ \citep{richtaric-2020-dist}}  & $\bO\left(n\sqrt{\kappa_b}\log\frac{1}{\epsilon}\right)$ & $\bO\left(\sqrt{\kappa_b\kappa_c}\log\frac{1}{\epsilon}\right)$ & no\\
\hline
\multicolumn{4}{c}{Stochastic algorithms on a single machine}\\
\hline
 \tabincell{c}{VR methods\\ \citep{SAG}\\\citep{SAGA}\\\citep{SVRG}} & $\bO\left(\left(\kappa_s+n\right)\log\frac{1}{\epsilon}\right)$ & $\backslash$ & no\\
 \tabincell{c}{Katyusha\\ \citep{kat}}  & $\bO\left(\left(\sqrt{n\kappa_s}+n\right)\log\frac{1}{\epsilon}\right)$ & $\backslash$ & no\\
\hline
\multicolumn{4}{c}{Stochastic decentralized algorithms}\\
\hline
 \tabincell{c}{GT-SAGA\\ \citep{ranxin2019}} & $\bO\left(\left(\kappa_s^2\kappa_c^2+n\right)\log\frac{1}{\epsilon}\right)$ & $\bO\left(\left(\kappa_s^2\kappa_c^2+n\right)\log\frac{1}{\epsilon}\right)$ & no\\
 \tabincell{c}{GT-SVRG\\ \citep{ranxin2019}} & $\bO\left(\left(\kappa_s^2\kappa_c^2\log\kappa_s+n\right)\log\frac{1}{\epsilon}\right)$ & $\bO\left(\left(\kappa_s^2\kappa_c^2\log\kappa_s+n\right)\log\frac{1}{\epsilon}\right)$ & no\\
 \tabincell{c}{ADFS\\ \citep{ADFS}} & $\bO\left(\left(\sqrt{n\kappa_s}+n\right)\log\frac{1}{\epsilon}\right)$ & $\bO\left(\sqrt{\kappa_b\kappa_c}\log\frac{1}{\epsilon}\right)$ & yes\\
 \tabincell{c}{DVR+CA\\ \citep{DVR}} & $\bO\left(\left(\kappa_s+n\right)\log\frac{1}{\epsilon}\right)$ & $\bO\left(\kappa_b\sqrt{\kappa_c}\log\frac{1}{\epsilon}\right)$ & no\\
 \tabincell{c}{DVR+Catalyst\\ \citep{DVR}} & $\widetilde \bO\left(\left(\sqrt{n\kappa_s}+n\right)\log\frac{1}{\epsilon}\right)$ & $\widetilde \bO\left(\sqrt{\kappa_b\kappa_c}\sqrt{\frac{n\kappa_b}{\kappa_s}}\log\frac{1}{\epsilon}\right)$ & no\\
\hline
\tabincell{c}{Lower bounds\\ \citep{ADFS}}   & $\varOmega\left(\left(\sqrt{n\kappa_s}+n\right)\log\frac{1}{\epsilon}\right)$ & $\varOmega\left(\sqrt{\kappa_b\kappa_c}\log\frac{1}{\epsilon}\right)$ & $\backslash$\\
\hline
\multicolumn{4}{c}{Our results for stochastic decentralized optimization}\\
\hline
 VR-EXTRA        & $\bO\left(\left(\kappa_s+n\right)\log\frac{1}{\epsilon}\right)$ & $\bO\left(\left(\kappa_b+\kappa_c\right)\log\frac{1}{\epsilon}\right)$ & no\\&&&\\
 VR-DIGing       & $\bO\left(\left(\kappa_s+n\right)\log\frac{1}{\epsilon}\right)$ & $\bO\left(\left(\kappa_b+\kappa_c^2\right)\log\frac{1}{\epsilon}\right)$ & no\\&&&\\
 Acc-VR-EXTRA    & $\bO\left(\left(\sqrt{n\kappa_s}+n\right)\log\frac{1}{\epsilon}\right)$ & $\bO\left(\sqrt{\kappa_b\kappa_c}\log\frac{1}{\epsilon}\right)$ & no\\&&&\\
 Acc-VR-DIGing   & $\bO\left(\left(\sqrt{n\kappa_s}+n\right)\log\frac{1}{\epsilon}\right)$ & $\bO\left(\kappa_c\sqrt{\kappa_b}\log\frac{1}{\epsilon}\right)$ & no\\&&&\\
 Acc-VR-EXTRA+CA   & $\bO\left(\left(\sqrt{n\kappa_s}+n\right)\log\frac{1}{\epsilon}\right)$ & $\bO\left(\sqrt{\kappa_b\kappa_c}\log\frac{1}{\epsilon}\right)$ & no\\&&&\\
 Acc-VR-DIGing+CA   & $\bO\left(\left(\sqrt{n\kappa_s}+n\right)\log\frac{1}{\epsilon}\right)$ & $\bO\left(\sqrt{\kappa_b\kappa_c}\log\frac{1}{\epsilon}\right)$ & no\\
\hline
\hline
\end{tabular}
\end{center}
\end{table*}

\begin{enumerate}
\item We extend the widely used EXTRA and DIGing methods to deal with large-scale distributed data by combining them with the powerful VR technique. We prove the $\bO((\kappa_s+n)\log\frac{1}{\epsilon})$ stochastic gradient computation complexity and the $\bO((\kappa_b+\kappa_c)\log\frac{1}{\epsilon})$ communication complexity for VR-EXTRA, which are the best complexities among the non-accelerated stochastic decentralized methods as far as we know. The stochastic gradient computation complexity is the same as the single-machine VR methods, while the communication complexity is the same as the full batch EXTRA. For VR-DIGing, we establish the $\bO((\kappa_s+n)\log\frac{1}{\epsilon})$ stochastic gradient computation complexity and the $\bO((\kappa_b+\kappa_c^2)\log\frac{1}{\epsilon})$ communication complexity. The latter one is a little worse than that of VR-EXTRA on the dependence of $\kappa_c$. Due to the parallelism across $m$ nodes, running VR-EXTRA and VR-DIGing with $mn$ samples is as fast as running the single-machine VR methods with $n$ samples.

\item To further speed up the convergence, we combine EXTRA and DIGing with the accelerated VR technique. The proposed Acc-VR-EXTRA achieves the optimal $\bO((\sqrt{n\kappa_s}+n)\log\frac{1}{\epsilon})$ stochastic gradient computation complexity and the optimal $\bO(\sqrt{\kappa_b\kappa_c}\log\frac{1}{\epsilon})$ communication complexity under some mild conditions to restrict the size of $\kappa_c$. The proposed Acc-VR-DIGing has the optimal $\bO((\sqrt{n\kappa_s}+n)\log\frac{1}{\epsilon})$ stochastic gradient computation complexity and the $\bO(\kappa_c\sqrt{\kappa_b}\log\frac{1}{\epsilon})$ communication complexity with a little worse dependence on $\kappa_c$. The two methods are implemented in a single loop, and thus they are practical. We further combine Acc-VR-EXTRA and Acc-VR-DIGing with the Chebyshev acceleration to remove the restrictions on the size of $\kappa_c$, and improve the communication complexity of Acc-VR-DIGing to be optimal. Our complexities do not hide any poly-logarithmic factor. To the best of our knowledge, our methods are the first to exactly achieve both the optimal stochastic gradient computation complexity and the communication complexity in the class of gradient-type methods.

\end{enumerate}

Table \ref{table-comp} summarizes the complexity comparisons to the state-of-the-art stochastic decentralized methods. Our VR-EXTRA has the same stochastic gradient computation complexity as DVR-CA, but our communication cost is lower than theirs when $\kappa_c\leq \bO(\kappa_b^2)$. On the other hand, by combining with Chebyshev acceleration, our VR-EXTRA and VR-DIGing can also obtain the $\bO(\kappa_b\sqrt{\kappa_c}\log\frac{1}{\epsilon})$ communication complexity. For the accelerated methods, our Acc-VR-EXTRA-CA and Acc-VR-DIGing-CA outperform DVR-Catalyst on the stochastic gradient computation complexity at least by the poly-logarithmic factor $\bO(\log\kappa_b)$, and our communication cost is also lower than that of DVR-Catalyst by the factor $\bO\left(\sqrt{\frac{n\kappa_b}{\kappa_s}}\right)$. On the other hand, DVR and its Catalyst acceleration require $\bO(np)$ memory at each node, while our methods only need $\bO(p)$ memory\footnote{This is similar to the memory cost comparison between SAG/SAGA and SVRG.}.
Although ADFS has the same complexities as our Acc-VR-EXTRA-CA and Acc-VR-DIGing-CA, our methods are gradient-type methods, while theirs requires to compute the dual gradient at each iteration, which is much more expensive.

\section{Non-accelerated Variance Reduced EXTRA and DIGing}
We first review the classical EXTRA and DIGing methods in Section \ref{sec-non-acc-rev}. Then  we develop the variance reduced EXTRA and DIGing in Sections \ref{sec-non-acc-vr} and \ref{sec-non-acc-vr2}. At last, we discuss the complexities of the proposed methods in Section \ref{sec-non-acc-comp}.
\subsection{Review of EXTRA and DIGing}\label{sec-non-acc-rev}
A traditional way to analyze the decentralized optimization model is to write problem (\ref{problem}) in the following equivalent manner:
\begin{eqnarray}\notag
\min_{x_{(1)},\cdots,x_{(m)}} \sum_{i=1}^m f_{(i)}(x_{(i)}),\quad \mbox{s.t.}\quad x_{(1)}=x_{(2)}=\cdots=x_{(m)}.
\end{eqnarray}
Following \citep{sulaiman2020} and using the notations in (\ref{aggregate}), we further reformulate the above problem as the following linearly constrained problem:
\begin{eqnarray}\label{problem_consd}
\min_{\x} f(\x)+\frac{1}{2\alpha}\|\V\x\|^2,\quad \mbox{s.t.}\quad \U\x=0,
\end{eqnarray}
where the symmetric matrices $\U\in\R^{m\times m}$ and $\V\in\R^{m\times m}$ satisfy
\begin{equation}
\U\x=0\Leftrightarrow  x_{(1)}=\cdots=x_{(m)}\quad\mbox{and}\quad \V\x=0\Leftrightarrow x_{(1)}=\cdots=x_{(m)}.\label{UV-define}
\end{equation}
where $\frac{1}{2\alpha}\|\V\x\|^2$ can be regarded as the augmented term in the augmented Lagrange method \citep{Bertsekas-1982}, which may speed up the convergence than the methods based on the pure Lagrangian function. Introducing the following augmented Lagrangian function
\begin{equation}\notag
L(\x,\lambda)=f(\x)+\frac{1}{2\alpha}\|\V\x\|^2+\frac{1}{\alpha}\<\U\x,\blambda\>,
\end{equation}
we can apply the basic gradient method with a step-size $\alpha$ in the Gauss$-$Seidel-like order to compute the saddle point of problem (\ref{problem_consd}), which leads to the following iterations \citep{sulaiman2020,shi2017}:
\begin{eqnarray}
\begin{aligned}\label{alm}
&\x^{k+1}=\x^k-\left(\alpha\nabla f(\x^k)+\U\blambda^k+\V^2\x^k\right),\\
&\blambda^{k+1}=\blambda^k+\U\x^{k+1}.
\end{aligned}
\end{eqnarray}

Iteration (\ref{alm}) is a unified algorithmic framework, and different choices of $\U$ and $\V$ give different methods \citep{sulaiman2020}. Specifically, when we choose $\U=\sqrt{\frac{\I-\W}{2}}$ and $\V=\sqrt{\frac{\I-\W}{2}}$, (\ref{alm}) reduces to the famous EXTRA algorithm \citep{shi2015extra}, which consists of the following iterations:
\begin{eqnarray}
\begin{aligned}\notag
\x^{k+1}=(\I+\W)\x^k-\frac{\I+\W}{2}\x^{k-1}-\alpha\left(\nabla f(\x^k)-\nabla f(\x^{k-1})\right).
\end{aligned}
\end{eqnarray}
When we choose $\U=\I-\W$ and $\V=\sqrt{\I-\W^2}$, (\ref{alm}) reduces to the DIGing \citep{shi2017} method with the following iterations:
\begin{eqnarray}
\begin{aligned}\notag
&\s^{k+1}=\W\s^k+\nabla f(\x^k)-\nabla f(\x^{k-1}),\\
&\x^{k+1}=\W\x^k-\alpha\s^{k+1}.
\end{aligned}
\end{eqnarray}
Both EXTRA and DIGing rely on tracking the differences of gradients at each iteration.
\subsection{Development of VR-EXTRA and VR-DIGing}\label{sec-non-acc-vr}
Now, we come to extend EXTRA and DIGing with the variance reduction technique proposed in SVRG \citep{SVRG}. Specifically, SVRG maintains a snapshot vector $w_{(i)}^k$ after several SGD iterations, and keeps an iterative estimator $\widetilde\nabla f_{(i)}(x_{(i)}^k)=\nabla f_{(i),j}(x_{(i)}^k)-\nabla f_{(i),j}(w_{(i)}^k)+\nabla f_{(i)}(w_{(i)}^k)$ of the full batch gradient for some randomly selected $j$. When extending EXTRA and DIGing to stochastic decentralized optimization, a straightforward idea is to replace the local gradient $\nabla f_{(i)}(x_{(i)}^k)$ in (\ref{alm}) by its VR estimator $\widetilde \nabla f_{(i)}(x_{(i)}^k)$. However, in this way the resultant algorithm needs the same number of stochastic gradient evaluations and communication rounds to precision $\epsilon$. As summarized in Table \ref{table-comp}, our goal is to provide computation and communication complexities matching those of SVRG and EXTRA/DIGing, respectively, which are not equal. To address this issue, we use the mini-batch VR technique, that is, select $b$ independent samples with replacement as a mini-batch $\S_{(i)}$, and use this mini-batch to update the VR estimator. By carefully choosing the mini-batch size $b$, we can balance the communication and stochastic gradient computation costs. Moreover, to simplify the algorithm development and analysis, we adopt the loopless SVRG proposed in \citep{Kovalev-2020-loopless}. Combining the above ideas, we have the following VR variant of (\ref{alm}) described in a distributed way:
\begin{subequations}
\begin{align}
&\nabla_{(i)}^k=\frac{1}{b}\sum_{j\in\S_{(i)}^k}\frac{1}{np_{(i),j}}\left(\nabla f_{(i),j}(x_{(i)}^k)-\nabla f_{(i),j}(w_{(i)}^k)\right)+\nabla f_{(i)}(w_{(i)}^k), \quad\forall i,\label{non-acc-s1}\\
&x_{(i)}^{k+1}=x_{(i)}^k-\left(\alpha\nabla_{(i)}^k+\sum_{j\in\N_{(i)}}\U_{ij}\blambda_{(j)}^k+\sum_{j\in\N_{(i)}}(\V^2)_{ij} x_{(j)}^k\right),\quad\forall i,\label{non-acc-s2}\\
&\blambda_{(i)}^{k+1}=\blambda_{(i)}^k+\sum_{j\in\N_{(i)}}\U_{ij}x_{(j)}^{k+1},\quad\forall i,\label{non-acc-s3}\\
&w_{(i)}^{k+1}=\left\{
  \begin{array}{l}
    x_{(i)}^k\mbox{ with probability }\frac{b}{n},\\
    w_{(i)}^k\mbox{ with probability }1-\frac{b}{n},
  \end{array}
\right. \forall i,\label{non-acc-s4}
\end{align}
\end{subequations}
where the mini-batch VR estimator update rule (\ref{non-acc-s1}) is motivated by \citep{kat}, in which each sample $j$ on node $i$ is selected with probability $p_{(i),j}=\frac{L_{(i),j}}{\sum_{j=1}^nL_{(i),j}}$. The probabilistic update of the snapshot vector in (\ref{non-acc-s4}) is motivated by \citep{Kovalev-2020-loopless}, in which we update the full batch gradient $\nabla f_{(i)}(w_{(i)}^{k+1})$ if $w_{(i)}^{k+1}=x_{(i)}^k$; otherwise, we use the old one. Steps (\ref{non-acc-s2}) and (\ref{non-acc-s3}) come from (\ref{alm}), but replacing the local gradients by their VR estimators. In steps (\ref{non-acc-s1}) and (\ref{non-acc-s4}), each node selects $\S_{(i)}^k$ and computes $w_{(i)}^{k+1}$ independent of the other nodes.

At last, we write (\ref{non-acc-s1})-(\ref{non-acc-s4}) in the EXTRA/DIGing style. Similar to (\ref{aggregate}), we denote
\begin{eqnarray}
\begin{aligned}\label{aggrgate-nabla}
\quad\nabla^k=\left(
  \begin{array}{c}
    (\nabla_{(1)}^k)^T\\
    \vdots\\
    (\nabla_{(m)}^k)^T
  \end{array}
\right)
\end{aligned}
\end{eqnarray}
to simplify the algorithm description. From steps (\ref{non-acc-s2}) and (\ref{non-acc-s3}), we have
\begin{eqnarray}
\begin{aligned}\label{alg-cont1}
\x^{k+1}=(2\I-\U^2-\V^2)\x^k-(\I-\V^2)\x^{k-1}-\alpha\left(\nabla^k-\nabla^{k-1}\right)
\end{aligned}
\end{eqnarray}
in the compact form. Plugging $\U=\sqrt{\frac{\I-\W}{2}}$ and $\V=\sqrt{\frac{\I-\W}{2}}$ into (\ref{alg-cont1}), we have
\begin{eqnarray}
\begin{aligned}\notag
\x^{k+1}=(\I+\W)\x^k-\frac{\I+\W}{2}\x^{k-1}-\alpha\left(\nabla^k-\nabla^{k-1}\right),
\end{aligned}
\end{eqnarray}
which is the VR variant of EXTRA, called VR-EXTRA. Plugging $\U=\I-\W$ and $\V=\sqrt{\I-\W^2}$ into (\ref{alg-cont1}), we have
\begin{eqnarray}
\begin{aligned}\notag
\x^{k+1}=2\W\x^k-\W^2\x^{k-1}-\alpha\left(\nabla^k-\nabla^{k-1}\right),
\end{aligned}
\end{eqnarray}
which is further equivalent to the following method, called VR-DIGing,
\begin{eqnarray}
\begin{aligned}\notag
&\s^{k+1}=\W\s^k+\nabla^k-\nabla^{k-1},\\
&\x^{k+1}=\W\x^k-\alpha\s^{k+1}.
\end{aligned}
\end{eqnarray}

We see that VR-EXTRA and VR-DIGing are quite similar to the original EXTRA and DIGing. The only difference is that we replace the local gradients by their VR estimators. Thus the implementation is as simple as that of the original EXTRA and DIGing. We give the specific descriptions of VR-EXTRA and VR-DIGing in Algorithm \ref{extra} in a distributed way, including the parameter settings. To discuss EXTRA and DIGing in a unified framework, we denote
\begin{eqnarray}\label{kappa-def}
\kappa=2\kappa_c\mbox{ for EXTRA}\qquad\mbox{and}\qquad\kappa=\kappa_c^2\mbox{ for DIGing.}
\end{eqnarray}
See Lemma \ref{non-acc-lemma-uni} for the reason. We will use $\kappa$ frequently in this paper when we do not distinguish EXTRA and DIGing, and the readers can use (\ref{kappa-def}) to get the specific properties of EXTRA and DIGing, respectively.

\begin{algorithm}[t]
   \caption{VR-EXTRA and VR-DIGing}
   \label{extra}
\begin{algorithmic}
   \STATE Initialize: $x_{(i)}^0=w_{(i)}^0=x_{int}$, $\blambda_{(i)}^0=0$, compute $x_{(i)}^1$ and $w_{(i)}^1$ for all $i$ by (\ref{non-acc-s2}) and (\ref{non-acc-s4}), respectively. Let $\alpha=\bO(\frac{1}{\max\{L_f,\kappa\mu\}})$ and $b=\frac{\max\{\overline L_f,n\mu\}}{\max\{L_f,\kappa\mu\}}$, where $\kappa=2\kappa_c$ for EXTRA, and $\kappa=\kappa_c^2$ for DIGing. Let $s_{(i)}^1=\nabla f_{(i)}(w_{(i)}^0)$ for DIGing.
   \STATE Let distribution $\D_{(i)}$ be to output $j\in[1,n]$ with probability $p_{(i),j}=\frac{L_{(i),j}}{n\overline L_{(i)}}$.
   \FOR{$k=1,2,...$}
   \STATE Step 1: $\S_{(i)}^k\leftarrow b$ independent samples from $\D_{(i)}$ with replacement, $\forall i$,
   \STATE Step 2: Compute $\nabla_{(i)}^k$ by (\ref{non-acc-s1}), $\forall i$,
   \STATE Step 3: For EXTRA, compute $x_{(i)}^{k+1}$ by
   \begin{equation}\notag
   x_{(i)}^{k+1}=\left(x_{(i)}^k+\sum_{j\in\N_{(i)}}\W_{ij}x_{(j)}^k\right)-\frac{1}{2}\left(x_{(i)}^{k-1}+\sum_{j\in\N_{(i)}}\W_{ij}x_{(j)}^{k-1}\right)-\alpha\left(\nabla_{(i)}^k-\nabla_{(i)}^{k-1}\right), \forall i,
   \end{equation}
   \hspace*{1.2cm} For DIGing, compute $x_{(i)}^{k+1}$ by
   \begin{equation}\notag
   s_{(i)}^{k+1}=\sum_{j\in\N_{(i)}}\W_{ij}s_{(j)}^k+\nabla_{(i)}^k-\nabla_{(i)}^{k-1},\qquad x_{(i)}^{k+1}=\sum_{j\in\N_{(i)}}\W_{ij}x_{(j)}^k-\alpha s_{(i)}^{k+1},\quad\forall i,
   \end{equation}
   \STATE Step 4: Compute $w_{(i)}^{k+1}$ by (\ref{non-acc-s4}), $\forall i$.
   \ENDFOR
\end{algorithmic}
\end{algorithm}

\subsection{Extension to Large $\kappa$}\label{sec-non-acc-vr2}
The particular choice of the mini-batch size $b$ in Algorithm \ref{extra} may be smaller than 1 when $\kappa$ is large, which makes the algorithm meaningless. We discuss EXTRA and DIGing in a unified way in this section, so we use $\kappa$ in this section, which is defined by (\ref{kappa-def}). In fact, $b\geq 1$ if and only if $\kappa\leq\max\{\kappa_s,n\}$, see the proof of Theorem \ref{non-acc-theorem} in Section \ref{sec-proof}. In this section we consider the case of $\kappa>\max\{\kappa_s,n\}$.

Intuitively speaking, when $\kappa$ is very large such that $\kappa_b+\kappa\geq \kappa_s+n$, to reach the desired $\bO((\kappa_b+\kappa)\log\frac{1}{\epsilon})$ communication complexity and the $\bO((\kappa_s+n)\log\frac{1}{\epsilon})$ stochastic gradient computation complexity, as summarized in Table \ref{table-comp}, we should perform less than 1 stochastic gradient evaluation in average at each iteration. This observation motivates us to introduce some zero samples, that is to say, let $f_{(i),n+1}=\cdots=f_{(i),n'}=0$ for all $i$, and consider problem
\begin{eqnarray}\label{problem2}
\min_{x\in\R^p} \sum_{i=1}^m f_{(i)}'(x),\quad\mbox{where}\quad f_{(i)}'(x)=\frac{1}{n'}\sum_{j=1}^{n'} f_{(i),j}(x).
\end{eqnarray}
The zero samples do not spend time to compute the stochastic gradient. We see that problems (\ref{problem2}) and (\ref{problem}) are equivalent. To use Algorithm \ref{extra} to solve problem (\ref{problem2}), we denote
\begin{equation}\notag
L_{(i),j}=\frac{n\mu n'-n\overline L_{(i)}}{n'-n}\quad\mbox{for all}\quad  n<j\leq n',
\end{equation}
and let each sample be selected with probability $\frac{L_{(i),j}}{\sum_{j=1}^{n'}L_{(i),j}}$. Then  we select the samples in $[1,n]$ with probability $\frac{\overline L_{(i)}}{\mu n'}$, and select the zero samples with probability $1-\frac{\overline L_{(i)}}{\mu n'}$. It can be seen that $f_{(i)}'(x)$ is $\frac{nL_{(i)}}{n'}$-smooth and $\frac{n\mu}{n'}$-strongly convex. Define the following notations:
\begin{equation}\label{new-L2}
n'=\kappa,\quad \mu'=\frac{n\mu}{n'},\quad L_f'=\max_i\frac{nL_{(i)}}{n'}=\frac{nL_f}{n'},\quad \overline L_{(i)}'=\frac{\sum_{j=1}^{n'}L_{(i),j}}{n'},\quad \overline L_f'=\max_i \overline L_{(i)}'.
\end{equation}
We can easily check $\alpha=\bO(\frac{1}{\max\{L_f',\kappa\mu'\}})=\bO(\frac{1}{n\mu})$, and $b=\frac{\max\{\overline L_f',n'\mu'\}}{\max\{L_f',\kappa\mu'\}}=1$. See the proof of Theorem \ref{non-acc-theorem2} in Section \ref{sec-proof}. Then we can use Algorithm \ref{extra} to solve problem (\ref{problem2}).

\subsection{Complexities}\label{sec-non-acc-comp}
We prove the convergence of VR-EXTRA and VR-DIGing in a unified framework. From Assumption \ref{assumption_w}, we have the following easy-to-identify lemma, where the third inequality in (\ref{UVbound}) can be proved similarly to Lemma 4 in \citep{li-2018-pm}.
\begin{lemma}\label{non-acc-lemma-uni}
Suppose that Assumption \ref{assumption_w} holds with $\omega=0$ for EXTRA. Let $\U=\V=\sqrt{\frac{\I-\W}{2}}$. Then we have
\begin{eqnarray}
\quad\qquad\|\U\x\|^2\leq\|\V\x\|^2,\quad \|\V\x\|^2\leq\frac{1}{2}\|\x\|^2, \quad\mbox{and}\quad\|\U\blambda\|^2\geq\frac{1}{\kappa}\|\blambda\|^2,\forall \blambda\in\mbox{Span}(\U),\label{UVbound}
\end{eqnarray}
where $\kappa=\frac{2}{1-\sigma_2(\W)}=2\kappa_c$. Suppose that Assumption \ref{assumption_w} holds with $\omega=\frac{\sqrt{2}}{2}$ for DIGing. Let $\U=\I-\W$ and $\V=\sqrt{\I-\W^2}$. Then (\ref{UVbound}) also holds with $\kappa=\frac{1}{(1-\sigma_2(\W))^2}=\kappa_c^2$.
\end{lemma}

Denote the following set of random variables:
\begin{equation}
\S^k=\cup_{i=1}^m\S_{(i)}^k,\qquad \xi^k=\{\S^0,\w^1,\S^1,\w^2\cdots,\S^{k-1},\w^k\}.\notag
\end{equation}
The next theorem gives the communication complexity and stochastic gradient computation complexity of algorithm (\ref{non-acc-s1})-(\ref{non-acc-s4}) in a unified way.

\begin{theorem}\label{non-acc-theorem}
Suppose that Assumption \ref{assumption_f} holds, and $\U$ and $\V$ satisfy (\ref{UV-define}) and (\ref{UVbound}). Let $\alpha=\frac{1}{28\max\{L_f,\kappa\mu\}}$ and $\blambda^0=0$.
\begin{enumerate}
\item If $\kappa\leq \max\{\kappa_s,n\}$, let $b=\frac{\max\{\overline L_f,n\mu\}}{\max\{L_f,\kappa\mu\}}$. Then algorithm (\ref{non-acc-s1})-(\ref{non-acc-s4}) requires the time of $\bO((\kappa_b+\kappa)\log\frac{1}{\epsilon})$ communication rounds and $\bO((\kappa_s+n)\log\frac{1}{\epsilon})$ stochastic gradient evaluations to find $\x^k$ such that $\E_{\xi^k}[\|\x^k-\x^*\|^2]\leq\epsilon$.

\item If $\kappa\geq \max\{\kappa_s,n\}$, let $b=1$. Then algorithm (\ref{non-acc-s1})-(\ref{non-acc-s4}) requires the time of $\bO((\kappa_b+\kappa)\log\frac{1}{\epsilon})$ communication rounds and $\bO((\kappa_b+\kappa)\log\frac{1}{\epsilon})$ stochastic gradient evaluations to find $\x^k$ such that $\E_{\xi^k}\big[\|\x^k-\x^*\|^2\big]\leq\epsilon$.
\end{enumerate}
\end{theorem}

\begin{remark}\label{remark3}
Let's explain the time of one communication rounds and one stochastic gradient evaluation. At each iteration, algorithm (\ref{non-acc-s1})-(\ref{non-acc-s4}) performs one round of communication, that is, each node $i$ receives information $x_{(j)}^k$ and $\blambda_{(j)}^k$ from its neighbors for all $j\in\N_{(i)}$. Then each node $i$ selects $\S_{(i)}^k$ randomly and computes $\nabla_{(i)}^k$ with $b$ stochastic gradient evaluations. $\nabla f_{(i)}(w_{(i)}^{k+1})$ is updated with probability $b/n$, and each time with $n$ stochastic gradient evaluations. So each node computes $b$ stochastic gradients in average at each iteration. Since the computation is performed in parallel across all the nodes, we say that each iteration requires the time of one communication round and $b$ stochastic gradient evaluations in average. 
\end{remark}

From Theorem \ref{non-acc-theorem}, we see that when $\kappa\geq\max\{\kappa_s,n\}$, the stochastic gradient computation cost increases to $\bO((\kappa_b+\kappa)\log\frac{1}{\epsilon})$. We can use the zero-sample strategy described in Section \ref{sec-non-acc-vr2} to reduce the computation cost to $\bO((\kappa_s+n)\log\frac{1}{\epsilon})$, as described in the following theorem.
\begin{theorem}\label{non-acc-theorem2}
Suppose that Assumption \ref{assumption_f} and conditions (\ref{UV-define}) and (\ref{UVbound}) hold. Assume $\kappa> \max\{\kappa_s,n\}$. Applying Algorithm \ref{extra} to solve problem (\ref{problem2}), it requires the time of $\bO((\kappa_b+\kappa)\log\frac{1}{\epsilon})$ communication rounds and $\bO((\kappa_s+n)\log\frac{1}{\epsilon})$ stochastic gradient evaluations to find an $\epsilon$-precision solution of problem (\ref{problem}) such that $\E_{\xi^k}[\|\x^k-\x^*\|^2]\leq\epsilon$.
\end{theorem}
We see that by introducing the zero samples with carefully designed $n'$ and $L_{(i),j}$ for $n<j\leq n'$, the complexities in Theorem \ref{non-acc-theorem2} keep the same as those in part one of Theorem \ref{non-acc-theorem}.

For the particular VR-EXTRA and VR-DIGing methods, we have the following complexities accordingly, where we replace $\kappa$ in Theorems \ref{non-acc-theorem} and \ref{non-acc-theorem2} by $2\kappa_c$ and $\kappa_c^2$, respectively.
\begin{corollary}
Suppose that Assumptions \ref{assumption_f} and \ref{assumption_w} hold with $\omega=0$. Use the zero-sample strategy if $2\kappa_c\geq\max\{\kappa_s,n\}$. Then the VR-EXTRA method in Algorithm \ref{extra} requires the time of $\bO((\kappa_b+\kappa_c)\log\frac{1}{\epsilon})$ communication rounds and $\bO((\kappa_s+n)\log\frac{1}{\epsilon})$ stochastic gradient evaluations to find $\x^k$ such that $\E_{\xi^k}[\|\x^k-\x^*\|^2]\leq\epsilon$.
\end{corollary}

\begin{corollary}
Suppose that Assumptions \ref{assumption_f} and \ref{assumption_w} hold with $\omega=\frac{\sqrt{2}}{2}$. Use the zero-sample strategy if $\kappa_c^2\geq\max\{\kappa_s,n\}$. Then  the VR-DIGing method in Algorithm \ref{extra} requires the time of $\bO((\kappa_b+\kappa_c^2)\log\frac{1}{\epsilon})$ communication rounds and $\bO((\kappa_s+n)\log\frac{1}{\epsilon})$ stochastic gradient evaluations to find $\x^k$ such that $\E_{\xi^k}[\|\x^k-\x^*\|^2]\leq\epsilon$.
\end{corollary}

\begin{remark}\label{remark2}
$\vspace*{-0.4cm}\linebreak$
\begin{enumerate}
\item The communication complexity of VR-DIGing has a worse dependence on $\kappa_c$ than that of VR-EXTRA. This is because EXTRA uses $\U=\sqrt{\frac{\I-\W}{2}}$ in problem (\ref{problem_consd}), while DIGing uses $\U=\I-\W$. From Lemma \ref{non-acc-lemma-uni}, we see that different choice of $\U$ gives different order of $\kappa_c$.
\item From Table \ref{table-comp}, we see that EXTRA and VR-EXTRA have the same communication complexity, and DIGing and VR-DIGing also have the same communication complexity. Thus extending EXTRA and DIGing to stochastic decentralized optimization does not need to pay a price of more communication cost theoretically.
\item When $\kappa\leq\max\{\kappa_s,n\}$, running VR-EXTRA and VR-DIGing with $mn$ samples needs the time of $\bO((\kappa_s+n)\log\frac{1}{\epsilon})$ stochastic gradient evaluations by parallelism, which is the same as that of running the single-machine VR methods with $n$ samples when we ignore the communication time. On the other hand, when we run the single-machine VR methods with $mn$ samples, the required time increases to $\bO((\kappa_s+mn)\log\frac{1}{\epsilon})$. Thus the linear speedup is achieved when $n$ is larger than $\kappa_s$. The situation of $\kappa>\max\{\kappa_s,n\}$ is more complicated because at each iteration, some machines would be computing gradients, while others would be idle if the zero-sample is chosen. Parallelism is destroyed and the actual running time would be larger than the time of $\bO((\kappa_s+n)\log\frac{1}{\epsilon})$ stochastic gradient evaluations.
\item Both in theory and in practice, we can choose a larger mini-batch size $b$ than the particular choice given in Algorithm \ref{extra}, at the expense of a higher stochastic gradient computation complexity than $\bO((\kappa_s+n)\log\frac{1}{\epsilon})$. However, the communication complexity remains unchanged. See the proof of Theorem \ref{non-acc-theorem}. Denote $\tau$ to be the ratio between the practical running time of performing one communication round and one stochastic gradient computation. If $\kappa_s+n\leq \tau(\kappa_b+\kappa)$, that is, communications dominate the total running time, we can increase the mini-batch size to $\frac{\max\{\overline L_f,n\mu\}}{\max\{L_f,\kappa\mu\}}\frac{\tau(\kappa_b+\kappa)}{\kappa_s+n}=\tau$, which does not increase the total running time of $\bO(\tau(\kappa_b+\kappa)\log\frac{1}{\epsilon})$.
\end{enumerate}
\end{remark}

\begin{algorithm}[t]
   \caption{Acc-VR-EXTRA and Acc-VR-DIGing}
   \label{accextra}
\begin{algorithmic}
   \STATE Initialize: $x^0_{(i)}=z^0_{(i)}=w^0_{(i)}$, $\blambda^0_{(i)}=0$ for all $i$, $\alpha=\bO(\frac{1}{L_f})$, $b=\max\{\frac{\max\{\sqrt{n\overline L_f/\mu},n\}}{\max\{\sqrt{\kappa L_f/\mu},\kappa\}},\frac{\overline L_f}{L_f}\}$, $\theta_1=\min\{\frac{1}{2}\sqrt{\frac{\kappa\mu}{L_f}},\frac{1}{2}\}$, $\theta_2=\frac{\overline L_f}{2L_fb}$, where $\kappa=2\kappa_c$ for EXTRA, and $\kappa=\kappa_c^2$ for DIGing.
   \STATE Let $\U=\V=\sqrt{\frac{\I-\W}{2}}$ for EXTRA, and $\U=\I-\W$ and $\V=\sqrt{\I-\W^2}$ for DIGing.
   \STATE Let distribution $\D_{(i)}$ be to output $j\in[1,n]$ with probability $p_{(i),j}=\frac{L_{(i),j}}{n\overline L_{(i)}}$.
   \FOR{$k=0,1,2,...$}
   \STATE $\S_{(i)}^k\leftarrow b$ independent samples from $\D_{(i)}$ with replacement, $\forall i$,
   \STATE Perform steps (\ref{acc-s1})-(\ref{acc-s6}), $\forall i$,
   \ENDFOR
\end{algorithmic}
\end{algorithm}

\section{Accelerated Variance Reduced EXTRA and DIGing}
In this section, we develop the accelerated VR-EXTRA and VR-DIGing methods. In algorithm (\ref{non-acc-s1})-(\ref{non-acc-s4}), we combine (\ref{alm}) with the loopless SVRG to get the non-accelerated methods. To develop the accelerated methods, a straightforward idea is to combine (\ref{alm}) with the loopless Katyusha proposed in \citep{Kovalev-2020-loopless}, which leads to the following algorithm (\ref{acc-s1})-(\ref{acc-s6}). We give the parameter settings in Algorithm \ref{accextra}. We will not write (\ref{acc-s1})-(\ref{acc-s6}) in the EXTRA/DIGing style since the resultant methods are complex, and they are not very similar to the original EXTRA and DIGing besides the feature of tracking the differences of gradients.
\begin{subequations}
\begin{align}
&y_{(i)}^k=\theta_1z_{(i)}^k+\theta_2w_{(i)}^k+(1-\theta_1-\theta_2)x_{(i)}^k,\quad\forall i,\label{acc-s1}\\
&\nabla_{(i)}^k=\frac{1}{b}\sum_{j\in\S_{(i)}^k}\frac{1}{np_{(i),j}}\left(\nabla f_{(i),j}(y_{(i)}^k)-\nabla f_{(i),j}(w_{(i)}^k)\right)+\nabla f_{(i)}(w_{(i)}^k),\quad\forall i,\label{acc-s2}\\
&z_{(i)}^{k+1}=\frac{1}{1\hspace*{-0.05cm}+\hspace*{-0.05cm}\frac{\mu\alpha}{\theta_1}}\hspace*{-0.08cm}\left(\hspace*{-0.08cm}\frac{\mu\alpha}{\theta_1}y_{(i)}^k\hspace*{-0.05cm}+\hspace*{-0.05cm}z_{(i)}^k\hspace*{-0.05cm}-\hspace*{-0.05cm}\frac{1}{\theta_1}\hspace*{-0.08cm}\left(\hspace*{-0.08cm}\alpha\nabla_{(i)}^k\hspace*{-0.05cm}+\hspace*{-0.05cm}\sum_{j\in\N_{(i)}}\U_{ij}\blambda_{(j)}^k\hspace*{-0.05cm}+\hspace*{-0.05cm}\theta_1\sum_{j\in\N_{(i)}}(\V^2)_{ij}z_{(j)}^k\hspace*{-0.08cm}\right)\hspace*{-0.08cm}\right),\forall i,\label{acc-s3}\\
&\blambda_{(i)}^{k+1}=\blambda_{(i)}^k+\theta_1\sum_{j\in\N_{(i)}}\U_{ij}z_{(j)}^{k+1},\quad\forall i,\label{acc-s4}\\
&x_{(i)}^{k+1}=y_{(i)}^k+\theta_1\left(z_{(i)}^{k+1}-z_{(i)}^k\right),\quad\forall i,\label{acc-s5}\\
&w_{(i)}^{k+1}=\left\{
  \begin{array}{l}
    x_{(i)}^k\mbox{ with probability }\frac{b}{n},\\
    w_{(i)}^k\mbox{ with probability }1-\frac{b}{n},
  \end{array}
\right. \forall i.\label{acc-s6}
\end{align}
\end{subequations}

In the above algorithm, steps (\ref{acc-s1}) and (\ref{acc-s5}) are the Nesterov's acceleration steps, which are motivated by \citep{kat,Kovalev-2020-loopless}. Steps (\ref{acc-s3}) and (\ref{acc-s4}) involve the operation of $\U\x$, which is uncomputable for $\U=\sqrt{\frac{\I-\W}{2}}$ in EXTRA in the distributed environment. Introducing the auxiliary variable $\widetilde\blambda^k=\U\blambda^k$ and multiplying both sides of (\ref{acc-s4}) by $\U$ leads to
\begin{eqnarray}
\begin{aligned}\label{acc-extra2}
&\z^{k+1}=\frac{1}{1+\frac{\mu\alpha}{\theta_1}}\left(\frac{\mu\alpha}{\theta_1}\y^k+\z^k-\frac{1}{\theta_1}\left(\alpha\nabla^k+\widetilde\blambda^k+\theta_1\V^2\z^k\right)\right),\\
&\widetilde\blambda^{k+1}=\widetilde\blambda^k+\theta_1\U^2\z^{k+1},
\end{aligned}
\end{eqnarray}
in the compact form. From the definitions of $\U=\V=\sqrt{\frac{\I-\W}{2}}$, we only need to compute $\W\z$, which corresponds to the gossip-style communications. For DIGing, we do not need such auxiliary variables.
\subsection{Complexities}

Theorem \ref{acc-theorem} gives the complexities of algorithm (\ref{acc-s1})-(\ref{acc-s6}) in a unified way, and Corollaries \ref{corollary-extra} and \ref{corollary-diging} provide the complexities for the particular Acc-VR-EXTRA and Acc-VR-DIGing methods, respectively.

\begin{theorem}\label{acc-theorem}
Suppose that Assumption \ref{assumption_f} holds, and $\U$ and $\V$ satisfy (\ref{UV-define}) and (\ref{UVbound}). Let $\theta_1=\min\{\frac{1}{2}\sqrt{\frac{\kappa\mu}{L_f}},\frac{1}{2}\}$, $\theta_2=\frac{\overline L_f}{2L_fb}$, $\alpha=\frac{1}{10L_f}$, $\blambda^0=0$, and $b=\max\{\frac{\max\{\sqrt{n\overline L_f/\mu},n\}}{\max\{\sqrt{\kappa L_f/\mu},\kappa\}},\frac{\overline L_f}{L_f}\}$.
\begin{enumerate}
\item If $\kappa\leq\frac{nL_f}{\overline L_f}$, such that $b=\frac{\max\{\sqrt{n\overline L_f/\mu},n\}}{\max\{\sqrt{\kappa L_f/\mu},\kappa\}}$, then  algorithm (\ref{acc-s1})-(\ref{acc-s6}) requires the time of $\bO((\kappa+\sqrt{\kappa_b\kappa})\log\frac{1}{\epsilon})$ communication rounds and $\bO((\sqrt{n\kappa_s}+n)\log\frac{1}{\epsilon})$ stochastic gradient evaluations to find $\z^k$ such that $\E_{\xi^k}\big[\|\z^k-\x^*\|^2\big]\leq\epsilon$.
\item If $\kappa\geq\frac{nL_f}{\overline L_f}$, such that $b=\frac{\overline L_f}{L_f}$, then algorithm (\ref{acc-s1})-(\ref{acc-s6}) requires the time of $\bO((\kappa+\sqrt{\kappa_b\kappa})\log\frac{1}{\epsilon})$ communication rounds and $\bO(\frac{\overline L_f}{L_f}(\kappa+\sqrt{\kappa_b\kappa})\log\frac{1}{\epsilon})$ stochastic gradient evaluations to find $\z^k$ such that $\E_{\xi^k}\big[\|\z^k-\x^*\|^2\big]\leq\epsilon$.
\end{enumerate}
\end{theorem}

\begin{remark}\label{acc-remark}
$\vspace*{-0.4cm}\linebreak$
\begin{enumerate}
\item As introduced in Section \ref{sec-assumption}, we have $\kappa_b\leq\kappa_s\leq n\kappa_b$, and we always assume $\kappa_s\ll n\kappa_b$ in the analysis of stochastic algorithms. Thus we can expect $\frac{nL_f}{\overline L_f}$ to be large for large-scale data. On the other hand, $\kappa_c$ depends on the network scale and connectivity. For example, $\kappa_c=\bO(1)$ for the commonly used Erd\H{o}s$-$R\'{e}nyi random graph, and $\kappa_c=\bO(m\log m)$ for the geometric graph. In the worst case, for example, the linear graph or cycle graph, we have $\kappa_c=\bO(m^2)$ \citep{Nedic-2018}. Thus we can also expect $\kappa$ to be not very large when the number of distributed nodes is limited and the network is well connected. So we can expect that the assumption $\kappa\leq\frac{nL_f}{\overline L_f}$ always holds for large-scale distributed data, for example, thousands of nodes and each node with millions of data.
\item Similar to VR-EXTRA and VR-DIGing, both in theory and in practice, we can choose a larger mini-batch size $b$ than the particular choice given in Algorithm \ref{accextra}, at the expense of higher stochastic gradient computation complexities than the ones given in Theorem \ref{acc-theorem}. However, the $\bO((\sqrt{\kappa_b\kappa}+\kappa)\log\frac{1}{\epsilon})$ communication complexity remains unchanged. See the proof of Theorem \ref{acc-theorem}. We take part one of Theorem \ref{acc-theorem} as an example. Recall the definition of $\tau$ in Remark \ref{remark2}(4). If $\max\{\sqrt{n\kappa_s},n\}\leq \tau\max\{\sqrt{\kappa_b\kappa},\kappa\}$, that is, communications dominate the total running time, we can increase the mini-batch size to $\frac{\max\{\sqrt{n\overline L_f/\mu},n\}}{\max\{\sqrt{\kappa L_f/\mu},\kappa\}}\frac{\tau\max\{\sqrt{\kappa_b\kappa},\kappa\}}{\max\{\sqrt{n\kappa_s},n\}}=\tau$, which does not increase the total running time of $\bO(\tau(\sqrt{\kappa_b\kappa}+\kappa)\log\frac{1}{\epsilon})$.
\end{enumerate}
\end{remark}
\begin{corollary}\label{corollary-extra}
Suppose that Assumptions \ref{assumption_f} and \ref{assumption_w} hold with $\omega=0$. Under the parameter settings in Theorem \ref{acc-theorem} with $\kappa=2\kappa_c$, if $2\kappa_c\leq\frac{nL_f}{\overline L_f}$ and $\kappa_c\leq\kappa_b$, the Acc-VR-EXTRA algorithm requires the time of $\bO((\kappa_c+\sqrt{\kappa_b\kappa_c})\log\frac{1}{\epsilon})=\bO(\sqrt{\kappa_b\kappa_c}\log\frac{1}{\epsilon})$ communication rounds and $\bO((\sqrt{n\kappa_s}+n)\log\frac{1}{\epsilon})$ stochastic gradient evaluations to find $\z^k$ such that $\E_{\xi^k}[\|\z^k-\x^*\|^2]\leq\epsilon$.
\end{corollary}

\begin{corollary}\label{corollary-diging}
Suppose that Assumptions \ref{assumption_f} and \ref{assumption_w} hold with $\omega=\frac{\sqrt{2}}{2}$. Under the parameter settings in Theorem \ref{acc-theorem} with $\kappa=\kappa_c^2$, if $\kappa_c^2\leq\frac{nL_f}{\overline L_f}$ and $\kappa_c^2\leq\kappa_b$, the Acc-VR-DIGing algorithm requires the time of $\bO((\kappa_c^2+\kappa_c\sqrt{\kappa_b})\log\frac{1}{\epsilon})=\bO(\kappa_c\sqrt{\kappa_b}\log\frac{1}{\epsilon})$ communication rounds and $\bO((\sqrt{n\kappa_s}+n)\log\frac{1}{\epsilon})$ stochastic gradient evaluations  to find $\z^k$ such that $\E_{\xi^k}[\|\z^k-\x^*\|^2]\leq\epsilon$.
\end{corollary}
\begin{remark}\label{acc-remark2}
$\vspace*{-0.4cm}\linebreak$
\begin{enumerate}
\item From Table \ref{table-comp}, we see that the communication complexity and stochastic gradient computation complexity of Acc-VR-EXTRA are both optimal under the restrictions of $2\kappa_c\leq\frac{nL_f}{\overline L_f}$ and $\kappa_c\leq\kappa_b$. Similarly, the stochastic gradient computation complexity of Acc-VR-DIGing is also optimal. However, its communication complexity is worse than the corresponding lower bound by the $\bO(\sqrt{\kappa_c})$ factor.
\item Running Acc-VR-EXTRA and Acc-VR-DIGing with $mn$ samples needs the time of $\bO((\sqrt{n\kappa_s}+n)\log\frac{1}{\epsilon})$ stochastic gradient evaluations, which is the same as that of running the single-machine Katyusha with $n$ samples when we ignore the communication time. On the other hand, when we run the single-machine Katyusha with $mn$ samples, the required time increases to $\bO((\sqrt{mn\kappa_s}+mn)\log\frac{1}{\epsilon})$. Since acceleration takes effect only when $\kappa_s\gg mn$, the parallelism speeds up Katyusha by the $\sqrt{m}$ factor. On the other hand, when $\kappa_s\leq n$, the linear speedup is achieved.
\end{enumerate}
\end{remark}

At last, we compare VR-EXTRA and VR-DIGing with Acc-VR-EXTRA and Acc-VR-DIGing. Both the accelerated methods and non-accelerated methods have their own advantages.
\begin{enumerate}
\item The accelerated methods need less stochastic gradient computation evaluations than the non-accelerated methods when $\kappa_s>n$ and $\kappa\leq\frac{nL_f}{\overline L_f}$. In this case, acceleration takes effect. Otherwise, the accelerated methods have no advantage over non-accelerated methods on the computation cost. Moreover, the non-accelerated methods have the superiority of simple implementation.
\item The accelerated methods need less communication rounds than their non-accelerated counterparts when $\kappa\leq\kappa_b$. When dealing with large-scale distributed data in machine learning, we may expect that computation often dominates the total running time. Otherwise, the full batch accelerated decentralized methods such as APAPC \citep{richtaric-2020-dist} may be a better choice.
\end{enumerate}

\subsection{Chebyshev Acceleration}

In this section we remove the restrictions on the size of $\kappa_c$ in Corollaries \ref{corollary-extra} and \ref{corollary-diging}, which come from matrix $\U$, as shown in (\ref{UVbound}). To make $\kappa$ small, our goal is to construct a new matrix $\hat\U$ by $\U$ such that $\mbox{Ker}(\hat\U)=\mbox{Span}(\1)$ and $\|\hat\U\blambda\|^2\geq\frac{1}{c}\|\blambda\|^2$ for all $\blambda\in\mbox{Span}(\hat\U)$, where $c$ is a much smaller constant than $\kappa$. Moreover, the construction procedure should not take more than $\bO(\sqrt{\kappa_c})$ time. Then  we only need to replace $\U$ and $\V$ by $\hat\U$ and some matrix $\hat\V$ in algorithm (\ref{acc-s1})-(\ref{acc-s6}), where $\hat\U$ and $\hat\V$ should satisfy the relations in (\ref{UVbound}). We follow \citep{dasent} to use Chebyshev acceleration to construct $\hat\U$, which is a common acceleration scheme to minimize $c$.

\subsubsection{Review of Chebyshev Acceleration}
We first give a brief description of Chebyshev acceleration, which was first used to accelerate distributed optimization in \citep{dasent}. We first introduce the Chebyshev polynomials defined as $T_0(x)=1$, $T_1(x)=x$, and $T_{k+1}(x)=2xT_k(x)-T_{k-1}(x)$ for all $k\geq 1$. Given a positive semidefinite symmetric matrix $L\in\R^{m\times m}$ such that $\mbox{Ker}(L)=\mbox{Span}(\1)$, denote $\lambda_1(L)\geq\lambda_2(L)\geq\cdots\geq \lambda_{m-1}(L)>\lambda_m(L)=0$ as the eigenvalues of $L$. Following the notations in \citep{dasent}, we define $\gamma(L)=\frac{\lambda_{m-1}(L)}{\lambda_1(L)}$, $c_1=\frac{1-\sqrt{\gamma(L)}}{1+\sqrt{\gamma(L)}}$, $c_2=\frac{1+\gamma(L)}{1-\gamma(L)}$, and $c_3=\frac{2}{\lambda_1(L)+\lambda_{m-1}(L)}$. Then  $c_3L$ has the spectrum in $[1-c_2^{-1},1+c_2^{-1}]$. For any polynomial $p_t(x)$ of degree at most $t$, Theorem 6.1 in \citep{Auzinger} tells us that the solution of the following problem
\begin{equation}\notag
\min_{p_t:p_t(0)=0}\max_{x\in[1-c_2^{-1},1+c_2^{-1}]} |p_t(x)-1|
\end{equation}
is
\begin{equation}
P_t(x)=1-\frac{T_t(c_2(1-x))}{T_t(c_2)}.\notag
\end{equation}
Moreover, from Corollary 6.1 in \citep{Auzinger}, we have
\begin{equation}\notag
\min_{p_t:p_t(0)=0}\max_{x\in[1-c_2^{-1},1+c_2^{-1}]} |p_t(x)-1|=\frac{2c_1^t}{1+c_1^{2t}},
\end{equation}
which means that the spectrum of $P_t(c_3L)$ lies in $\left[1-\frac{2c_1^t}{1+c_1^{2t}},1+\frac{2c_1^t}{1+c_1^{2t}}\right]$. For the particular choice of $t=\frac{3}{\sqrt{\gamma(L)}}$, it can be checked that $c_1^t\leq \Big(\left(1-\sqrt{\gamma(L)}\right)^{1/\sqrt{\gamma(L)}}\Big)^3\leq e^{-3}$, which further leads to $\frac{2c_1^t}{1+c_1^{2t}}\leq \frac{2}{e^3+e^{-3}}\leq 0.1$. Thus the spectrum of $P_t(c_3L)$ is a subinterval of $[0.9,1.1]$. On the other hand, it can be checked that $P_t(c_3L)$ is a gossip matrix satisfying $\mbox{Ker}(P_t(c_3L))=\mbox{Span}(\1)$ \citep{dasent}. In practice, we can compute the operation $P_t(c_3L)\x$ by the following procedure \citep{dasent}:

\begin{algorithmic}
   \STATE Input: $\x$,
   \STATE Initialize: $a^0=1$, $a^1=c_2$, $\z^0=\x$, $\z^1=c_2(\I-c_3L)\x$,
   \FOR{$s=1,2,...,t-1$}
   \STATE $a^{s+1}=2c_2a^s-a^{s-1}$,
   \STATE $\z^{s+1}=2c_2(\I-c_3L)\z^s-\z^{s-1}$.
   \ENDFOR
   \STATE Output: $P_t(c_3L)\x=\x-\frac{\z^t}{a^t}$.
\end{algorithmic}

\subsubsection{Remove the Restrictions on the Size of $\kappa_c$}
For the particular choice of $\U=\V=\sqrt{\frac{\I-\W}{2}}$, define $\hat\U=\hat\V=\sqrt{\frac{1}{2.2}P_t(c_3\U^2)}$, where $c_3=\frac{2}{\lambda_1(\U^2)+\lambda_{n-1}(\U^2)}$. Then  we have $\|\hat\V\x\|^2\leq\frac{1}{2}\|\x\|^2$ and $\|\hat\U\blambda\|^2\geq\frac{0.9}{2.2}\|\blambda\|^2\geq\frac{1}{3}\|\blambda\|^2$ for all $\blambda\in\mbox{Span}(\hat\U)$ with $t=\frac{3}{\sqrt{\gamma(\U^2)}}=\frac{3}{\sqrt{1-\sigma_2(\W)}}=\bO(\sqrt{\kappa_c})$. So $\hat\U$ and $\hat\V$ satisfy the relations in (\ref{UVbound}), and replacing $\U$ and $\V$ by $\hat\U$ and $\V$ does not destroy the proof of Theorem \ref{acc-theorem}. In the algorithm implementation, we only need to replace the operations $\U^2\z$ and $\V^2\z$ in (\ref{acc-extra2}) by $\frac{1}{2.2}P_t(c_3\U^2)\z$. Moreover, replacing $\kappa$ by the constant $3$ in Theorem \ref{acc-theorem}, we can expect that the assumptions on $\kappa$ always hold. Since we need $\bO(\sqrt{\kappa_c})$ time to construct $\hat\U$ at each iteration, so the communication complexity remains $\left(\sqrt{\kappa_b\kappa_c}\log\frac{1}{\epsilon}\right)$.

\begin{corollary}\label{acc-corollary1}
Suppose that Assumptions \ref{assumption_f} and \ref{assumption_w} hold with $\omega=0$. Under the parameter settings in Theorem \ref{acc-theorem} with $\kappa=3$, Acc-VR-EXTRA with Chebyshev acceleration (CA) requires the time of $\left(\sqrt{\kappa_b\kappa_c}\log\frac{1}{\epsilon}\right)$ communication rounds and $\bO((\sqrt{n\kappa_s}+n)\log\frac{1}{\epsilon})$ stochastic gradient evaluations to find $\z^k$ such that $\E_{\xi^k}[\|\z^k-\x^*\|^2]\leq\epsilon$.
\end{corollary}

For the particular choice of $\U=\I-\W$ and $\V=\sqrt{\I-\W^2}$, define $\hat\U=\frac{2-\sqrt{2}}{2.2}P_t(c_3\U)$, $\hat\W=\I-\hat\U$, and $\hat\V=\sqrt{\I-\hat\W^2}$, where $c_3=\frac{2}{\lambda_1(\U)+\lambda_{n-1}(\U)}$. Then  we have $\|\hat\V\x\|^2\leq\frac{1}{2}\|\x\|^2$ and $\|\hat\U\blambda\|^2\geq\frac{1}{20}\|\blambda\|^2$ for all $\blambda\in\mbox{Span}(\hat\U)$ with $t=\frac{3}{\sqrt{\gamma(U)}}=\bO(\sqrt{\kappa_c})$. Similar to the above analysis for the Acc-VR-EXTRA-CA method, we have the following complexity corollary. Note that since we replace $\kappa$ by the constant $20$ in Theorem \ref{acc-theorem}, and use the fact that the construction of $\hat\U$ needs $\bO(\sqrt{\kappa_c})$ time at each iteration, we can reduce the communication cost from $\bO(\kappa_c\sqrt{\kappa_b}\log\frac{1}{\epsilon})$ to $\bO(\sqrt{\kappa_b\kappa_c}\log\frac{1}{\epsilon})$.

\begin{corollary}\label{acc-corollary2}
Suppose that Assumptions \ref{assumption_f} and \ref{assumption_w} hold with $\omega=\frac{\sqrt{2}}{2}$. Under the parameter settings of Theorem \ref{acc-theorem} with $\kappa=20$, Acc-VR-DIGing-CA requires the time of $\bO(\sqrt{\kappa_b\kappa_c}\log\frac{1}{\epsilon})$ communication rounds and $\bO((\sqrt{n\kappa_s}+n)\log\frac{1}{\epsilon})$ stochastic gradient evaluations to find $\z^k$ such that $\E_{\xi^k}[\|\z^k-\x^*\|^2]\leq\epsilon$.
\end{corollary}

\begin{remark}
$\vspace*{-0.4cm}\linebreak$
\begin{enumerate}
\item From Corollaries \ref{acc-corollary1} and \ref{acc-corollary2}, we see that the restrictions on $\kappa_c$ in Corollaries \ref{corollary-extra} and \ref{corollary-diging} have been removed, and the communication complexity of Acc-VR-DIGing has been improved to be optimal by Chebyshev acceleration.
\item We can also combine Chebyshev acceleration with the non-accelerated VR-EXTRA and VR-DIGing, and give the $\bO\left(\left(\kappa_s+n\right)\log\frac{1}{\epsilon}\right)$ stochastic gradient computation complexity and the $\bO\left(\left(\kappa_b\sqrt{\kappa_c}\right)\log\frac{1}{\epsilon}\right)$ communication complexity, which are the same as those of DVR \citep{DVR}.
\end{enumerate}
\end{remark}

\section{Proof of Theorems}\label{sec-proof}

We prove Theorems \ref{non-acc-theorem} and \ref{acc-theorem} in this section. We first introduce some useful properties. For $L$-smooth and convex function $f(\x)$, we have \citep{Nesterov-2004}
\begin{eqnarray}
\begin{aligned}\label{smooth-cond}
f(\y)-f(\x)-\<\nabla f(\x),\y-\x\>\geq \frac{1}{2L}\|\nabla f(\y)-\nabla f(\x)\|^2.
\end{aligned}
\end{eqnarray}
Recall that $x^*$ is the optimal solution of problem (\ref{problem}), then $\x^*$ is also the optimal solution of the following linearly constrained convex problem
\begin{eqnarray}\label{problem3}
\min_{\x} f(\x),\quad \mbox{s.t.}\quad \U\x=0,
\end{eqnarray}
where $\U$ satisfies (\ref{UV-define}). Furthermore, there exists $\blambda^*\in\mbox{Span}(\U)$ such that
\begin{eqnarray}
\begin{aligned}\label{opt-cond}
\nabla f(\x^*)+\frac{1}{\alpha}\U\blambda^*=0.
\end{aligned}
\end{eqnarray}
The existence of $\blambda^*$ is proved in \citep[Lemma 3.1]{shi2015extra}. $\U\x^*=0$ and (\ref{opt-cond}) are the Karush-Kuhn-Tucker (KKT) optimality conditions of problem (\ref{problem3}).

\subsection{Non-accelerated VR-EXTRA and VR-DIGing}
We first give a classical property of the VR technique \citep{SVRG,xiao-2014-svrg-siam}.
\begin{lemma}\label{VR-lemma}
Suppose that Assumption \ref{assumption_f} and condition (\ref{UV-define}) hold. Then for algorithm (\ref{non-acc-s1})-(\ref{non-acc-s4}), we have
\begin{eqnarray}
\begin{aligned}\label{non-acc-vr}
\E_{\S^k}\big[\|\nabla^k-\nabla f(\x^k)\|^2\big]\leq&\frac{4\overline L_f}{b}\left(f(\x^k)-f(\x^*)+\frac{1}{\alpha}\<\blambda^*,\U\x^k\>\right)\\
&+\frac{4\overline L_f}{b}\left(f(\w^k)-f(\x^*)+\frac{1}{\alpha}\<\blambda^*,\U\w^k\>\right).
\end{aligned}
\end{eqnarray}
\end{lemma}
\begin{proof}
From the proof of Lemma D.2 in \citep{kat}, we have
\begin{eqnarray}\notag
\E_{\S_{(i)}^k}\big[\|\nabla_{(i)}^k-\nabla f_{(i)}(x_{(i)}^k)\|^2\big]\leq\frac{1}{b}\E_{j\sim \D_{(i)}}\left[\left\|\frac{1}{np_{(i),j}}\left(\nabla f_{(i),j}(x_{(i)}^k)-\nabla f_{(i),j}(w_{(i)}^k)\right)\right\|^2\right],
\end{eqnarray}
where $\D_{(i)}$ is defined in Algorithm \ref{extra}. Using identity $\|a+b\|^2\leq2\|b\|^2+2\|b\|^2$ and (\ref{smooth-cond}), from the proof of Lemma D.2 in \citep{kat} and recalling that $p_{(i),j}=\frac{L_{(i),j}}{n\overline L_{(i)}}$, we have
\begin{eqnarray}
\begin{aligned}\notag
&\E_{\S_{(i)}^k}\big[\|\nabla_{(i)}^k-\nabla f_{(i)}(x_{(i)}^k)\|^2\big]\\
~~&\leq\frac{2}{b}\E_{j\sim \D_{(i)}}\left[\left\|\frac{1}{np_{(i),j}}\left(\nabla f_{(i),j}(x_{(i)}^k)-\nabla f_{(i),j}(x^*)\right)\right\|^2\right]\\
~~&\quad+\frac{2}{b}\E_{j\sim \D_{(i)}}\left[\left\|\frac{1}{np_{(i),j}}\left(\nabla f_{(i),j}(w_{(i)}^k)-\nabla f_{(i),j}(x^*)\right)\right\|^2\right]\\
~~&\leq\frac{4\overline L_{(i)}}{b}\left(f_{(i)}(x_{(i)}^k)-f_{(i)}(x^*)-\<\nabla f_{(i)}(x^*),x_{(i)}^k-x^*\>\right)\\
~~&\quad+\frac{4\overline L_{(i)}}{b}\left(f_{(i)}(w_{(i)}^k)-f_{(i)}(x^*)-\<\nabla f_{(i)}(x^*),w_{(i)}^k-x^*\>\right).
\end{aligned}
\end{eqnarray}
From the convexity of $f_{(i)}(x)$, the definitions in (\ref{aggregate}), (\ref{def_f}), and (\ref{aggrgate-nabla}), and the fact that $\S^k_{(i)}$ and $\S^k_{(j)}$ are selected independently for all $i$ and $j$, we have
\begin{eqnarray}
\begin{aligned}\notag
\E_{\S^k}\big[\|\nabla^k-\nabla f(\x^k)\|^2\big]\leq&\frac{4\overline L_f}{b}\left(f(\x^k)-f(\x^*)-\<\nabla f(\x^*),\x^k-\x^*\>\right)\\
&+\frac{4\overline L_f}{b}\left(f(\w^k)-f(\x^*)-\<\nabla f(\x^*),\w^k-\x^*\>\right).
\end{aligned}
\end{eqnarray}
From the optimality condition in (\ref{opt-cond}) and $\U\x^*=0$, we have the conclusion.
\end{proof}

The following property is also useful in the analysis of mini-batch VR methods.
\begin{lemma}
For algorithm (\ref{non-acc-s1})-(\ref{non-acc-s4}), we have
\begin{eqnarray}
\begin{aligned}\label{non-acc-nabla-y}
\E_{\S^k}\big[\nabla^k\big]=\nabla f(\x^k).
\end{aligned}
\end{eqnarray}
\end{lemma}
\begin{proof}
From the definition of $\nabla_{(i)}^k$ in (\ref{non-acc-s1}), and the fact that the elements in $\S_{(i)}^k$ are selected independently with replacement, we have
\begin{eqnarray}
\begin{aligned}\notag
\E_{\S_{(i)}^k}\big[\nabla_{(i)}^k\big]=&\E_{j\sim \D_{(i)}}\left[ \frac{1}{np_{(i),j}}\left(\nabla f_{(i),j}(x_{(i)}^k)-\nabla f_{(i),j}(w_{(i)}^k)\right)+\nabla f_{(i)}(w_{(i)}^k) \right]=\nabla f_{(i)}(x_{(i)}^k).
\end{aligned}
\end{eqnarray}
Using the definitions in (\ref{aggregate}) and (\ref{aggrgate-nabla}), we have the conclusion.
\end{proof}

The next lemma describes a progress in one iteration of (\ref{non-acc-s1})-(\ref{non-acc-s4}).
\begin{lemma}\label{non-acc-lemma1}
Suppose that Assumption \ref{assumption_f} and conditions (\ref{UV-define}) and (\ref{UVbound}) hold. Then  for algorithm (\ref{non-acc-s1})-(\ref{non-acc-s4}), we have
\begin{eqnarray}
\begin{aligned}\label{non-acc-one-iter}
&\E_{\S^k}\left[f(\x^{k+1})-f(\x^*)+\frac{1}{\alpha}\<\blambda^*,\U\x^{k+1}\>\right]\\
~~&\leq\left(\frac{1}{2\alpha}-\frac{\mu}{2}\right)\|\x^k-\x^*\|^2-\frac{1}{2\alpha}\E_{\S^k}\big[\|\x^{k+1}-\x^*\|^2\big]\\
~~&\quad+\frac{1}{2\alpha}\left( \|\blambda^k-\blambda^*\|^2-\E_{\S^k}\big[\|\blambda^{k+1}-\blambda^*\|^2\big] \right)\\
~~&\quad+\frac{1}{2\tau}\E_{\S^k}\big[\|\nabla f(\x^k)-\nabla^k\|^2\big]-\frac{1}{2\alpha}\|\V\x^k\|^2-\left(\frac{1}{4\alpha}-\frac{\tau+L_m}{2}\right)\E_{\S^k}\big[\|\x^{k+1}-\x^k\|^2\big],
\end{aligned}
\end{eqnarray}
for some $\tau>0$ and $L_m=\max\{L_f,\kappa\mu\}$.
\end{lemma}
\begin{proof}
From the $L_f$-smoothness of $f(\x)$ and the definition of $L_m$, we have
\begin{align}
f(\x^{k+1})\leq& f(\x^k)+\<\nabla f(\x^k),\x^{k+1}-\x^k\>+\frac{L_m}{2}\|\x^{k+1}-\x^k\|^2\notag\\
=& f(\x^k)+\<\nabla f(\x^k)-\nabla^k,\x^{k+1}-\x^k\>+\<\nabla^k,\x^{k+1}-\x^k\>+\frac{L_m}{2}\|\x^{k+1}-\x^k\|^2\notag\\
&\hspace*{-0.25cm}\begin{aligned}\label{non-acc-cont1}
\overset{a}\leq& f(\x^k)+\frac{1}{2\tau}\|\nabla f(\x^k)-\nabla^k\|^2+\frac{\tau+L_m}{2}\|\x^{k+1}-\x^k\|^2\\
&+\<\nabla^k,\x^{k+1}-\x^*\>+\<\nabla^k,\x^*-\x^k\>,
\end{aligned}
\end{align}
where we use Young's inequality in $\overset{a}\leq$. Since
\begin{eqnarray}
&&\nabla^k=\frac{1}{\alpha}(\x^k-\x^{k+1})-\frac{1}{\alpha}\U\blambda^k-\frac{1}{\alpha}\V^2\x^k\label{non-acc-nabla},\\
&&\blambda^{k+1}=\blambda^k+\U\x^{k+1}\label{non-acc-lambda},
\end{eqnarray}
from (\ref{non-acc-s2}) and (\ref{non-acc-s3}), we have
\begin{eqnarray}
&&\hspace*{-2cm}\<\nabla^k,\x^{k+1}-\x^*\>\notag\\
&&\hspace*{-2cm}~\overset{b}=\frac{1}{\alpha}\<\x^k-\x^{k+1},\x^{k+1}-\x^*\>-\frac{1}{\alpha}\<\blambda^k,\U\x^{k+1}\>-\frac{1}{\alpha}\<\V\x^k,\V\x^{k+1}\>\notag\\
&&\hspace*{-2cm}~\overset{c}=\frac{1}{\alpha}\<\x^k-\x^{k+1},\x^{k+1}-\x^*\>-\frac{1}{\alpha}\<\blambda^*,\U\x^{k+1}\>\notag\\
&&\hspace*{-2cm}~\quad-\frac{1}{\alpha}\<\blambda^k-\blambda^*,\blambda^{k+1}-\blambda^k\>-\frac{1}{\alpha}\<\V\x^k,\V\x^{k+1}\>\notag\\
&&\hspace*{-2cm}~=\frac{1}{2\alpha}\left( \|\x^k-\x^*\|^2-\|\x^{k+1}-\x^*\|^2-\|\x^{k+1}-\x^k\|^2 \right)\notag\\
&&\hspace*{-2cm}~\quad+\frac{1}{2\alpha}\left( \|\blambda^k-\blambda^*\|^2-\|\blambda^{k+1}-\blambda^*\|^2+\|\blambda^{k+1}-\blambda^k\|^2 \right)\notag\\
&&\hspace*{-2cm}~\quad-\frac{1}{2\alpha}\left(\|\V\x^k\|^2+\|\V\x^{k+1}\|^2-\|\V\x^k-\V\x^{k+1}\|^2\right)-\frac{1}{\alpha}\<\blambda^*,\U\x^{k+1}\>\notag\\
&&\hspace*{-2cm}\begin{aligned}\label{non-acc-cont4}
&~\overset{d}\leq\frac{1}{2\alpha}\left( \|\x^k-\x^*\|^2-\|\x^{k+1}-\x^*\|^2\right)+\frac{1}{2\alpha}\left( \|\blambda^k-\blambda^*\|^2-\|\blambda^{k+1}-\blambda^*\|^2 \right)\\
&~\quad-\frac{1}{2\alpha}\|\V\x^k\|^2-\frac{1}{4\alpha}\|\x^{k+1}-\x^k\|^2-\frac{1}{\alpha}\<\blambda^*,\U\x^{k+1}\>,
\end{aligned}
\end{eqnarray}
where we use $\U\x^*=0$, $\V\x^*=0$, and the symmetry of $\U$ and $\V$ in $\overset{b}=$, (\ref{non-acc-lambda}) in $\overset{c}=$, $\|\blambda^{k+1}-\blambda^k\|^2=\|\U\x^{k+1}\|^2\leq\|\V\x^{k+1}\|^2$ and $\|\V(\x^{k+1}-\x^k)\|^2\leq\frac{1}{2}\|\x^{k+1}-\x^k\|^2$ in $\overset{d}\leq$. On the other hand, from (\ref{non-acc-nabla-y}) and the strong convexity of $f(\x)$, we have
\begin{eqnarray}
\E_{\S^k}\left[\<\nabla^k,\x^*-\x^k\>\right]=\<\nabla f(\x^k),\x^*-\x^k\>\leq f(\x^*)-f(\x^k)-\frac{\mu}{2}\|\x^k-\x^*\|^2.\label{non-acc-cont5}
\end{eqnarray}
Plugging (\ref{non-acc-cont4}) and (\ref{non-acc-cont5}) into (\ref{non-acc-cont1}), we have
\begin{eqnarray}
\begin{aligned}\notag
&\E_{\S^k}\big[f(\x^{k+1})\big]\\
~~&\leq f(\x^*)-\frac{\mu}{2}\|\x^k-\x^*\|^2+\frac{1}{2\tau}\E_{\S^k}\big[\|\nabla f(\x^k)-\nabla^k\|^2\big]\\
~~&\quad+\frac{1}{2\alpha}\left( \|\x^k-\x^*\|^2-\E_{\S^k}\big[\|\x^{k+1}-\x^*\|^2\big]\right)+\frac{1}{2\alpha}\left( \|\blambda^k-\blambda^*\|^2-\E_{\S^k}\big[\|\blambda^{k+1}-\blambda^*\|^2\big] \right)\\
~~&\quad-\frac{1}{2\alpha}\|\V\x^k\|^2-\left(\frac{1}{4\alpha}-\frac{\tau+L_m}{2}\right)\E_{\S^k}\big[\|\x^{k+1}-\x^k\|^2\big]-\frac{1}{\alpha}\E_{\S^k}\left[\<\blambda^*,\U\x^{k+1}\>\right].
\end{aligned}
\end{eqnarray}
Rearranging the terms, we have the conclusion.
\end{proof}

To prove the linear convergence, we should make the constant before $\|\blambda^k-\blambda^*\|^2$ in (\ref{non-acc-one-iter}) smaller than that before $\|\blambda^{k+1}-\blambda^*\|^2$, which is established in the next lemma.
\begin{lemma}\label{non-acc-lemma2}
Suppose that Assumption \ref{assumption_f} and conditions  (\ref{UV-define}) and (\ref{UVbound}) hold. Let $\alpha=\frac{1}{28L_m}$ and $\blambda^0=0$. Then  for algorithm (\ref{non-acc-s1})-(\ref{non-acc-s4}), we have
\begin{eqnarray}
\begin{aligned}\label{non-acc-cont2}
&\frac{1}{2}\E_{\S^k}\left[f(\x^{k+1})-f(\x^*)+\frac{1}{\alpha}\<\blambda^*,\U\x^{k+1}\>\right]\\
~~&\leq\left(\frac{1}{2\alpha}-\frac{\mu}{2}\right)\|\x^k-\x^*\|^2-\frac{1}{2\alpha}\E_{\S^k}\big[\|\x^{k+1}-\x^*\|^2\big]\\
~~&\quad+\left(\frac{1}{2\alpha}-\frac{1-\nu}{4\kappa L_m\alpha^2}\right)\|\blambda^k-\blambda^*\|^2-\frac{1}{2\alpha}\E_{\S^k}\big[\|\blambda^{k+1}-\blambda^*\|^2\big]\\
~~&\quad+\frac{\overline L_f}{6L_mb}\left(f(\x^k)-f(\x^*)+\frac{1}{\alpha}\<\blambda^*,\U\x^k\>\right)+\frac{\overline L_f}{6L_mb}\left(f(\w^k)-f(\x^*)+\frac{1}{\alpha}\<\blambda^*,\U\w^k\>\right),\hspace*{-1cm}
\end{aligned}
\end{eqnarray}
with $\nu=\frac{3140}{3141}$.
\end{lemma}
\begin{proof}
From the optimality condition in (\ref{opt-cond}) and the smoothness property in (\ref{smooth-cond}), we have
\begin{eqnarray}
&&\hspace*{-1cm}f(\x^{k+1})-f(\x^*)+\frac{1}{\alpha}\<\lambda^*,\U\x^{k+1}\>=f(\x^{k+1})-f(\x^*)-\<\nabla f(\x^*),\x^{k+1}-\x^*\>\notag\\
&&\hspace*{-1cm}~\geq\frac{1}{2L_m}\|\nabla f(\x^{k+1})-\nabla f(\x^*)\|^2=\frac{1}{2L_m\alpha^2}\|\alpha\nabla f(\x^{k+1})+\U\lambda^*\|^2\notag\\
&&\hspace*{-1cm}~\overset{a}=\frac{1}{2L_m\alpha^2}\hspace*{-0.05cm}\left\| \x^{k+1}\hspace*{-0.05cm}-\hspace*{-0.05cm}\x^k\hspace*{-0.05cm}+\hspace*{-0.05cm}\U(\blambda^k\hspace*{-0.05cm}-\hspace*{-0.05cm}\blambda^*)\hspace*{-0.05cm}+\hspace*{-0.05cm}\V^2\x^k\hspace*{-0.05cm}+\hspace*{-0.05cm}\alpha\nabla^k\hspace*{-0.05cm}-\hspace*{-0.05cm}\alpha\nabla f(\x^k)\hspace*{-0.05cm}+\hspace*{-0.05cm}\alpha\nabla f(\x^k)\hspace*{-0.05cm}-\hspace*{-0.05cm}\alpha\nabla f(\x^{k+1}) \right\|^2\hspace*{-0.05cm}\notag\\
&&\hspace*{-1cm}~\overset{b}\geq\frac{1-\nu}{2L_m\alpha^2}\|\U(\blambda^k-\blambda^*)\|^2-\frac{1}{2L_m\alpha^2}\left(\frac{1}{\nu}-1\right)\left\| \x^{k+1}-\x^k+\V^2\x^k\right.\notag\\
&&\hspace*{-1cm}~\quad\hspace*{2.5cm}\left.+\alpha\nabla^k-\alpha\nabla f(\x^k)+\alpha\nabla f(\x^k)-\alpha\nabla f(\x^{k+1}) \right\|^2\notag\\
&&\hspace*{-1cm}\begin{aligned}\label{non-acc-cont6}
~&\overset{c}\geq\frac{1-\nu}{2\kappa L_m\alpha^2}\|\blambda^k-\blambda^*\|^2-\left(\frac{2}{L_m\alpha^2}+2L_m\right)\left(\frac{1}{\nu}-1\right)\|\x^{k+1}-\x^k\|^2\\
~&\quad-\frac{2}{L_m}\left(\frac{1}{\nu}-1\right)\|\nabla^k-\nabla f(\x^k)\|^2-\frac{2}{L_m\alpha^2}\left(\frac{1}{\nu}-1\right)\|\V\x^k\|^2,
\end{aligned}
\end{eqnarray}
where we use (\ref{non-acc-nabla}) in $\overset{a}=$, $\|a-b\|^2\geq (1-\nu)\|a\|^2-(\frac{1}{\nu}-1)\|b\|^2$ in $\overset{b}\geq$ for some $0<\nu<1$, (\ref{UVbound}), $\|\sum_{i=1}^n a_i\|^2\leq n\sum_{i=1}^n\|a_i\|^2$, the $L_f$-smoothness of $f(\x)$, and $\|\V^2\x^k\|^2\leq \|\V\x^k\|^2$ in $\overset{c}\geq$, where the requirement of $\blambda^k-\blambda^*\in\mbox{Span}(\U)$ in (\ref{UVbound}) holds since $\blambda^0\in\mbox{Span}(\U)$, the update in (\ref{non-acc-s3}), and $\blambda^*\in\mbox{Span}(\U)$. Dividing both sides of (\ref{non-acc-cont6}) by 2 and plugging it into (\ref{non-acc-one-iter}), we have
\begin{eqnarray}
\begin{aligned}\notag
&\frac{1}{2}\E_{\S^k}\left[f(\x^{k+1})-f(\x^*)+\frac{1}{\alpha}\<\blambda^*,\U\x^{k+1}\>\right]\\
~~&\leq\left(\frac{1}{2\alpha}-\frac{\mu}{2}\right)\|\x^k-\x^*\|^2-\frac{1}{2\alpha}\E_{\S^k}\big[\|\x^{k+1}-\x^*\|^2\big]\\
~~&\quad+\left(\frac{1}{2\alpha}-\frac{1-\nu}{4\kappa L_m\alpha^2}\right)\|\blambda^k-\blambda^*\|^2-\frac{1}{2\alpha}\E_{\S^k}\big[\|\blambda^{k+1}-\blambda^*\|^2\big]\\
~~&\quad+\left(\frac{1}{2\tau}+\frac{1}{L_m}\left(\frac{1}{\nu}-1\right)\right)\E_{\S^k}\big[\|\nabla f(\x^k)-\nabla^k\|^2\big]-\left(\frac{1}{2\alpha}-\frac{1}{L_m\alpha^2}\left(\frac{1}{\nu}-1\right)\right)\|\V\x^k\|^2\\
~~&\quad-\left(\frac{1}{4\alpha}-\frac{\tau+L_m}{2}-\left(\frac{1}{L_m\alpha^2}+L_m\right)\left(\frac{1}{\nu}-1\right)\right)\E_{\S^k}\big[\|\x^k-\x^{k+1}\|^2\big].
\end{aligned}
\end{eqnarray}
Letting $\tau=12.5L_m$, $\nu=\frac{3140}{3141}$, and $\alpha=\frac{1}{28L_m}$, such that $\frac{1}{2\alpha}-\frac{1}{L_m\alpha^2}(\frac{1}{\nu}-1)\geq 0$, $\frac{1}{4\alpha}-\frac{\tau+L_m}{2}-(\frac{1}{L_m\alpha^2}+L_m)(\frac{1}{\nu}-1)\geq 0$, and $\frac{1}{2\tau}+\frac{1}{L_m}(\frac{1}{\nu}-1)\leq\frac{1}{24L_m}$, and using (\ref{non-acc-vr}), we have the conclusion.
\end{proof}

Now, we are ready to prove Theorem \ref{non-acc-theorem}.
\begin{proof}
From step (\ref{non-acc-s4}), we have
\begin{eqnarray}
\begin{aligned}\notag
&\E_{w_{(i)}^{k+1}}\left[f_{(i)}(w_{(i)}^{k+1})-\<\nabla f_{(i)}(x^*),w_{(i)}^{k+1}-x^*\>\right]\\
~~&=\frac{b}{n}\left(f_{(i)}(x_{(i)}^k)-\<\nabla f_{(i)}(x^*),x_{(i)}^k-x^*\>\right)+\left(1-\frac{b}{n}\right)\left(f_{(i)}(w_{(i)}^k)-\<\nabla f_{(i)}(x^*),w_{(i)}^k-x^*\>\right).
\end{aligned}
\end{eqnarray}
From the definitions in (\ref{aggregate}), the optimality condition in (\ref{opt-cond}), and the fact that each $w_{(i)}^{k+1}$ is computed independently at each node, we further have
\begin{eqnarray}
\begin{aligned}\label{non-acc-cont3}
&\E_{\w^{k+1}}\left[f(\w^{k+1})-f(\x^*)+\frac{1}{\alpha}\<\blambda^*,\U\w^{k+1}\>\right]\\
~~&=\frac{b}{n}\left(f(\x^k)-f(\x^*)+\frac{1}{\alpha}\<\blambda^*,\U\x^k\>\right)+\left(1-\frac{b}{n}\right)\left(f(\w^k)-f(\x^*)+\frac{1}{\alpha}\<\blambda^*,\U\w^k\>\right).
\end{aligned}
\end{eqnarray}
Multiplying both sides of (\ref{non-acc-cont3}) by $\frac{n}{b}(\frac{1}{2}-\frac{b}{10n}-\frac{\overline L_f}{6L_mb})$ and adding it to (\ref{non-acc-cont2}), taking expectation with respect to $\xi^k$, from the easy-to-identity equation $\frac{n}{b}(\frac{1}{2}-\frac{b}{10n}-\frac{\overline L_f}{6L_mb})(1-\frac{b}{n})+\frac{\overline L_f}{6L_mb}\leq \frac{n}{b}(\frac{1}{2}-\frac{b}{10n}-\frac{\overline L_f}{6L_mb})(1-\frac{b}{10n})$ under the condition $\frac{\overline L_f}{L_mb}\leq 1$, we have
\begin{eqnarray}
\begin{aligned}\notag
&\frac{1}{2}\E_{\xi^{k+1}}\left[f(\x^{k+1})-f(\x^*)+\frac{1}{\alpha}\<\blambda^*,\U\x^{k+1}\>\right]\\
~~&\quad+\frac{n}{b}\left(\frac{1}{2}-\frac{b}{10n}-\frac{\overline L_f}{6L_mb}\right)\E_{\xi^{k+1}}\left[f(\w^{k+1})-f(\x^*)+\frac{1}{\alpha}\<\blambda^*,\U\w^{k+1}\>\right]\\
~~&\quad+\frac{1}{2\alpha}\E_{\xi^{k+1}}\big[\|\x^{k+1}-\x^*\|^2\big]+\frac{1}{2\alpha}\E_{\xi^{k+1}}\big[\|\blambda^{k+1}-\blambda^*\|^2\big]\\
~~&\leq\left(\frac{1}{2}-\frac{b}{10n}\right)\E_{\xi^k}\left[f(\x^k)-f(\x^*)+\frac{1}{\alpha}\<\blambda^*,\U\x^k\>\right]\\
~~&\quad+\frac{n}{b}\left(\frac{1}{2}-\frac{b}{10n}-\frac{\overline L_f}{6L_mb}\right)\left(1-\frac{b}{10n}\right)\E_{\xi^k}\left[f(\w^k)-f(\x^*)+\frac{1}{\alpha}\<\blambda^*,\U\w^k\>\right]\\
~~&\quad+\left(\frac{1}{2\alpha}-\frac{\mu}{2}\right)\E_{\xi^k}\big[\|\x^k-\x^*\|^2\big]+\left(\frac{1}{2\alpha}-\frac{1-\nu}{4\kappa L_m\alpha^2}\right)\E_{\xi^k}\big[\|\blambda^k-\blambda^*\|^2\big]
\end{aligned}
\end{eqnarray}
\begin{eqnarray}
\hspace*{-1.5cm}\begin{aligned}\notag
~~&\overset{a}\leq\left\{\frac{1}{2}\E_{\xi^k}\left[f(\x^k)-f(\x^*)+\frac{1}{\alpha}\<\blambda^*,\U\x^k\>\right]\right.\\
~~&\qquad\left.+\frac{n}{b}\left(\frac{1}{2}-\frac{b}{10n}-\frac{\overline L_f}{6L_mb}\right)\E_{\xi^k}\left[f(\w^k)-f(\x^*)+\frac{1}{\alpha}\<\blambda^*,\U\w^k\>\right]\right.\\
~~&\qquad\left.+\frac{1}{2\alpha}\E_{\xi^k}\big[\|\x^k-\x^*\|^2\big]+\frac{1}{2\alpha}\E_{\xi^k}\big[\|\blambda^k-\blambda^*\|^2\big]\right\}\\
~~&\quad\times\max\left\{1-\frac{b}{5n},1-\frac{b}{10n},1-\alpha\mu,1-\frac{1-\nu}{2\kappa L_m\alpha}\right\},
\end{aligned}
\end{eqnarray}
where we use the fact $f(\x)+\frac{1}{\alpha}\<\blambda^*,\U\x\>\geq f(\x^*)+\frac{1}{\alpha}\<\blambda^*,\U\x^*\>$ for any $\x$, and $\frac{1}{2}-\frac{b}{10n}-\frac{\overline L_f}{6L_mb}>0$ in $\overset{a}\leq$. From the setting of $\alpha=\frac{1}{28L_m}$, we know the algorithm needs $\bO((\frac{n}{b}+\frac{L_m}{\mu}+\kappa)\log\frac{1}{\epsilon})\overset{b}=\bO((\frac{n}{b}+\frac{L_f}{\mu}+\kappa)\log\frac{1}{\epsilon})$ iterations to find $\x^k$ such that $\E_{\xi^k}\big[\|\x^k-\x^*\|^2\big]\leq\epsilon$, where we use $L_m=\max\{L_f,\kappa\mu\}$ in $\overset{b}=$. Recall that each iteration requires the time of one communication round and $b$ stochastic gradient evaluations in average, see Remark \ref{remark3}. So the communication complexity is $\bO((\frac{n}{b}+\frac{L_f}{\mu}+\kappa)\log\frac{1}{\epsilon})$, and the stochastic gradient computation complexity is $\bO((n+\frac{bL_f}{\mu}+b\kappa)\log\frac{1}{\epsilon})=\bO((n+\frac{bL_m}{\mu})\log\frac{1}{\epsilon})$.

Case 1. If $\kappa\leq \max\{\frac{\overline L_f}{\mu},n\}$, we have $\kappa\mu\leq\max\{\overline L_f,n\mu\}$ and $b=\frac{\max\{\overline L_f,n\mu\}}{\max\{L_f,\kappa\mu\}}\geq\frac{\max\{\overline L_f,n\mu\}}{\max\{\overline L_f,n\mu\}}=1$, where we use $L_f\leq\overline L_f\leq\max\{\overline L_f,n\mu\}$. We also have $b=\frac{\max\{\overline L_f,n\mu\}}{\max\{L_f,\kappa\mu\}}\leq \frac{\max\{\overline L_f,n\mu\}}{L_f}\leq n$, where we use $\overline L_f\leq nL_f$ given in (\ref{L-relation}) and $L_f\geq\mu$. This verifies that the setting of $b$ is meaningful. On the other hand, since $b=\frac{\max\{\overline L_f,n\mu\}}{L_m}$, we have $\frac{\overline L_f}{L_mb}=\frac{\overline L_f}{\max\{\overline L_f,n\mu\}}\leq 1$ and $\frac{n}{b}=\frac{nL_m}{\max\{\overline L_f,n\mu\}}\leq\frac{nL_m}{n\mu}=\max\{\frac{L_f}{\mu},\kappa\}$. Then  the communication complexity is $\bO((\frac{L_f}{\mu}+\kappa)\log\frac{1}{\epsilon})$, and the stochastic gradient computation complexity is $\bO((\frac{\overline L_f}{\mu}+n)\log\frac{1}{\epsilon})$, where we use $\frac{bL_m}{\mu}=\frac{\max\{\overline L_f,n\mu\}}{\mu}\leq\frac{\overline L_f}{\mu}+n$.

If we choose $b\geq \frac{\max\{\overline L_f,n\mu\}}{L_m}$, similar to the above analysis, we also have $\frac{\overline L_f}{L_mb}\leq 1$ and $\frac{n}{b}\leq\max\{\frac{L_f}{\mu},\kappa\}$. So the communication complexity remains unchanged, but the stochastic gradient computation complexity increases. This verifies Remark \ref{remark2}(4). Specially, if we let $b=\frac{\max\{\overline L_f,n\mu\}}{\max\{L_f,\kappa\mu\}}\frac{\tau(\kappa_b+\kappa)}{\kappa_s+n}$, we have $\frac{bL_m}{\mu}=\tau(\kappa_b+\kappa)$.

Case 2. If $\kappa\geq \max\{\frac{\overline L_f}{\mu},n\}$, letting $b=1$, we have $\frac{\overline L_f}{L_mb}=\frac{\overline L_f}{L_m}\leq\frac{\overline L_f}{\kappa\mu}\leq 1$. The communication complexity and stochastic gradient computation complexity are both $\bO((\frac{L_f}{\mu}+n+\kappa)\log\frac{1}{\epsilon})=\bO((\frac{L_f}{\mu}+\kappa)\log\frac{1}{\epsilon})$, where we use $\kappa\geq n$.
\end{proof}

Now, we prove Theorem \ref{non-acc-theorem2}, which reduces the stochastic gradient computation complexity from $\bO((\frac{L_f}{\mu}+\kappa)\log\frac{1}{\epsilon})$ to $\bO((\frac{\overline L_f}{\mu}+n)\log\frac{1}{\epsilon})$ in the case of $\kappa\geq \max\{\frac{\overline L_f}{\mu},n\}$ by the zero-sample strategy.
\begin{proof}
From the definitions in (\ref{new-L2}), it can be easily checked that $\overline L_{(i)}'=\frac{\sum_{j=1}^{n'}L_{(i),j}}{n'}=\frac{\sum_{j=1}^{n}L_{(i),j}+\sum_{j=n+1}^{n'}L_{(i),j}}{n'}=\frac{n\overline L_{(i)}+n\mu n'-n\overline L_{(i)}}{n'}=n\mu$, $\overline L_f'=\overline L_{(i)}'=n\mu$, and $b=\frac{\max\{\overline L_f',n'\mu'\}}{\max\{L_f',\kappa\mu'\}}=\frac{n\mu}{\max\{\frac{nL_f}{\kappa},n\mu\}}=\frac{n\mu}{n\mu}=1$, where we use $\kappa\geq\kappa_s\geq\frac{L_f}{\mu}$.

Replacing $n$, $L_f$, $\mu$, and $\overline L_f$ in the proof of Theorem \ref{non-acc-theorem} by $n'$, $L_f'$, $\mu'$, and $\overline L_f'$ given in (\ref{new-L2}), respectively, we know $L_m=\max\{L_f',\kappa\mu'\}=n\mu$ and $\frac{\overline L_f'}{L_mb}=\frac{n\mu}{n\mu}=1$. So the algorithm needs $\bO((\frac{n'}{b}+\frac{L_f'}{\mu'}+\kappa)\log\frac{1}{\epsilon})=\bO((\kappa+\frac{L_f}{\mu})\log\frac{1}{\epsilon})=\bO(\kappa\log\frac{1}{\epsilon})$ iterations to find $\x^k$ such that $\E_{\xi^k}\big[\|\x^k-\x^*\|^2\big]\leq\epsilon$. So the communication complexity and the stochastic gradient computation complexity are both $\bO(\kappa\log\frac{1}{\epsilon})$. Since we select the samples in $\big[1,n\big]$ with probability $\frac{\sum_{j=1}^n L_{(i),j}}{\sum_{j=1}^{n'} L_{(i),j} }=\frac{n\overline L_f}{n\mu\kappa}=\frac{\overline L_f}{\mu\kappa}$, and the zero samples do not spend the computation time, so the valid number of stochastic gradient evaluations is $\bO(\frac{\overline L_f}{\mu\kappa}\kappa\log\frac{1}{\epsilon})=\bO(\frac{\overline L_f}{\mu}\log\frac{1}{\epsilon})$. On the other hand, we compute the full batch gradient with probability $\frac{b}{n'}=\frac{1}{\kappa}$, which takes $\bO(n\frac{1}{\kappa}\kappa\log\frac{1}{\epsilon})=\bO(n\log\frac{1}{\epsilon})$ valid stochastic gradient evaluations in total. So the final valid stochastic gradient computation complexity is $\bO((\frac{\overline L_f}{\mu}+n)\log\frac{1}{\epsilon})$.

At last, we explain that the zero samples do not destroy the proof of Theorem \ref{non-acc-theorem}. For the zero sample $f_{(i),j}(x)=0$, we have $\nabla f_{(i),j}(x)=0$. So it also satisfies the convexity and $L_{(i),j}$-smooth property (\ref{smooth-cond}) even for positive $L_{(i),j}$. We can check that (\ref{non-acc-vr}) and (\ref{non-acc-nabla-y}) also hold. In the proofs of Lemmas \ref{non-acc-lemma1} and \ref{non-acc-lemma2}, we use the smoothness and strong convexity of $f_{(i)}'(x)$, as explained in Section \ref{sec-non-acc-vr2}, which also hold.
\end{proof}

\subsection{Accelerated VR-EXTRA and VR-DIGing}
From Lemma D.2 in \citep{kat} and similar to Lemma \ref{VR-lemma}, we have
\begin{eqnarray}
\begin{aligned}\label{acc-vr}
\qquad\qquad\E_{\S^k}\big[\|\nabla^k-\nabla f(\y^k)\|^2\big]\leq\frac{2\overline L_f}{b}\left(f(\w^k)-f(\y^k)-\<\nabla f(\y^k),\w^k-\y^k\>\right).
\end{aligned}
\end{eqnarray}
Similar to (\ref{non-acc-nabla-y}), we also have
\begin{eqnarray}
\begin{aligned}\label{acc-nabla-y}
&\E_{\S^k}\big[\nabla^k\big]=\nabla f(\y^k).
\end{aligned}
\end{eqnarray}

The following lemma is the counterpart of Lemma \ref{non-acc-lemma1}, which gives a progress in one iteration of procedure (\ref{acc-s1})-(\ref{acc-s6}).
\begin{lemma}
Suppose that Assumption \ref{assumption_f} and conditions (\ref{UV-define}) and (\ref{UVbound}) hold. Let $\theta_1+\theta_2\leq1$. Then  for algorithm (\ref{acc-s1})-(\ref{acc-s6}), we have
\begin{eqnarray}
\begin{aligned}\label{acc-cont8}
&\E_{\S^k}\left[f(\x^{k+1})-f(\x^*)+\frac{1}{\alpha}\<\blambda^*,\U\x^{k+1}\>\right]\\
~~&\leq(1-\theta_1-\theta_2)\left(f(\x^k)-f(\x^*)+\frac{1}{\alpha}\<\blambda^*,\U\x^k\>\right)\\
~~&\quad+\theta_2\left(f(\w^k)-f(\x^*)+\frac{1}{\alpha}\<\blambda^*,\U\w^k\>\right)\\
~~&\quad+\left(\frac{\overline L_f}{\tau b}-\theta_2\right)\left(f(\w^k)-f(\y^k)-\<\nabla f(\y^k),\w^k-\y^k\>\right)\\
~~&\quad+\frac{\theta_1^2}{2\alpha}\|\z^k-\x^*\|^2-\left(\frac{\theta_1^2}{2\alpha}+\frac{\mu\theta_1}{2}\right)\E_{\S^k}\big[\|\z^{k+1}-\x^*\|^2\big]\\
~~&\quad+\frac{1}{2\alpha}\|\blambda^k-\blambda^*\|^2-\frac{1}{2\alpha}\E_{\S^k}\big[\|\blambda^{k+1}-\blambda^*\|^2\big]\\
~~&\quad-\frac{\theta_1^2}{2\alpha}\|\V\z^k\|^2-\left(\frac{\theta_1^2}{4\alpha}-\frac{\tau\theta_1^2+L_f\theta_1^2}{2}\right)\E_{\S^k}\big[\|\z^{k+1}-\z^k\|^2\big]-\frac{\mu\theta_1}{2}\E_{\S^k}\big[\|\z^{k+1}-\y^k\|^2\big],
\end{aligned}
\end{eqnarray}
for some $\tau>0$.
\end{lemma}
\begin{proof}
From the $L_f$-smoothness of $f(\x)$, similar to (\ref{non-acc-cont1}), we have
\begin{align}
f(\x^{k+1})\leq& f(\y^k)+\<\nabla f(\y^k),\x^{k+1}-\y^k\>+\frac{L_f}{2}\|\x^{k+1}-\y^k\|^2\notag\\
\leq& f(\y^k)+\frac{1}{2\tau}\|\nabla f(\y^k)-\nabla^k\|^2+\frac{\tau+L_f}{2}\|\x^{k+1}-\y^k\|^2+\<\nabla^k,\x^{k+1}-\y^k\>\notag\\
&\hspace*{-0.25cm}\begin{aligned}\label{acc-cont1}
\overset{a}=& f(\y^k)+\frac{1}{2\tau}\|\nabla f(\y^k)-\nabla^k\|^2+\frac{\tau\theta_1^2+L_f\theta_1^2}{2}\|\z^{k+1}-\z^k\|^2\\
&+\theta_1\<\nabla^k,\z^{k+1}-\z^*\>+\theta_1\<\nabla^k,\z^*-\z^k\>.
\end{aligned}
\end{align}
where we use
\begin{equation}
\x^{k+1}-\y^k=\theta_1(\z^{k+1}-\z^k)\label{acc-cont6}
\end{equation}
in $\overset{a}=$, which comes from (\ref{acc-s5}). Since
\begin{eqnarray}
&&\nabla^k=\frac{\theta_1}{\alpha}(\z^k-\z^{k+1})+\mu(\y^k-\z^{k+1})-\frac{1}{\alpha}\U\blambda^k-\frac{\theta_1}{\alpha}\V^2\z^k\label{acc-cont4},\\
&&\blambda^{k+1}=\blambda^k+\theta_1\U\z^{k+1}\label{acc-lambda},
\end{eqnarray}
from (\ref{acc-s3}) and (\ref{acc-s4}), similar to (\ref{non-acc-cont4}), we have
\begin{eqnarray}
&&\hspace*{-2cm}\theta_1\<\nabla^k,\z^{k+1}-\x^*\>\notag\\
&&\hspace*{-2cm}~=\frac{\theta_1^2}{\alpha}\<\z^k-\z^{k+1},\z^{k+1}-\x^*\>+\mu\theta_1\<\y^k-\z^{k+1},\z^{k+1}-\x^*\>\notag\\
&&\hspace*{-2cm}~\quad-\frac{\theta_1}{\alpha}\<\blambda^k,\U\z^{k+1}\>-\frac{\theta_1^2}{\alpha}\<\V\z^k,\V\z^{k+1}\>\notag\\
&&\hspace*{-2cm}~\overset{b}=\frac{\theta_1^2}{\alpha}\<\z^k-\z^{k+1},\z^{k+1}-\x^*\>+\mu\theta_1\<\y^k-\z^{k+1},\z^{k+1}-\x^*\>\notag\\
&&\hspace*{-2cm}~\quad-\frac{\theta_1}{\alpha}\<\blambda^*,\U\z^{k+1}\>-\frac{1}{\alpha}\<\blambda^k-\blambda^*,\blambda^{k+1}-\blambda^k\>-\frac{\theta_1^2}{\alpha}\<\V\z^k,\V\z^{k+1}\>\notag\\
&&\hspace*{-2cm}~=\frac{\theta_1^2}{2\alpha}\left( \|\z^k-\x^*\|^2-\|\z^{k+1}-\x^*\|^2-\|\z^{k+1}-\z^k\|^2 \right)\notag\\
&&\hspace*{-2cm}~\quad + \frac{\mu\theta_1}{2}\left( \|\y^k-\x^*\|^2-\|\z^{k+1}-\x^*\|^2-\|\z^{k+1}-\y^k\|^2 \right)\notag\\
&&\hspace*{-2cm}~\quad +\frac{1}{2\alpha}\left( \|\blambda^k-\blambda^*\|^2-\|\blambda^{k+1}-\blambda^*\|^2+\|\blambda^{k+1}-\blambda^k\|^2 \right)\notag\\
&&\hspace*{-2cm}~\quad -\frac{\theta_1^2}{2\alpha}\left(\|\V\z^k\|^2+\|\V\z^{k+1}\|^2-\|\V\z^{k+1}-\V\z^k\|^2\right)-\frac{\theta_1}{\alpha}\<\blambda^*,\U\z^{k+1}\>\notag\\
&&\hspace*{-2cm}\begin{aligned}\label{acc-cont2}
~&\overset{c}\leq\frac{\theta_1^2}{2\alpha}\left( \|\z^k-\x^*\|^2-\|\z^{k+1}-\x^*\|^2\right) + \frac{\mu\theta_1}{2}\left( \|\y^k-\x^*\|^2-\|\z^{k+1}-\x^*\|^2 \right)\\
~&\quad+\frac{1}{2\alpha}\left( \|\blambda^k-\blambda^*\|^2-\|\blambda^{k+1}-\blambda^*\|^2 \right)-\frac{\theta_1^2}{2\alpha}\|\V\z^k\|^2-\frac{\theta_1^2}{4\alpha}\|\z^{k+1}-\z^k\|^2\\
~&\quad-\frac{\mu\theta_1}{2}\|\z^{k+1}-\y^k\|^2-\frac{\theta_1}{\alpha}\<\blambda^*,\U\z^{k+1}\>,
\end{aligned}
\end{eqnarray}
where we use (\ref{acc-lambda}) in $\overset{b}=$, $\|\blambda^{k+1}-\blambda^k\|^2=\|\theta_1\U\z^{k+1}\|^2\leq\theta_1^2\|\V\z^{k+1}\|^2$ and $\|\V(\z^{k+1}-\z^k)\|^2\leq\frac{1}{2}\|\z^{k+1}-\z^k\|^2$ in $\overset{c}\leq$. On the other hand, from (\ref{acc-nabla-y}), we have
\begin{eqnarray}
&&\hspace*{-1.5cm}\theta_1\E_{\S^k}\left[\<\nabla^k,\x^*-\z^k\>\right]=\theta_1\<\nabla f(\y^k),\x^*-\z^k\>\notag\\
&&\hspace*{-1.5cm}~\overset{d}=\<\nabla f(\y^k),\theta_1\x^*+\theta_2\w^k+(1-\theta_1-\theta_2)\x^k-\y^k\>\notag\\
&&\hspace*{-1.5cm}~=\theta_1\<\nabla f(\y^k),\x^*-\y^k\>+(1-\theta_1-\theta_2)\<\nabla f(\y^k),\x^k-\y^k\>+\theta_2\<\nabla f(\y^k),\w^k-\y^k\>\notag\\
&&\hspace*{-1.5cm}~\overset{e}\leq\theta_1\hspace*{-0.08cm}\left(\hspace*{-0.08cm}f(\x^*)\hspace*{-0.08cm}-\hspace*{-0.08cm}f(\y^k)\hspace*{-0.08cm}-\hspace*{-0.08cm}\frac{\mu}{2}\|\y^k\hspace*{-0.08cm}-\hspace*{-0.08cm}\x^*\|^2\right)\hspace*{-0.08cm}+\hspace*{-0.08cm}(1\hspace*{-0.08cm}-\hspace*{-0.08cm}\theta_1\hspace*{-0.08cm}-\hspace*{-0.08cm}\theta_2)(f(\x^k)\hspace*{-0.08cm}-\hspace*{-0.08cm}f(\y^k))\hspace*{-0.08cm}+\hspace*{-0.08cm}\theta_2\hspace*{-0.08cm}\<\nabla f(\y^k),\w^k\hspace*{-0.08cm}-\hspace*{-0.08cm}\y^k\>\hspace*{-1cm}\notag\\
&&\hspace*{-1.5cm}\begin{aligned}\label{acc-cont3}
~&= \theta_1f(\x^*)\hspace*{-0.06cm}+\hspace*{-0.06cm}(1\hspace*{-0.06cm}-\hspace*{-0.06cm}\theta_1\hspace*{-0.06cm}-\hspace*{-0.06cm}\theta_2)f(\x^k)\hspace*{-0.06cm}-\hspace*{-0.06cm}(1\hspace*{-0.06cm}-\hspace*{-0.06cm}\theta_2)f(\y^k)\hspace*{-0.06cm}-\hspace*{-0.06cm}\frac{\mu\theta_1}{2}\|\y^k\hspace*{-0.06cm}-\hspace*{-0.06cm}\x^*\|^2\hspace*{-0.06cm}+\hspace*{-0.06cm}\theta_2\<\nabla f(\y^k),\w^k\hspace*{-0.06cm}-\hspace*{-0.06cm}\y^k\>,\hspace*{-1cm}
\end{aligned}
\end{eqnarray}
where we use (\ref{acc-s1}) in $\overset{d}=$, and the strong convexity of $f(\x)$ in $\overset{e}=$. Plugging (\ref{acc-cont2}) and (\ref{acc-cont3}) into (\ref{acc-cont1}), and using (\ref{acc-vr}), we have
\begin{eqnarray}
\begin{aligned}\notag
&\E_{\S^k}\big[f(\x^{k+1})\big]\\
~~&\leq \theta_1f(\x^*)+(1-\theta_1-\theta_2)f(\x^k)+\theta_2f(\y^k)-\frac{\mu\theta_1}{2}\|\y^k-\x^*\|^2+\theta_2\<\nabla f(\y^k),\w^k-\y^k\>\\
~~&\quad+\frac{\overline L_f}{\tau b}\left(f(\w^k)-f(\y^k)-\<\nabla f(\y^k),\w^k-\y^k\>\right)\\
~~&\quad+\frac{\theta_1^2}{2\alpha}\left( \|\z^k-\x^*\|^2-\E_{\S^k}\big[\|\z^{k+1}-\x^*\|^2\big]\right) + \frac{\mu\theta_1}{2}\left( \|\y^k-\x^*\|^2-\E_{\S^k}\big[\|\z^{k+1}-\x^*\|^2\big] \right)\\
~~&\quad+\frac{1}{2\alpha}\left( \|\blambda^k-\blambda^*\|^2-\E_{\S^k}\big[\|\blambda^{k+1}-\blambda^*\|^2\big] \right)-\frac{\theta_1}{\alpha}\E_{\S^k}\left[\<\blambda^*,\U\z^{k+1}\>\right]\\
~~&\quad-\frac{\theta_1^2}{2\alpha}\|\V\z^k\|^2-\left(\frac{\theta_1^2}{4\alpha}-\frac{\tau\theta_1^2+L_f\theta_1^2}{2}\right)\E_{\S^k}\big[\|\z^{k+1}-\z^k\|^2\big]-\frac{\mu\theta_1}{2}\E_{\S^k}\big[\|\z^{k+1}-\y^k\|^2\big]
\end{aligned}
\end{eqnarray}
\begin{eqnarray}
\begin{aligned}\notag
~~&\overset{f}= \theta_1f(\x^*)+(1-\theta_1-\theta_2)f(\x^k)+\theta_2f(\w^k)\\
~~&\quad+\left(\frac{\overline L_f}{\tau b}-\theta_2\right)\left(f(\w^k)-f(\y^k)-\<\nabla f(\y^k),\w^k-\y^k\>\right)\\
~~&\quad+\frac{\theta_1^2}{2\alpha}\|\z^k-\x^*\|^2-\left(\frac{\theta_1^2}{2\alpha}+\frac{\mu\theta_1}{2}\right)\E_{\S^k}\big[\|\z^{k+1}-\x^*\|^2\big]\\
~~&\quad+\frac{1}{2\alpha}\left( \|\blambda^k-\blambda^*\|^2-\E_{\S^k}\big[\|\blambda^{k+1}-\blambda^*\|^2\big] \right)\\
~~&\quad-\frac{1}{\alpha}\E_{\S^k}\left[\<\blambda^*,\U\x^{k+1}-\theta_2\U\w^k-(1-\theta_1-\theta_2)\U\x^k\>\right]\\
~~&\quad-\frac{\theta_1^2}{2\alpha}\|\V\z^k\|^2-\left(\frac{\theta_1^2}{4\alpha}-\frac{\tau\theta_1^2+L_f\theta_1^2}{2}\right)\E_{\S^k}\big[\|\z^{k+1}-\z^k\|^2\big]-\frac{\mu\theta_1}{2}\E_{\S^k}\big[\|\z^{k+1}-\y^k\|^2\big],
\end{aligned}
\end{eqnarray}
where we use (\ref{acc-s1}) and (\ref{acc-s5}) in $\overset{f}=$. Rearranging the terms, we have the conclusion.
\end{proof}

Similar to Lemma \ref{non-acc-lemma2}, we establish the smaller constant before $\|\blambda^k-\blambda^*\|^2$ than that before $\|\blambda^{k+1}-\blambda^*\|^2$ in the next lemma.
\begin{lemma}
Suppose that Assumption \ref{assumption_f} and conditions (\ref{UV-define}) and (\ref{UVbound}) hold. Choose $b$ such that $\theta_2=\frac{\overline L_f}{2L_fb}\leq\frac{1}{2}$. Let $\theta_1\leq\frac{1}{2}$, $\alpha=\frac{1}{10L_f}$, and $\blambda^0=0$. Then  for algorithm (\ref{acc-s1})-(\ref{acc-s6}), we have
\begin{eqnarray}
\begin{aligned}\label{acc-cont9}
&\left(1-\frac{\theta_1}{2}\right)\E_{\S^k}\left[f(\x^{k+1})-f(\x^*)+\frac{1}{\alpha}\<\blambda^*,\U\x^{k+1}\>\right]\\
~~&\leq(1-\theta_1-\theta_2)\left(f(\x^k)-f(\x^*)+\frac{1}{\alpha}\<\blambda^*,\U\x^k\>\right)+\theta_2\left(f(\w^k)-f(\x^*)+\frac{1}{\alpha}\<\blambda^*,\U\w^k\>\right)\hspace*{-1cm}\\
~~&\quad+\frac{\theta_1^2}{2\alpha}\|\z^k-\x^*\|^2-\left(\frac{\theta_1^2}{2\alpha}+\frac{\mu\theta_1}{2}\right)\E_{\S^k}\big[\|\z^{k+1}-\x^*\|^2\big]\\
~~&\quad+\left(\frac{1}{2\alpha}-\frac{(1-\nu)\theta_1}{4\kappa L_f\alpha^2}\right)\|\blambda^k-\blambda^*\|^2-\frac{1}{2\alpha}\E_{\S^k}\big[\|\blambda^{k+1}-\blambda^*\|^2\big],
\end{aligned}
\end{eqnarray}
with $\nu=\frac{127}{128}$.
\end{lemma}
\begin{proof}
From (\ref{smooth-cond}) and (\ref{opt-cond}), similar to (\ref{non-acc-cont6}), we have
\begin{eqnarray}
&&\hspace*{-1.5cm}f(\x^{k+1})-f(\x^*)+\frac{1}{\alpha}\<\lambda^*,\U\x^{k+1}\>\geq\frac{1}{2L_f\alpha^2}\|\alpha\nabla f(\x^{k+1})+\U\lambda^*\|^2\notag\\
&&\hspace*{-1.5cm}~\overset{a}=\frac{1}{2L_f\alpha^2}\left\| \alpha\mu(\z^{k+1}-\y^k)+\theta_1(\z^{k+1}-\z^k)+\U(\blambda^k-\blambda^*)+\theta_1\V^2\z^k\right.\notag\\
&&\hspace*{-1.5cm}~\quad\hspace{1.35cm}\left.+\alpha\nabla^k-\alpha\nabla f(\y^k)+\alpha\nabla f(\y^k)-\alpha\nabla f(\x^{k+1}) \right\|^2\notag\\
&&\hspace*{-1.5cm}~\geq\frac{1-\nu}{2L_f\alpha^2}\|\U(\blambda^k-\blambda^*)\|^2-\frac{1}{2L_f\alpha^2}\left(\frac{1}{\nu}-1\right)\left\| \alpha\mu(\z^{k+1}-\y^k)+\theta_1(\z^{k+1}-\z^k)\right.\notag\\
&&\hspace*{-1.5cm}~\quad\hspace{2.8cm}\left.+\theta_1\V^2\z^k+\alpha\nabla^k-\alpha\nabla f(\y^k)+\alpha\nabla f(\y^k)-\alpha\nabla f(\x^{k+1}) \right\|^2\notag\\
&&\hspace*{-1.5cm}\begin{aligned}\label{acc-cont7}
~&\overset{b}\geq\frac{1-\nu}{2\kappa L_f\alpha^2}\|\blambda^k-\blambda^*\|^2-\frac{5\mu^2}{2L_f}\left(\frac{1}{\nu}-1\right)\|\z^{k+1}-\y^k\|^2-\frac{5\theta_1^2}{2L_f\alpha^2}\left(\frac{1}{\nu}-1\right)\|\V\z^k\|^2\\
~&\quad-\frac{5}{2L_f}\left(\frac{1}{\nu}-1\right)\|\nabla^k-\nabla f(\y^k)\|^2-\left(\frac{5\theta_1^2}{2L_f\alpha^2}+\frac{5L_f\theta_1^2}{2}\right)\left(\frac{1}{\nu}-1\right)\|\z^{k+1}-\z^k\|^2,
\end{aligned}
\end{eqnarray}
where we use (\ref{acc-cont4}) in $\overset{a}=$, (\ref{UVbound}), the $L_f$-smoothness of $f(\x)$, (\ref{acc-cont6}), and $\|\V^2\z^k\|^2\leq \|\V\z^k\|^2$ in $\overset{b}\geq$. Multiplying both sides of (\ref{acc-cont7}) by $\frac{\theta_1}{2}$ and plugging it into (\ref{acc-cont8}), using (\ref{acc-vr}), we have
\begin{eqnarray}
\hspace*{-1.6cm}\begin{aligned}\notag
&\left(1-\frac{\theta_1}{2}\right)\E_{\S^k}\left[f(\x^{k+1})-f(\x^*)+\frac{1}{\alpha}\<\blambda^*,\U\x^{k+1}\>\right]\\
~~&\leq(1-\theta_1-\theta_2)\left(f(\x^k)-f(\x^*)+\frac{1}{\alpha}\<\blambda^*,\U\x^k\>\right)\\
~~&\quad+\theta_2\left(f(\w^k)-f(\x^*)+\frac{1}{\alpha}\<\blambda^*,\U\w^k\>\right)\\
~~&\quad+\left(\frac{\overline L_f}{\tau b}+\frac{5\overline L_f\theta_1}{2bL_f}\left(\frac{1}{\nu}-1\right)-\theta_2\right)\left(f(\w^k)-f(\y^k)-\<\nabla f(\y^k),\w^k-\y^k\>\right)\\
~~&\quad+\frac{\theta_1^2}{2\alpha}\|\z^k-\x^*\|^2-\left(\frac{\theta_1^2}{2\alpha}+\frac{\mu\theta_1}{2}\right)\E_{\S^k}\big[\|\z^{k+1}-\x^*\|^2\big]\\
~~&\quad+\left(\frac{1}{2\alpha}-\frac{(1-\nu)\theta_1}{4\kappa L_f\alpha^2}\right)\|\blambda^k-\blambda^*\|^2-\frac{1}{2\alpha}\E_{\S^k}\big[\|\blambda^{k+1}-\blambda^*\|^2\big]\\
~~&\quad-\left(\frac{\theta_1^2}{2\alpha}-\frac{5\theta_1^3}{4L_f\alpha^2}\left(\frac{1}{\nu}-1\right)\right)\|\V\z^k\|^2-\left(\frac{\mu\theta_1}{2}-\frac{5\mu^2\theta_1}{4L_f}\left(\frac{1}{\nu}-1\right)\right)\E_{\S^k}\big[\|\z^{k+1}-\y^k\|^2\big]\\
~~&\quad-\left(\frac{\theta_1^2}{4\alpha}-\frac{\tau\theta_1^2+L_f\theta_1^2}{2}-\left(\frac{5\theta_1^3}{4L_f\alpha^2}+\frac{5L_f\theta_1^3}{4}\right)\left(\frac{1}{\nu}-1\right)\right)\E_{\S^k}\big[\|\z^{k+1}-\z^k\|^2\big].
\end{aligned}
\end{eqnarray}
Letting $\theta_1\leq \frac{1}{2}$, $\theta_2=\frac{\overline L_f}{2L_fb}$, $\tau=3L_f$, $\nu=\frac{127}{128}$, and $\alpha=\frac{1}{10L_f}$ such that $\frac{\overline L_f}{\tau b}+\frac{5\overline L_f\theta_1}{2L_fb}(\frac{1}{\nu}-1)-\theta_2\leq 0$, $\frac{\theta_1^2}{2\alpha}-\frac{5\theta_1^3}{4L_f\alpha^2}(\frac{1}{\nu}-1)\geq 0$, $\frac{\mu\theta_1}{2}-\frac{5\mu^2\theta_1}{4L_f}(\frac{1}{\nu}-1)\geq 0$, and $\frac{\theta_1^2}{4\alpha}-\frac{\tau\theta_1^2+L_f\theta_1^2}{2}-(\frac{5\theta_1^3}{4L_f\alpha^2}+\frac{5L_f\theta_1^3}{4})(\frac{1}{\nu}-1)\geq 0$, we have the conclusion.
\end{proof}

Now, we are ready to prove Theorem \ref{acc-theorem}.
\begin{proof}
Let $b\geq\max\{\frac{\max\{\sqrt{n\overline L_f/\mu},n\}}{\max\{\sqrt{\kappa L_f/\mu},\kappa\}},\frac{\overline L_f}{L_f}\}$, then we know $\theta_2=\frac{\overline L_f}{2L_fb}\leq\frac{1}{2}$ and $b\geq 1$, where we use $\overline L_f\geq L_f$. We also have $\max\{\frac{\max\{\sqrt{n\overline L_f/\mu},n\}}{\max\{\sqrt{\kappa L_f/\mu},\kappa\}},\frac{\overline L_f}{L_f}\}\leq\max\{\max\{\sqrt{\frac{n\overline L_f}{\mu}},n\}\sqrt{\frac{\mu}{\kappa L_f}},\frac{\overline L_f}{L_f}\}=\max\{\sqrt{\frac{n\overline L_f}{\kappa L_f}},\sqrt{\frac{n^2\mu}{\kappa L_f}},\frac{\overline L_f}{L_f}\}\overset{a}\leq n$, where $\overset{a}\leq$ uses $\kappa\geq 1$ and $\mu\leq L_f\leq\overline L_f\leq nL_f$ given in (\ref{L-relation}). This verifies that the setting of $b$ is meaningful.

Multiplying both sides of (\ref{non-acc-cont3}) by $\frac{\theta_2}{\frac{b}{n}-\frac{\theta_1}{20\kappa}}$ and adding it to (\ref{acc-cont9}), we have
\begin{eqnarray}
\hspace*{-1.5cm}\begin{aligned}\label{acc-cont10}
&\left(1-\frac{\theta_1}{2}\right)\E_{\S^k}\left[f(\x^{k+1})-f(\x^*)+\frac{1}{\alpha}\<\blambda^*,\U\x^{k+1}\>\right]\\
~~&\quad+\frac{\theta_2}{\frac{b}{n}-\frac{\theta_1}{20\kappa}}\E_{\w^{k+1}}\left[f(\w^{k+1})-f(\x^*)+\frac{1}{\alpha}\<\blambda^*,\U\w^{k+1}\>\right]\\
~~&\leq\left(1-\theta_1-\theta_2+\frac{b}{n}\frac{\theta_2}{\frac{b}{n}-\frac{\theta_1}{20\kappa}}\right)\left(f(\x^k)-f(\x^*)+\frac{1}{\alpha}\<\blambda^*,\U\x^k\>\right)\\
~~&\quad+\left(\theta_2+\left(1-\frac{b}{n}\right)\frac{\theta_2}{\frac{b}{n}-\frac{\theta_1}{20\kappa}}\right)\left(f(\w^k)-f(\x^*)+\frac{1}{\alpha}\<\blambda^*,\U\w^k\>\right)\\
~~&\quad+\frac{\theta_1^2}{2\alpha}\|\z^k-\x^*\|^2-\left(\frac{\theta_1^2}{2\alpha}+\frac{\mu\theta_1}{2}\right)\E_{\S^k}\big[\|\z^{k+1}-\x^*\|^2\big]\\
~~&\quad+\left(\frac{1}{2\alpha}-\frac{(1-\nu)\theta_1}{4\kappa L_f\alpha^2}\right)\|\blambda^k-\blambda^*\|^2-\frac{1}{2\alpha}\E_{\S^k}\big[\|\blambda^{k+1}-\blambda^*\|^2\big].
\end{aligned}
\end{eqnarray}
We can easily check that
\begin{eqnarray}
\begin{aligned}\notag
&1-\theta_1-\theta_2+\frac{b}{n}\frac{\theta_2}{\frac{b}{n}-\frac{\theta_1}{20\kappa}}=1-\theta_1-\theta_2+\frac{\theta_2}{1-\frac{n\theta_1}{20b\kappa}}\\
~~&\quad=1-\theta_1+\frac{\frac{n\theta_1\theta_2}{20b\kappa}}{1-\frac{n\theta_1}{20b\kappa}}\overset{b}\leq1-\theta_1+\frac{\theta_1}{39}=1-\frac{38}{39}\theta_1,
\end{aligned}
\end{eqnarray}
and
\begin{eqnarray}
\begin{aligned}\notag
\theta_2+\left(1-\frac{b}{n}\right)\frac{\theta_2}{\frac{b}{n}-\frac{\theta_1}{20\kappa}}&=\theta_2+\left(\frac{\theta_1}{20\kappa}-\frac{b}{n}\right)\frac{\theta_2}{\frac{b}{n}-\frac{\theta_1}{20\kappa}}+\left(1-\frac{\theta_1}{20\kappa}\right)\frac{\theta_2}{\frac{b}{n}-\frac{\theta_1}{20\kappa}}\\
&=\left(1-\frac{\theta_1}{20\kappa}\right)\frac{\theta_2}{\frac{b}{n}-\frac{\theta_1}{20\kappa}},
\end{aligned}
\end{eqnarray}
where we check $\overset{b}\leq$ in the following two cases. In the first case, if $\kappa\leq\frac{L_f}{\mu}$, we have $\theta_1=\frac{1}{2}\sqrt{\frac{\kappa\mu}{L_f}}$ and $b\geq\max\{\sqrt{\frac{n\overline L_f}{\kappa L_f}},\sqrt{\frac{n^2\mu}{\kappa L_f}},\frac{\overline L_f}{L_f}\}$. So we have $\frac{n\theta_2}{20b\kappa}\overset{c}=\frac{n\overline L_f}{40L_fb^2\kappa}\leq\frac{n\overline L_f}{40L_f\kappa}\frac{\kappa L_f}{n\overline L_f}=\frac{1}{40}$ and $\frac{n\theta_1}{20b\kappa}=\frac{n}{40b\kappa}\sqrt{\frac{\kappa\mu}{L_f}}\leq\frac{n}{40\kappa}\sqrt{\frac{\kappa L_f}{n^2\mu}}\sqrt{\frac{\kappa\mu}{L_f}}=\frac{1}{40}$, where $\overset{c}=$ uses the setting of $\theta_2$. So we get $\overset{b}\leq$. In the second case, if $\kappa\geq\frac{L_f}{\mu}$, we have $\theta_1=\frac{1}{2}$ and $b\geq\max\{\sqrt{\frac{n\overline L_f}{\mu}}\frac{1}{\kappa},\frac{n}{\kappa},\frac{\overline L_f}{L_f}\}\geq \frac{n}{\kappa}$. So we have $\frac{n\theta_2}{20b\kappa}\overset{d}\leq\frac{n}{40b\kappa}\leq\frac{1}{40}$ and $\frac{n\theta_1}{20b\kappa}=\frac{n}{40b\kappa}\leq\frac{1}{40}$, where $\overset{d}\leq$ uses $\theta_2\leq\frac{1}{2}$ derived at the beginning of this proof. So we also get $\overset{b}\leq$.

Taking expectation with respect to $\xi^k$ on both sides of (\ref{acc-cont10}) and rearranging the terms, we have
\begin{eqnarray}
\begin{aligned}\notag
&\left(1-\frac{\theta_1}{2}\right)\E_{\xi^{k+1}}\left[f(\x^{k+1})-f(\x^*)+\frac{1}{\alpha}\<\blambda^*,\U\x^{k+1}\>\right]\\
~~&\quad+\frac{\theta_2}{\frac{b}{n}-\frac{\theta_1}{20\kappa}}\E_{\xi^{k+1}}\left[f(\w^{k+1})-f(\x^*)+\frac{1}{\alpha}\<\blambda^*,\U\w^{k+1}\>\right]\\
~~&\quad+\left(\frac{\theta_1^2}{2\alpha}+\frac{\mu\theta_1}{2}\right)\E_{\xi^{k+1}}\big[\|\z^{k+1}-\x^*\|^2\big]+\frac{1}{2\alpha}\E_{\xi^{k+1}}\big[\|\blambda^{k+1}-\blambda^*\|^2\big]\\
~~&\leq\left(1-\frac{38}{39}\theta_1\right)\E_{\xi^k}\left[f(\x^k)-f(\x^*)+\frac{1}{\alpha}\<\blambda^*,\U\x^k\>\right]\\
~~&\quad+\frac{\theta_2}{\frac{b}{n}-\frac{\theta_1}{20\kappa}}\left(1-\frac{\theta_1}{20\kappa}\right)\E_{\xi^k}\left[f(\w^k)-f(\x^*)+\frac{1}{\alpha}\<\blambda^*,\U\w^k\>\right]\\
~~&\quad+\frac{\theta_1^2}{2\alpha}\E_{\xi^k}\big[\|\z^k-\x^*\|^2\big]+\left(\frac{1}{2\alpha}-\frac{(1-\nu)\theta_1}{4\kappa L_f\alpha^2}\right)\E_{\xi^k}\big[\|\blambda^k-\blambda^*\|^2\big]\\
~~&\overset{f}\leq\left\{\left(1-\frac{\theta_1}{2}\right)\E_{\xi^k}\left[f(\x^k)-f(\x^*)+\frac{1}{\alpha}\<\blambda^*,\U\x^k\>\right]\right.\\
~~&\qquad\left.+\frac{\theta_2}{\frac{b}{n}-\frac{\theta_1}{20\kappa}}\E_{\xi^k}\left[f(\w^k)-f(\x^*)+\frac{1}{\alpha}\<\blambda^*,\U\w^k\>\right]\right.\\
~~&\qquad\left.+\left(\frac{\theta_1^2}{2\alpha}+\frac{\mu\theta_1}{2}\right)\E_{\xi^k}\big[\|\z^k-\x^*\|^2\big]+\frac{1}{2\alpha}\E_{\xi^k}\big[\|\blambda^k-\blambda^*\|^2\big]\right\}\\
~~&\quad\times \max\left\{\frac{1-\frac{38}{39}\theta_1}{1-\frac{\theta_1}{2}},1-\frac{\theta_1}{20\kappa},\frac{1}{1+\frac{\mu\alpha}{\theta_1}},1-\frac{(1-\nu)\theta_1}{2\kappa L_f\alpha}\right\},
\end{aligned}
\end{eqnarray}
where $\overset{f}\leq$ uses the fact that $f(\x)+\frac{1}{\alpha}\<\blambda^*,\U\x\>\geq f(\x^*)+\frac{1}{\alpha}\<\blambda^*,\U\x^*\>$ for any $\x$, and $\frac{n\theta_1}{20b\kappa}\leq\frac{1}{40}$ in the above analysis such that $\frac{b}{n}\geq\frac{\theta_1}{20\kappa}$.

From the settings of $\theta_1$ and $\alpha$ and $\kappa\geq 1$, we can easily check $\frac{1-\frac{38}{39}\theta_1}{1-\frac{\theta_1}{2}}\leq 1-\frac{18}{39}\theta_1\leq 1-\frac{18}{39}\frac{\theta_1}{\kappa}=\bO(1-\frac{\theta_1}{\kappa})$, $\frac{1}{1+\frac{\mu\alpha}{\theta_1}}\leq 1-\frac{\mu\alpha}{2\theta_1}=\bO(1-\frac{\mu}{L_f\theta_1})$ due to $\mu\alpha=\frac{\mu}{10L_f}\leq\frac{1}{10}\sqrt{\frac{\mu}{L_f}}\leq\theta_1$, and $1-\frac{(1-\nu)\theta_1}{2\kappa L_f\alpha}=\bO(1-\frac{\theta_1}{\kappa})$. Thus the algorithm needs $\bO((\frac{\kappa}{\theta_1}+\frac{L_f\theta_1}{\mu})\log\frac{1}{\epsilon})$ iterations to find $\z^k$ such that $\E_{\xi^k}\big[\|\z^k-\x^*\|^2\big]\leq\epsilon$.

We first consider the communication complexity.
\begin{enumerate}
\item If $\kappa\leq\frac{L_f}{\mu}$, we have $\theta_1=\frac{1}{2}\sqrt{\frac{\kappa\mu}{L_f}}$ and $\bO((\frac{\kappa}{\theta_1}+\frac{L_f\theta_1}{\mu})\log\frac{1}{\epsilon})=\bO(\sqrt{\frac{\kappa L_f}{\mu}}\log\frac{1}{\epsilon})$. So the communication complexity is $\bO(\sqrt{\frac{\kappa L_f}{\mu}}\log\frac{1}{\epsilon})$.
\item If $\kappa\geq\frac{L_f}{\mu}$, we have $\theta_1=\frac{1}{2}$ and $\bO((\frac{\kappa}{\theta_1}+\frac{L_f\theta_1}{\mu})\log\frac{1}{\epsilon})=\bO((\kappa+\frac{L_f}{\mu})\log\frac{1}{\epsilon})=\bO(\kappa\log\frac{1}{\epsilon})$. So the communication complexity is $\bO(\kappa\log\frac{1}{\epsilon})$.
\end{enumerate}
So the algorithm needs the time of $\bO(\max\{\sqrt{\frac{\kappa L_f}{\mu}},\kappa\}\log\frac{1}{\epsilon})$ communication rounds to find an $\epsilon$-optimal solution $\z^k$ such that $\E_{\xi^k}\big[\|\z^k-\x^*\|^2\big]\leq\epsilon$.

Next, we consider the stochastic gradient computation complexity.
\begin{enumerate}
\item If $\max\{\sqrt{\frac{n\overline L_f}{\mu}},n\}\frac{L_f}{\overline L_f}\geq \max\{\sqrt{\frac{\kappa L_f}{\mu}},\kappa\}$ such that $b=\frac{\max\{\sqrt{n\overline L_f/\mu},n\}}{\max\{\sqrt{\kappa L_f/\mu},\kappa\}}$, the stochastic gradient computation complexity is $\bO(b\max\{\sqrt{\frac{\kappa L_f}{\mu}},\kappa\}\log\frac{1}{\epsilon})=\bO(\max\{\sqrt{\frac{n\overline L_f}{\mu}},n\}\log\frac{1}{\epsilon})$.

\item If $\max\{\sqrt{\frac{n\overline L_f}{\mu}},n\}\frac{L_f}{\overline L_f}\leq \max\{\sqrt{\frac{\kappa L_f}{\mu}},\kappa\}$ such that $b=\frac{\overline L_f}{L_f}$, the stochastic gradient computation complexity is $\bO(b\max\{\sqrt{\frac{\kappa L_f}{\mu}},\kappa\}\log\frac{1}{\epsilon})=\bO(\frac{\overline L_f}{L_f}\max\{\sqrt{\frac{\kappa L_f}{\mu}},\kappa\}\log\frac{1}{\epsilon})$.
\item If we choose $b>\max\{\frac{\max\{\sqrt{n\overline L_f/\mu},n\}}{\max\{\sqrt{\kappa L_f/\mu},\kappa\}},\frac{\overline L_f}{L_f}\}$, the stochastic gradient computation complexity is higher than the above ones. But the communication complexity remains unchanged. This verifies Remark \ref{acc-remark}(2).
\end{enumerate}
At last, we discuss the condition $\max\{\sqrt{\frac{n\overline L_f}{\mu}},n\}\frac{L_f}{\overline L_f}\geq \max\{\sqrt{\frac{\kappa L_f}{\mu}},\kappa\}$. We know that $\phi(\kappa)\equiv\max\{\sqrt{\frac{\kappa L_f}{\mu}},\kappa\}$ is a piece-wise increasing function with respect to $\kappa$ such that $\phi(\kappa)=\left\{
  \begin{array}{cl}
    \sqrt{\frac{\kappa L_f}{\mu}}, & \mbox{if }0\leq\kappa\leq\frac{L_f}{\mu},\\
    \kappa, & \mbox{if }\kappa\geq\frac{L_f}{\mu},
  \end{array}
\right.$ and $\phi(\kappa)\left\{
  \begin{array}{cl}
    \leq \frac{L_f}{\mu}, & \mbox{if }\phi(\kappa)=\sqrt{\frac{\kappa L_f}{\mu}},\\
    \geq \frac{L_f}{\mu}, & \mbox{if }\phi(\kappa)=\kappa.
  \end{array}
\right.$

\begin{enumerate}
\item If $n\geq\frac{\overline L_f}{\mu}$, we have $\max\{\sqrt{\frac{n\overline L_f}{\mu}},n\}\frac{L_f}{\overline L_f}=\frac{nL_f}{\overline L_f}\geq\frac{L_f}{\mu}$. So the condition $\max\{\sqrt{\frac{n\overline L_f}{\mu}},n\}\frac{L_f}{\overline L_f}\geq \max\{\sqrt{\frac{\kappa L_f}{\mu}},\kappa\}$ is equivalent to $\frac{nL_f}{\overline L_f}\geq\kappa$.
\item If $n\leq\frac{\overline L_f}{\mu}$, we have $\max\{\sqrt{\frac{n\overline L_f}{\mu}},n\}\frac{L_f}{\overline L_f}=\sqrt{\frac{nL_f^2}{\mu\overline L_f}}\leq\frac{L_f}{\mu}$. So $\max\{\sqrt{\frac{n\overline L_f}{\mu}},n\}\frac{L_f}{\overline L_f}\geq \max\{\sqrt{\frac{\kappa L_f}{\mu}},\kappa\}$ is equivalent to $\sqrt{\frac{nL_f^2}{\mu\overline L_f}}\geq\sqrt{\frac{\kappa L_f}{\mu}}$, that is, $\frac{nL_f}{\overline L_f}\geq\kappa$.
\end{enumerate}
\vspace*{-0.7cm}
\end{proof}

\section{Numerical Experiments}\label{sec:exp}
Consider the following decentralized regularized logistic regression problem:
\begin{eqnarray}\notag
\min_{x\in\R^p} \sum_{i=1}^m f_{(i)}(x),\quad\mbox{where}\quad f_{(i)}(x)=\frac{\mu}{2}\|x\|^2+\frac{1}{n}\sum_{j=1}^n \log\left(1+\exp(-y_{(i),j}\A_{(i),j}^Tx)\right),
\end{eqnarray}
where the pairs $(\A_{(i),j},y_{(i),j})\in\R^p\times\{1,-1\}$ are taken from the RCV1 dataset\footnote{https://www.csie.ntu.edu.tw/~cjlin/libsvmtools/datasets/binary.html} with $p=47236$, $m=49$, and $n=500$. Denote $\A_{(i)}=[\A_{(i),1},\A_{(i),2},\cdots,\A_{(i),n}]\in\R^{p\times n}$ as the data matrix on the $i$th node. For this special problem and dataset, we observe $L_f=\max_i \frac{\|\A_{(i)}\|_2^2}{4n}+\mu\approx 0.016+\mu$ and $\overline L_f=\max_i \frac{\|\A_{(i)}\|_F^2}{4n}+\mu=\frac{1}{4}+\mu$, respectively. We test the performance of the proposed algorithms on different ratios between $\kappa_s$ and $n$. Specifically, we test on $\mu=5\times 10^{-5}$, $\mu=5\times 10^{-6}$, and $\mu=5\times 10^{-7}$, which correspond to $\kappa_s=\frac{\overline L_f}{\mu}\approx 5\times 10^3$, $\kappa_s\approx 5\times 10^4$, and $\kappa_s\approx 5\times 10^5$, respectively. Note that $n=500$. We also observe $\frac{n\kappa_b}{\kappa_s}\approx 31.9$.

We test the performance on two kinds of networks: the Erd\H{o}s$-$R\'{e}nyi random graph and the two-dimensional grid graph, where each pair of nodes has a connection with the ratio of $0.2$ for the first graph, and $m$ nodes are placed in the $\sqrt{m}\times \sqrt{m}$ grid and each node is connected with its neighbors around it for the second graph. Theoretically, $\kappa_c=\bO(1)$ for the first graph, and $\kappa_c=\bO(m\log m)$ for the second graph. Practically, we observe $\kappa_c=4.62$ and $\kappa_c=19.9$ for the two graphs, respectively. We set the weight matrix as $\W=\frac{M-\lambda_{\min}\I}{1-\lambda_{\min}}$ for both graphs, where $M$ is the Metropolis weight
matrix \citep{boyd2004} with $M_{ij}=\left\{
  \begin{array}{ll}
    \frac{1}{\max\{d(i),d(j)\}} & \mbox{if }(i,j)\in\mathcal{E},\\
    1-\sum_{l\in\N_{(i)}}M_{il} & \mbox{if } i=j,\\
    0 & \mbox{if } (i,j)\notin\mathcal{E}\mbox{ and }i\neq j.
  \end{array}
\right.$, $d(i)$ is the degree of node $i$, and $\lambda_{\min}<0$ is the smallest negative eigenvalue of $M$.

We compare the proposed VR-EXTRA, Acc-VR-EXTRA, Acc-VR-EXTRA-CA, VR-DIGing, Acc-VR-DIGing, and Acc-VR-DIGing-CA with EXTRA \citep{shi2015extra}, DIGing \citep{shi2017}, GT-SVRG \citep{ranxin2020}, APAPC \citep{richtaric-2020-dist}, and accelerated DVR (Acc-DVR) \citep{DVR}. For the case of $\mu=5\times 10^{-6}$, we choose the best step-sizes $\alpha=\frac{1}{L_f}$ for Acc-VR-EXTRA, Acc-VR-EXTRA-CA, Acc-VR-DIGing, and Acc-VR-DIGing-CA, $\alpha=\frac{3}{L_f}$ for VR-EXTRA, VR-DIGing, and GT-SVRG, and $\alpha=\frac{7}{L_f}$ for EXTRA and DIGing. 
The other parameters are chosen according to the theories. We set the parameters of APAPC according to Theorem 2 in \citep{richtaric-2020-dist}. For Acc-DVR, we follow the suggestions in the experimental section in \citep{DVR} to set the number of inner iterations to $\frac{n}{1-p_{comm}}$ (one pass over the local dataset), and the batch Lipschitz constant as $L_f=0.01\overline L_f$, which leads to better performance of Acc-DVR than the theoretical setting of $L_f=\max_i \frac{\|\A_{(i)}\|_2^2}{4n}+\mu$. We follow Algorithm 2 and Theorem 5 in \citep{DVR} to choose other parameters of Acc-DVR, that is, $\alpha=\frac{2}{L_f\kappa_c}$, $\eta=\min\{p_{comm}(\beta+\mu),\frac{p_{ij}}{\alpha(1+1/(4n\mu))}\}$, $\beta=\frac{\overline L_f}{n}-\mu$, $p_{comm}=(1+\frac{n+\kappa_s^{\beta}}{\kappa_b^{\beta}\kappa_c})^{-1}$, $p_{ij}=\frac{1-p_{comm}}{n}$, $\kappa_s^{\beta}=\frac{\overline L_f}{\beta+\mu}$, and $\kappa_b^{\beta}=\frac{L_f}{\beta+\mu}$\footnote{We do not find the parameters in APAPC and Acc-DVR in the role as step-sizes in the form of $\bO(\frac{1}{L_f})$, thus we set the parameters according to their theories directly.}. For the case of $\mu=5\times 10^{-7}$, we choose the same parameters as above. For the case of $\mu=5\times 10^{-5}$, we set the step-sizes $\alpha=\frac{2}{L_f}$ for VR-EXTRA, VR-DIGing, and GT-SVRG, and $\alpha=\frac{3}{L_f}$ for EXTRA and DIGing, and other parameters are chosen the same as above. Specially, since $\frac{n\kappa_b}{\kappa_s}<2\kappa_c$ and $\frac{n\kappa_b}{\kappa_s}<\kappa_c^2$ for the grid graph, we set the mini-batch size $b=\frac{\overline L_f}{L_f}$ for Acc-VR-EXTRA and Acc-VR-DIGing according to Remark \ref{acc-remark}(1).

\begin{figure*}
\centering
\begin{tabular}{@{\extracolsep{0.001em}}c@{\extracolsep{0.001em}}c}
\includegraphics[width=0.5\textwidth,keepaspectratio]{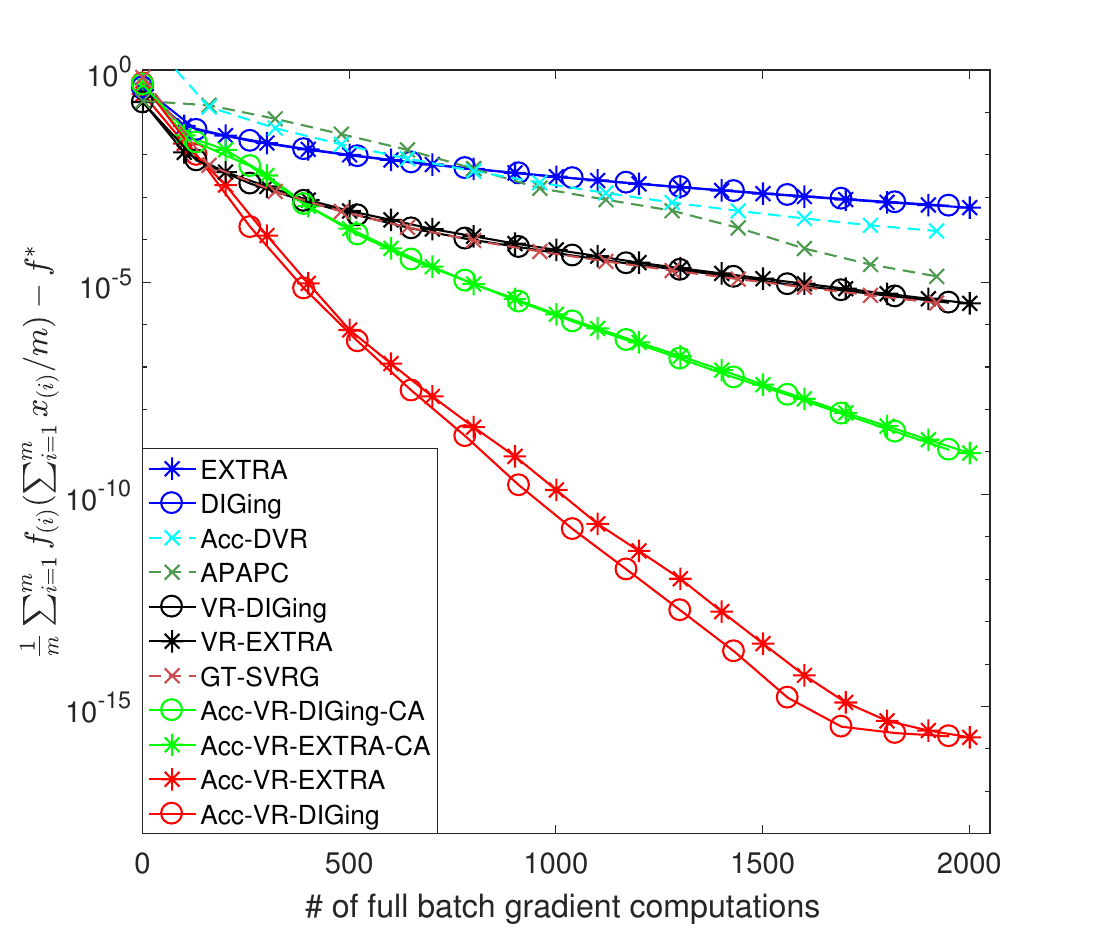}
&\includegraphics[width=0.5\textwidth,keepaspectratio]{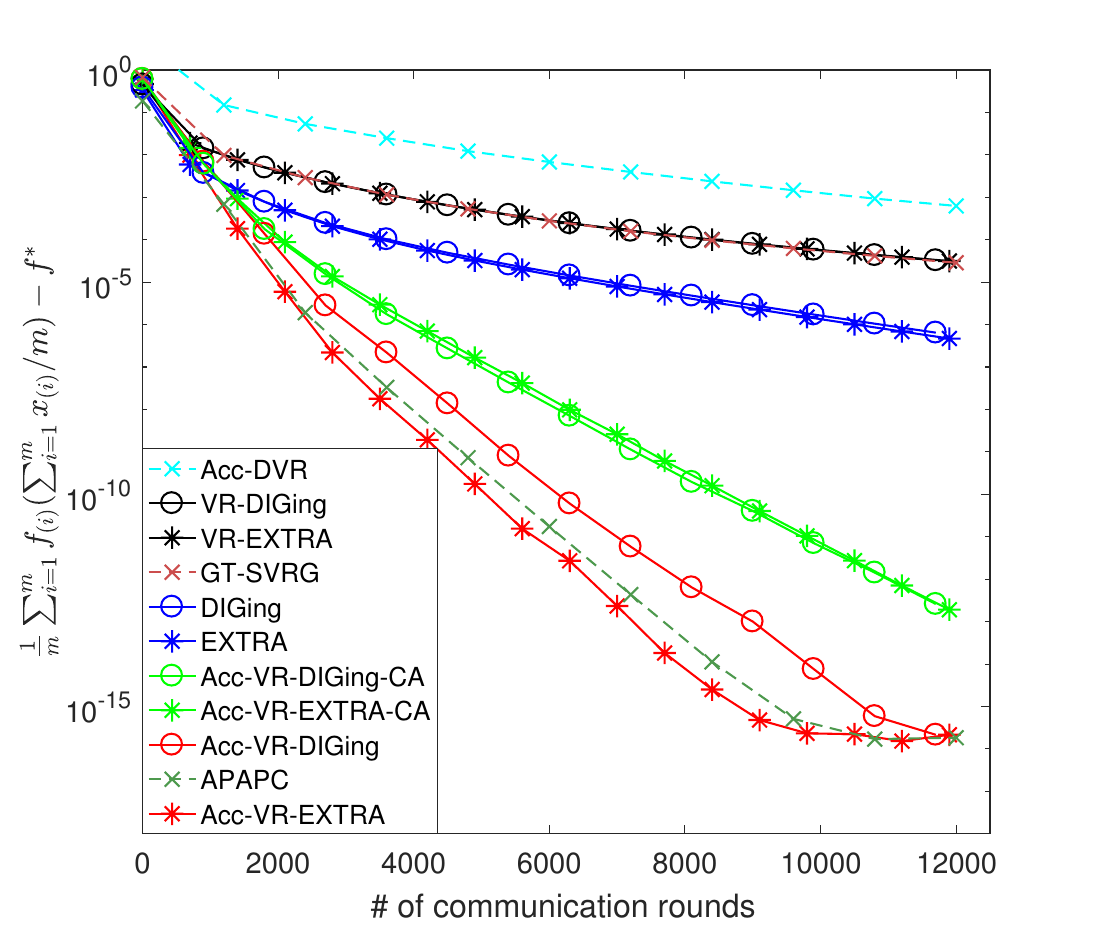}\\
\includegraphics[width=0.5\textwidth,keepaspectratio]{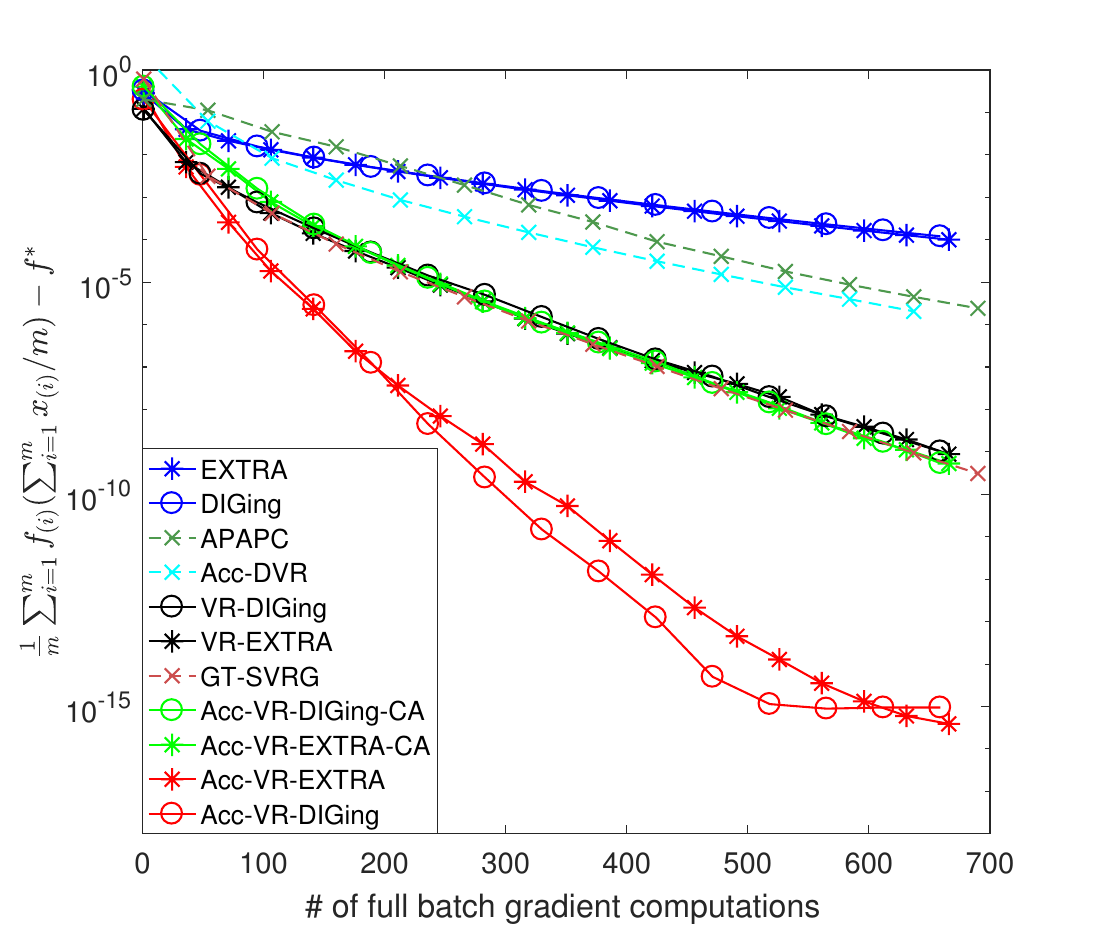}
&\includegraphics[width=0.5\textwidth,keepaspectratio]{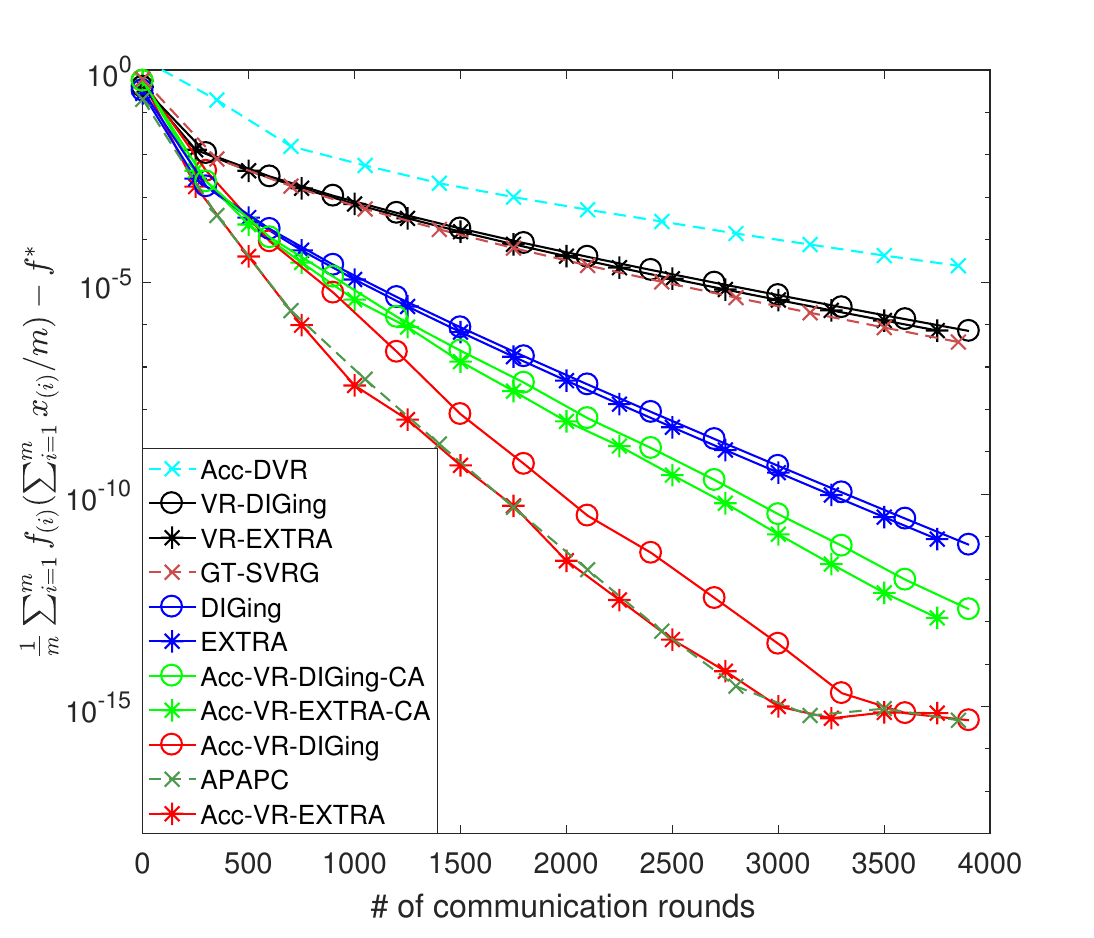}\\
\includegraphics[width=0.5\textwidth,keepaspectratio]{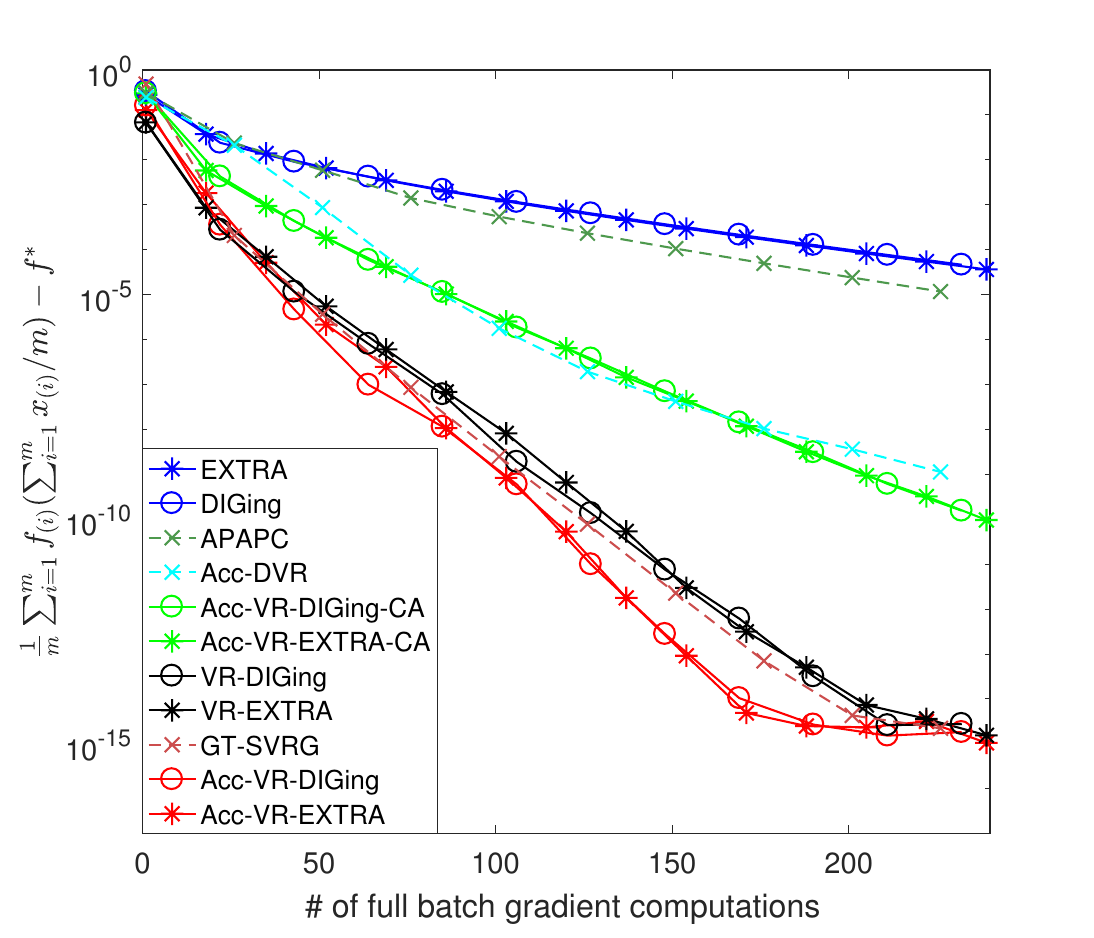}
&\includegraphics[width=0.5\textwidth,keepaspectratio]{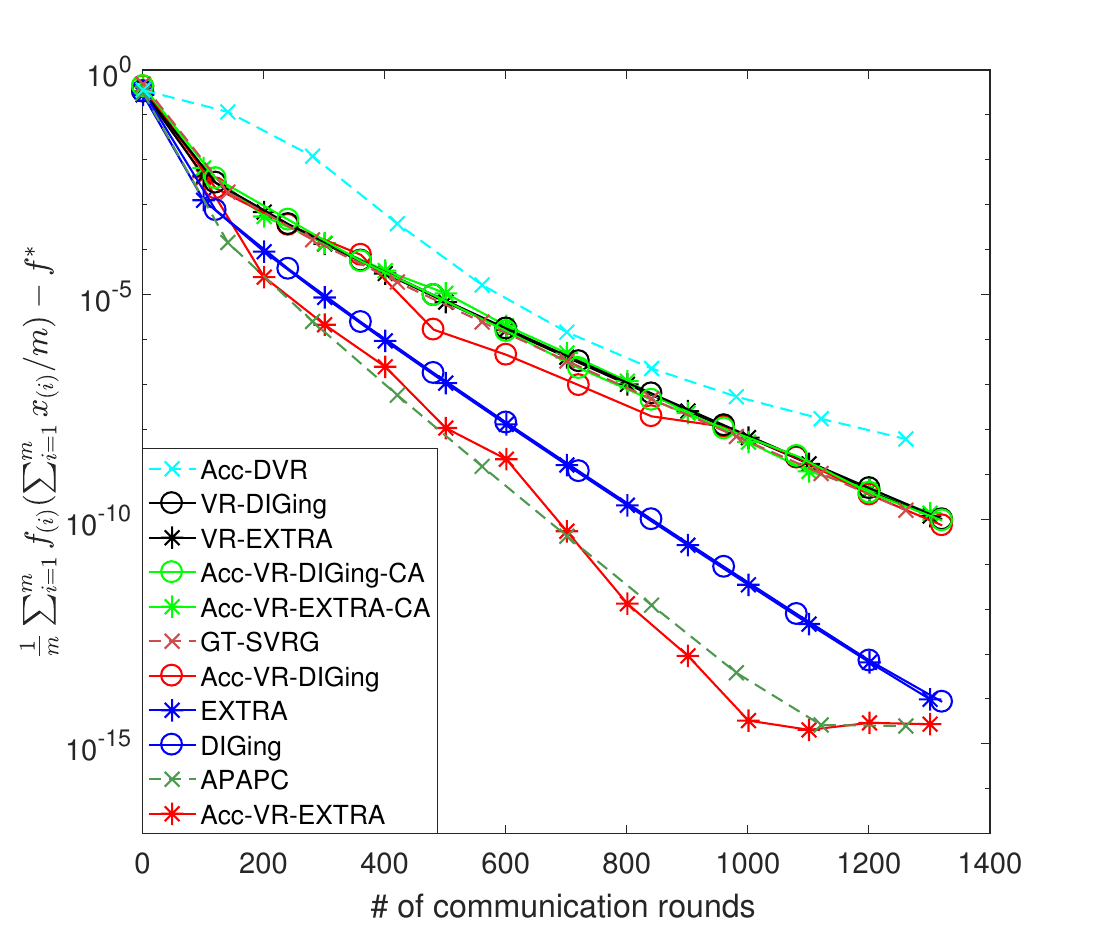}
\end{tabular}
\caption{Comparisons on Erd\H{o}s$-$R\'{e}nyi random graph with $\mu=5\times 10^{-7}$ (top), $\mu=5\times 10^{-6}$ (middle), and $\mu=5\times 10^{-5}$ (bottom).}\label{fig1}
\end{figure*}

\begin{figure*}
\centering
\begin{tabular}{@{\extracolsep{0.001em}}c@{\extracolsep{0.001em}}c}
\includegraphics[width=0.5\textwidth,keepaspectratio]{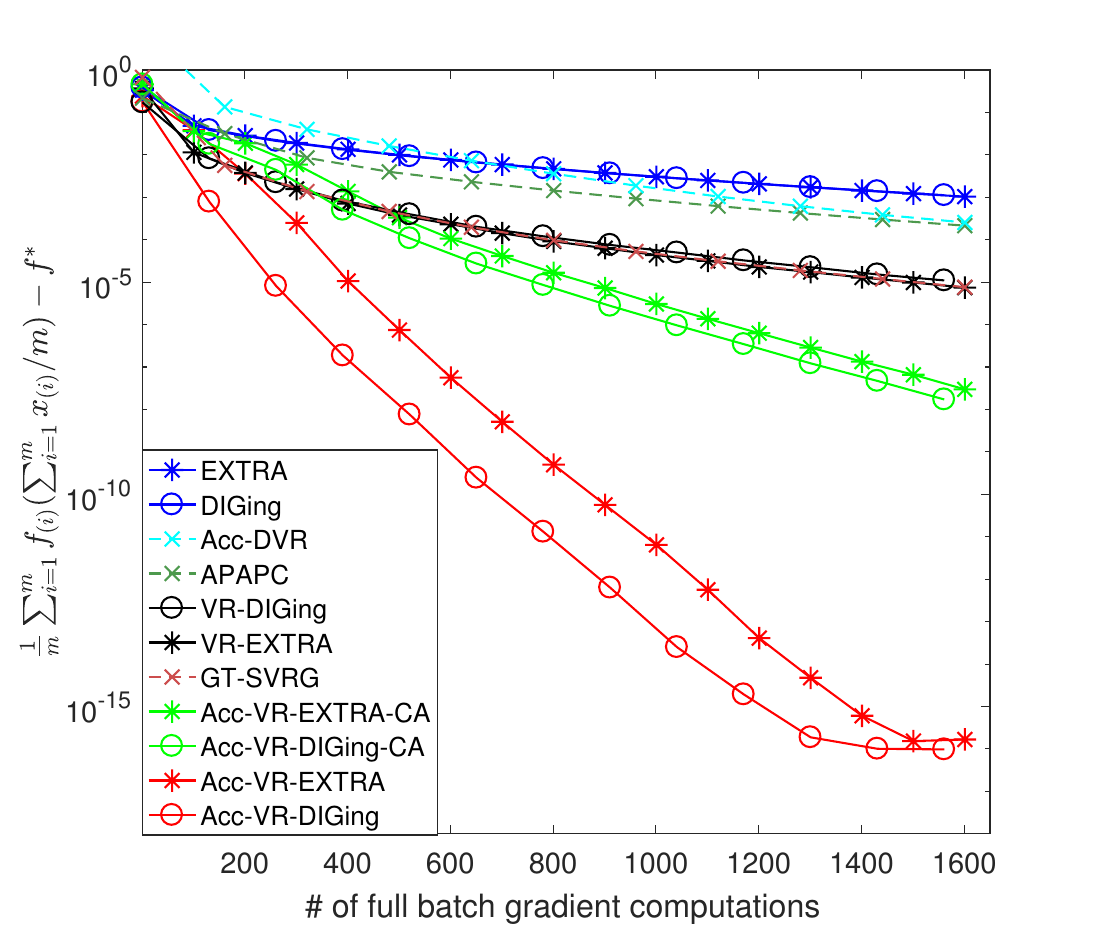}
&\includegraphics[width=0.5\textwidth,keepaspectratio]{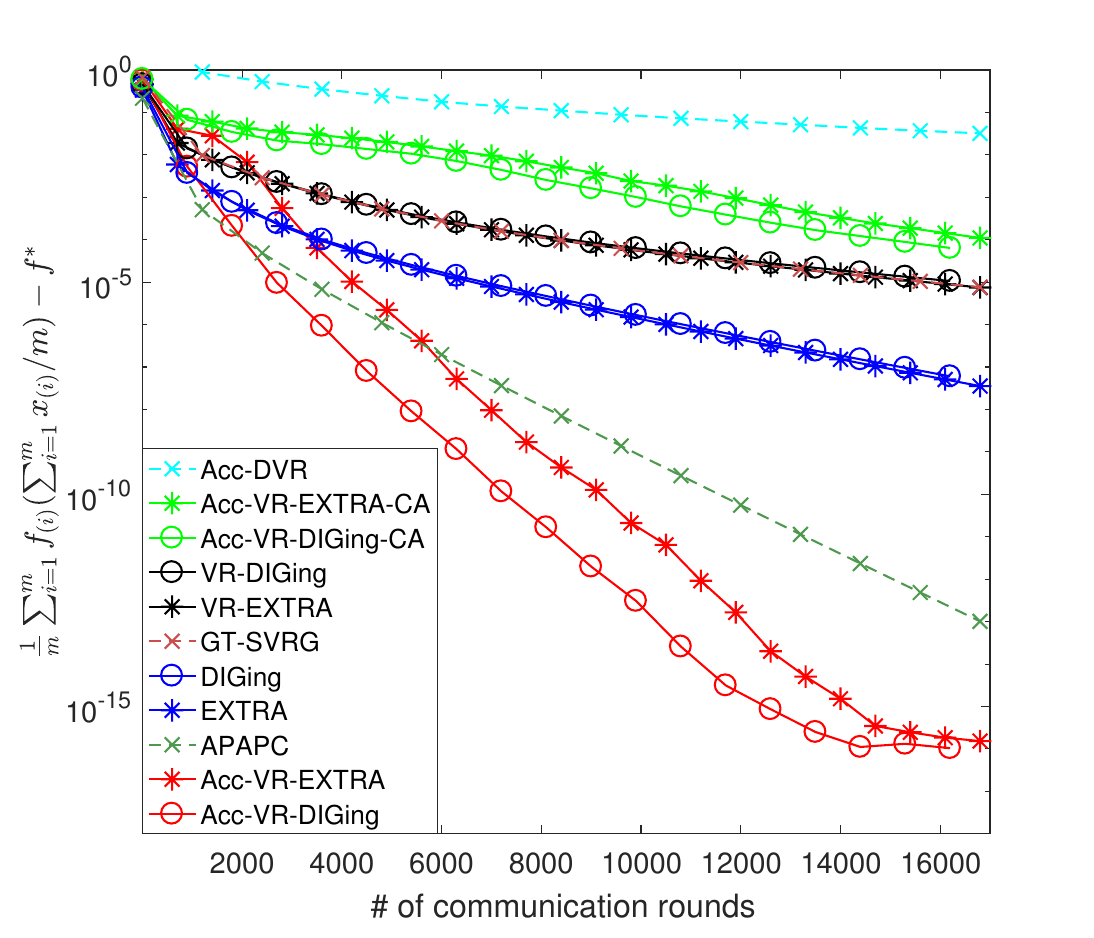}\\
\includegraphics[width=0.5\textwidth,keepaspectratio]{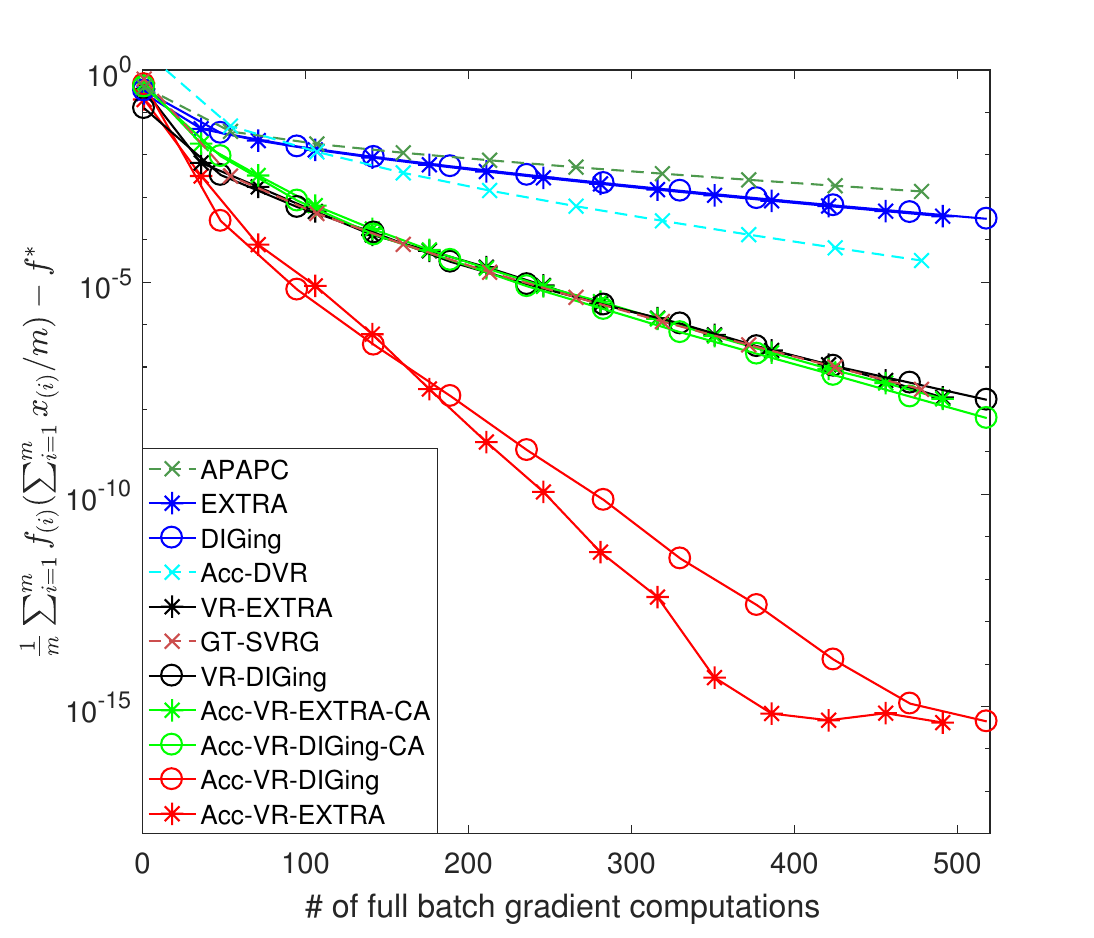}
&\includegraphics[width=0.5\textwidth,keepaspectratio]{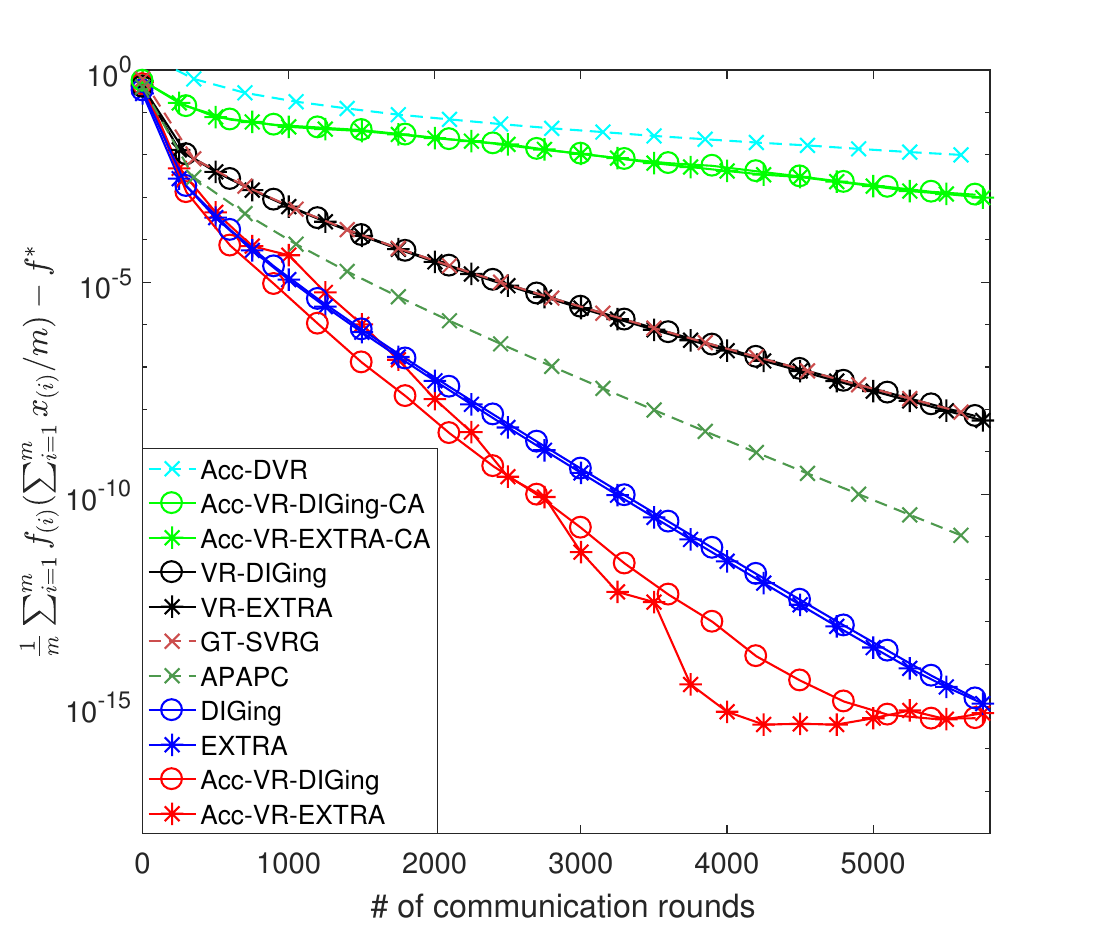}\\
\includegraphics[width=0.5\textwidth,keepaspectratio]{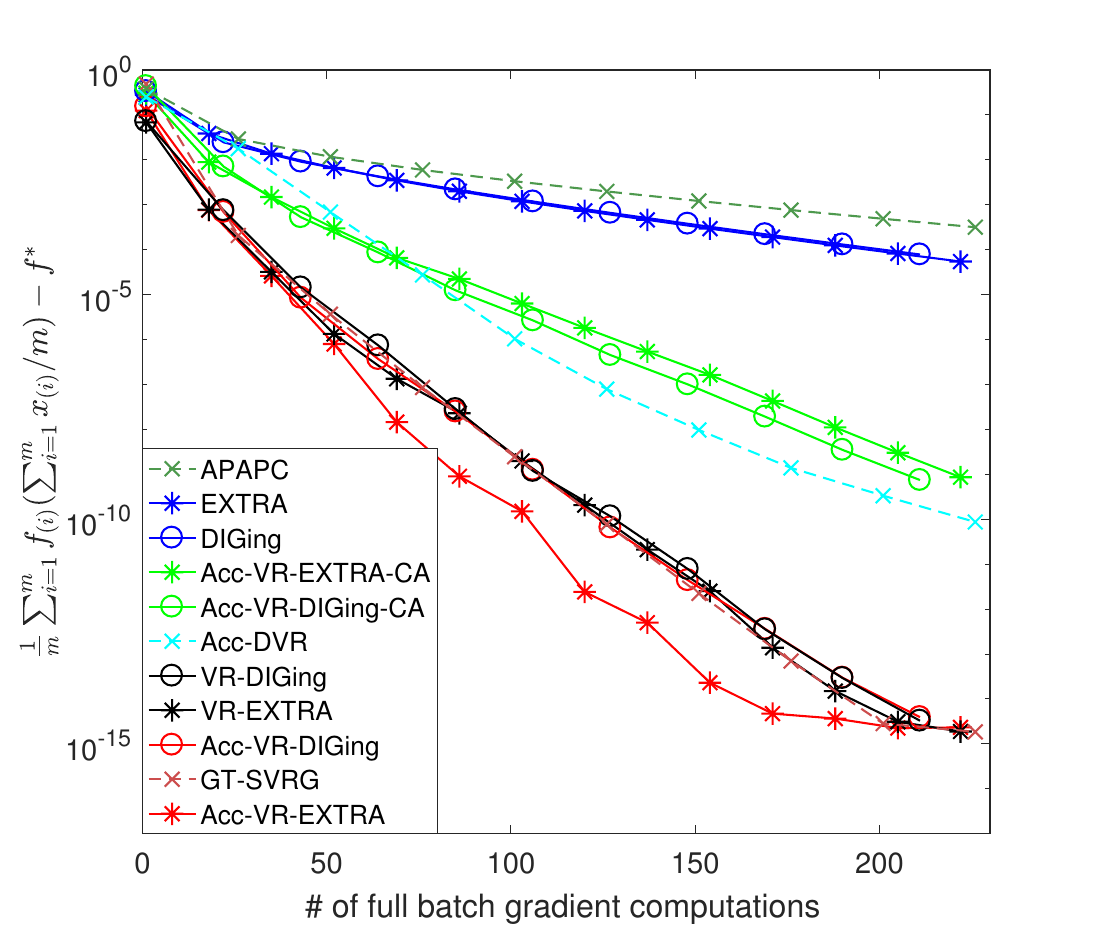}
&\includegraphics[width=0.5\textwidth,keepaspectratio]{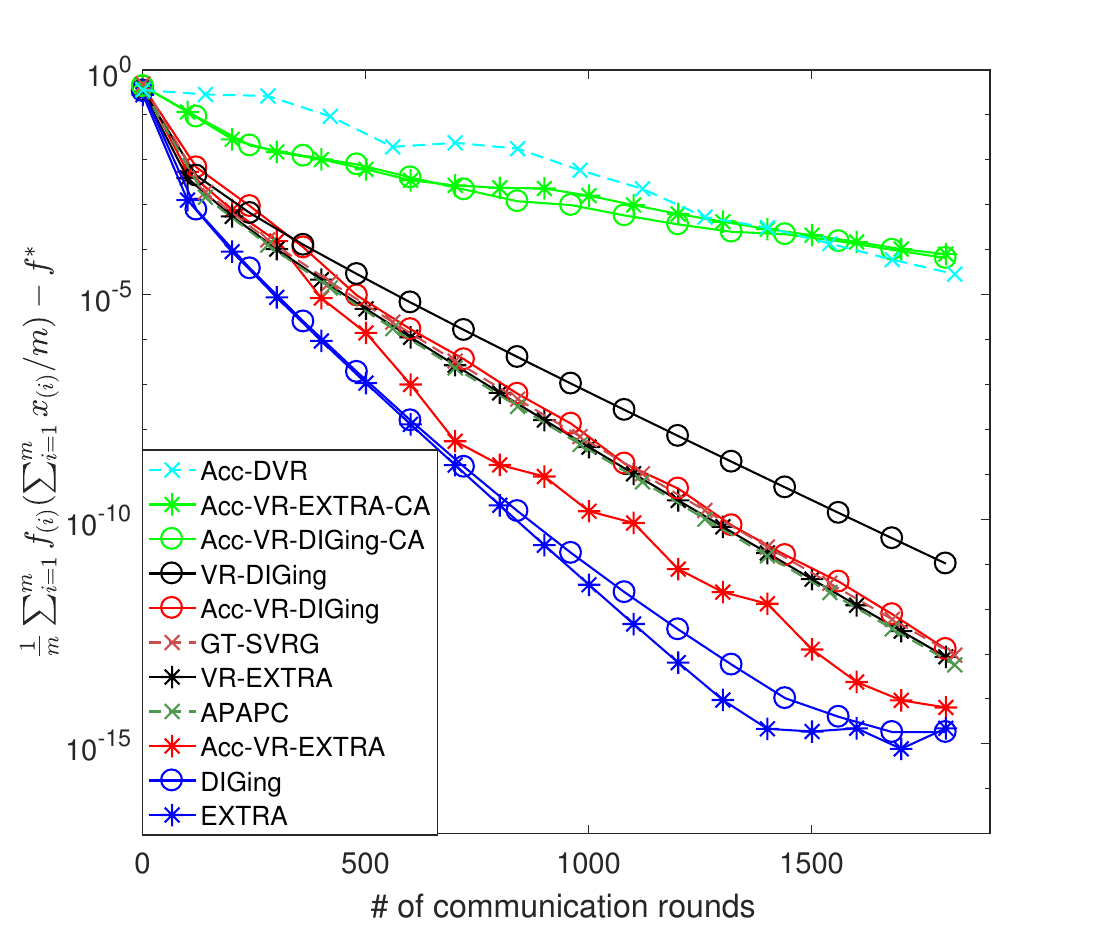}
\end{tabular}
\caption{Comparisons on grid graph with $\mu=5\times 10^{-7}$ (top), $\mu=5\times 10^{-6}$ (middle), and $\mu=5\times 10^{-5}$ (bottom).}\label{fig2}
\end{figure*}

Figures \ref{fig1} and \ref{fig2} plot the results on the Erd\H{o}s$-$R\'{e}nyi random graph and grid graph, respectively, where $f^*$ is approximated by the minimum value of the objective function over all iterations of all the compared algorithms. We have the following observations:
\begin{enumerate}
\item VR-EXTRA and VR-DIGing need less gradient computations than EXTRA and DIGing to reach the same precision for all cases of $\mu$. This verifies the efficiency of variance reduction in decentralized optimization to reduce the computation cost.
\item When considering the communication cost, VR-EXTRA and VR-DIGing perform worse than EXTRA and DIGing in practice, although they have the same communication complexities theoretically. This is reasonable since EXTRA and DIGing go through all the data at each communication round, while VR-EXTRA and VR-DIGing only use a mini-batch. We observe the mini-batch size of $b=\frac{\overline L_f}{L_f}\approx16$ in our experiment.
\item Acc-VR-EXTRA and Acc-VR-DIGing perform better than VR-EXTRA and VR-DIGing on both computations and communications when $\kappa_s$ is much larger than $n$. For example, $\kappa_s=1000n$ in the top plots and $\kappa_s=100n$ in the middle plots. Otherwise, as seen in the bottom plots with $\kappa_s=10n$, Acc-VR-EXTRA and Acc-VR-DIGing perform quite similarly to VR-EXTRA and VR-DIGing, especially on the computations. This verifies that acceleration only takes effect to reduce the computation cost when $\kappa_s\gg n$.
\item When considering the communication cost, Acc-VR-EXTRA performs similarly to the optimal full batch method of APAPC. This observation matches the theory that the two methods have the same communication complexity.
\item Acc-VR-EXTRA-CA and Acc-VR-DIGing-CA do not perform well in practice, although they are theoretically optimal. Thus Acc-VR-EXTRA-CA and Acc-VR-DIGing-CA are more interesting in theory, but they are not suggested in practice.
\end{enumerate}

\section{Conclusion and Future Research}\label{sec:conclusion}
This paper extends the widely used EXTRA and DIGing methods with variance reduction, and four VR-based stochastic decentralized algorithms are proposed. The proposed VR-EXTRA has the $\bO((\kappa_s+n)\log\frac{1}{\epsilon})$ stochastic gradient computation complexity and the $\bO((\kappa_b+\kappa_c)\log\frac{1}{\epsilon})$ communication complexity. The proposed VR-DIGing has a little worse communication complexity of $\bO((\kappa_b+\kappa_c^2)\log\frac{1}{\epsilon})$. Our stochastic gradient computation complexities keep the same as the single-machine VR methods such as SVRG, and our communication complexities are the same as those of EXTRA and DIGing, respectively. The proposed accelerated VR-EXTRA and VR-DIGing achieve both the optimal $\bO((\sqrt{n\kappa_s}+n)\log\frac{1}{\epsilon})$ stochastic gradient computation complexity and $\bO(\sqrt{\kappa_b\kappa_c}\log\frac{1}{\epsilon})$ communication complexity. They are also the same as the ones of single-machine accelerated VR methods such as Katyusha, and the accelerated full batch decentralized methods such as MSDA, respectively.

DIGing, called gradient tracking in other literatures, is a fundamental decentralized algorithm in the distributed optimization community. However, its state-of-the-art complexity is $\bO((\kappa_b+\kappa_c^2)\log\frac{1}{\epsilon})$, which is worse than the $\bO((\kappa_b+\kappa_c)\log\frac{1}{\epsilon})$ one of EXTRA. An open problem is that can we improve it from $\bO((\kappa_b+\kappa_c^2)\log\frac{1}{\epsilon})$ to $\bO((\kappa_b+\kappa_c)\log\frac{1}{\epsilon})$? Recently, \citet{Koloskova-2021-GT} gave some potential directions, where the complexity is improved from $\bO(\kappa_b^2\kappa_c^2\log\frac{1}{\epsilon})$ \citep{qu2017} to $\bO(\kappa_b\kappa_c\log\frac{1}{\epsilon})$. That is, the dependence on $\kappa_c^2$ has been reduced to $\kappa_c$. Currently, it is unclear whether the complexity can be further improved to $\bO((\kappa_b+\kappa_c)\log\frac{1}{\epsilon})$.

In our Algorithms \ref{extra} and \ref{accextra}, we set the parameters dependent on $\kappa_c$, $L_f$, and $\overline L_f$, where $\kappa_c$ needs the global knowledge of the network, while $L_f$ and $\overline L_f$ needs to know the parameters of the other nodes. It is important to design practical algorithms only dependent on the local parameters, such as the Lipschitz constant $L_{(i)}$ and strong-convexity constant $\mu_{(i)}$, while still keep the optimal complexities. Another open problem is whether the reformulation (\ref{problem_consd}) can be extended to directed graphs, where $\U$ and $\V$ cannot simply take the ones in this paper. Other interesting extensions include compression, asynchrony, and so on.

\bibliography{arxiv}
\bibliographystyle{icml2020}

\end{document}